\documentclass[10pt,twoside,a4paper,english,leqno]{article}

\title{Decoupled and unidirectional asymptotic models for the propagation of internal waves}
\newcommand{\shorttitle}{Decoupled and unidirectional asymptotic models for the propagation of internal waves}
\author{Vincent Duch\^ene\thanks{Department of Applied Physics and Applied Mathematics, Columbia University, New York, U.S.A.}}
\date{\today}

\usepackage{babel} 
\usepackage[utf8]{inputenc} 
\usepackage[T1]{fontenc} 
\usepackage[ec]{aeguill}
\usepackage{amsmath, amsthm} 
\usepackage{amsfonts,amssymb}
\usepackage{float} 
\usepackage[pdftex]{graphicx} 
\usepackage{epstopdf}
\usepackage{tabularx}
\usepackage{fancyhdr}
\usepackage{url} 
\usepackage{hyperref} 
\usepackage{verbatim} 
\usepackage[hang,centerlast]{subfigure}
\usepackage{enumerate} 

\numberwithin{equation}{section}

\makeatletter
\let\Title\@title
\let\Author\@author
\makeatother
 
\fancypagestyle{mystyle}{%
	\fancyfoot{}
	\fancyhead[CE]{\shorttitle} 
	\fancyhead[CO]{\scriptsize Vincent Duch\^ene} 
	\fancyhead[RE,LO]{} 
	\fancyhead[LE,RO]{\scriptsize\thepage}
	}
\fancypagestyle{plain}{\fancyhead{} } 

\newcommand{\RR}{\mathbb{R}}

\newcommand{\NN}{\mathbb{N}}

\renewcommand{\t}{\tilde}
\renewcommand{\b}{\bar}
\renewcommand{\S}{\mathcal{S}}
\renewcommand{\P}{\mathcal{P}}

\newcommand{\N}{\mathcal{N}}
\newcommand{\A}{\mathcal{A}}
\newcommand{\T}{\mathcal{T}}

\newcommand{\Q}{\mathcal{Q}}

\newcommand{\R}{\mathcal{R}}

\renewcommand{\O}{\mathcal{O}}

\newcommand{\eps}{\boldsymbol\varepsilon}

\newcommand{\nn}{\nonumber}
\newcommand{\p}{\mathfrak{p}}

\newcommand{\id}[1]{\left\vert_{ #1}\right.}

\newtheorem{Theorem}{Theorem}[section]
\newtheorem{Definition}[Theorem]{Definition}
\newtheorem{Proposition}[Theorem]{Proposition}
\newtheorem{Lemma}[Theorem]{Lemma}
\newtheorem{Corollary}[Theorem]{Corollary}
\newtheorem{Remark}[Theorem]{Remark}

\newtheorem{Hypothesis}[Theorem]{Hypothesis}

\begin{document}

\maketitle
\begin{abstract}
We study the relevance of various scalar equations, such as inviscid Burgers', Korteweg-de Vries (KdV), extended KdV, and higher order equations (of Camassa-Holm type), as asymptotic models for the propagation of internal waves in a two-fluid system. These scalar evolution equations may be justified with two approaches. The first method consists in approximating the flow with two decoupled, counterpropagating waves, each one satisfying such an equation. One also recovers homologous equations when focusing on a given direction of propagation, and seeking unidirectional approximate solutions. This second justification is more restrictive as for the admissible initial data, but yields greater accuracy. Additionally, we present several new coupled asymptotic models: a Green-Naghdi type model, its simplified version in the so-called Camassa-Holm regime, and a weakly decoupled model. All of the models are rigorously justified in the sense of consistency.
\end{abstract}

\section{Introduction}
\subsection{Motivation}
In this paper, we study asymptotic models for the propagation of internal waves in a two-fluid system, which consists in two layers of immiscible, homogeneous, ideal, incompressible and irrotational fluids under the only influence of gravity. We assume that there is no topography (the bottom is flat) and that the surface is confined by a flat rigid lid, although the case of a free surface could be handled with our method. The interface between the two layers of fluids is given as the graph of a function, $\zeta(t,x)$, which expresses the deviation from its rest position $\{(x,z),z=0\}$ at the spatial coordinate $x$ and at time $t$. As we will focus on scalar models, we restrict ourselves to the case of horizontal dimension one.
The governing equations of such a system, describing the evolution of the deformation of the interface, $\zeta(t,x)$, as well as the velocity flow inside the two layers of fluid is recalled in Figure~\ref{fig:SketchOfDomain} (see Section~\ref{sec:models} for more details).
\begin{figure}[htb]
\centering
 \includegraphics[width=0.8\textwidth]{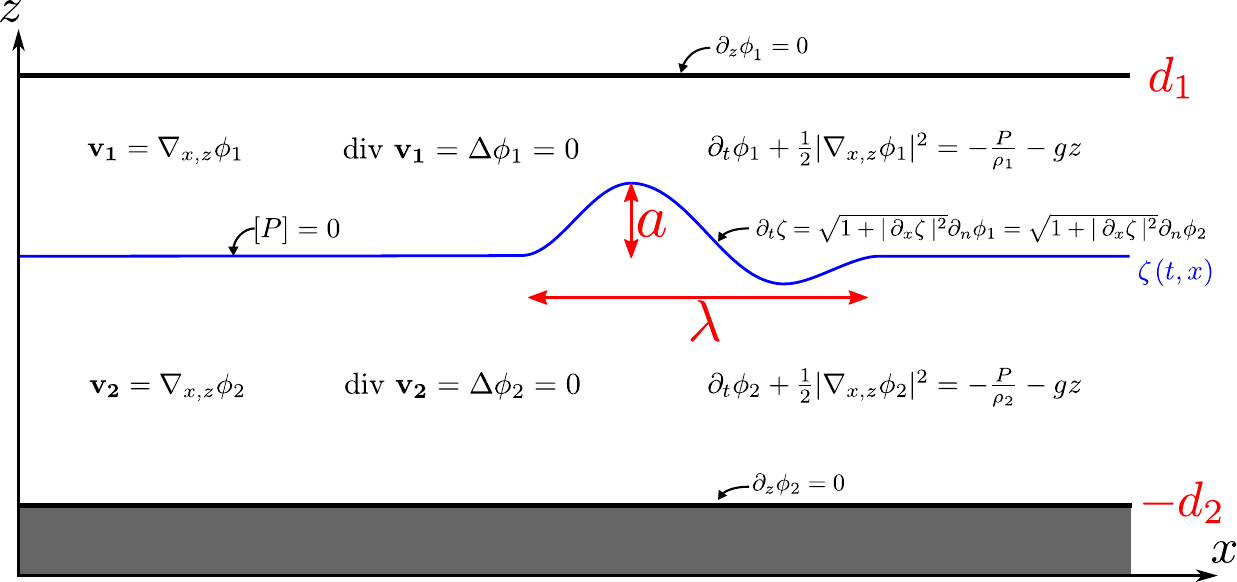}
 \caption{Sketch of the domain, and full Euler system}
\label{fig:SketchOfDomain}
\end{figure}

The mathematical theory for this system of equation (the so-called full Euler system) is extremely challenging; see~\cite{Lannes10} in particular. which has led to extensive works concerning simplified, asymptotic models, aiming at capturing the main characteristics of the flow with much simpler equations, provided the size of given parameters. Relevant dimensionless parameters that may be small in oceanographic situations include\footnote{Of course, these two dimensionless parameters describe only part of the many different regimes that may be considered; see~\eqref{eqn:defRegimeFNL} for the precise regime we consider in this work. Other relevant regimes are quickly discussed thereafter and we refer to the survey article~\cite{HelfrichMelville06} and references therein for a good overview of the oceanographic significance of asymptotic models in our regime.}
\[ \epsilon \ = \ \frac{a}{d_1} \ \quad ; \quad \mu \ = \ \frac{d_1^2}{\lambda^2} \ ;\]
where $\epsilon$ measures the amplitude of the deformation at the interface with respect to the depth of the layers of fluids, and $\mu$ measures the shallowness of the two layers of fluid, when compared with the typical wavelength of the deformation. Among other works, we would like to highlight~\cite{BonaLannesSaut08}, where many asymptotic models are presented and rigorously justified, in a wide range of regimes. In each case, the resulting model consists in two relatively simple evolution equations coupling the shear velocity and the deformation at the interface.

However, in this work, we are mainly interested in scalar models, which can be used in particular (but not only as we shall see) to describe the unidirectional propagation of gravity waves. The derivation and study of such models have a very rich and ancient history, starting with the work of Boussinesq~\cite{Boussinesq71} and Korteweg-de Vries~\cite{KortewegDe95} which introduced the famous {\em Korteweg-de Vries equation}
\begin{equation}\tag{KdV}\label{KdV}
\partial_t u \ + \ c\partial_x u \ + \ \alpha u\partial_x u \ + \ \nu\partial_x^3 u \ = \ 0,
\end{equation}
in the context of propagation of surface gravity waves, above one layer of homogeneous fluid (which we refer as ``water-wave problem'').
However, let us note that the complete rigorous justification of such model is recent:~\cite{KanoNishida86,SchneiderWayne00,BonaColinLannes05}. Such a justification is to be understood in the following sense: in the {\em long wave regime}
\[\epsilon \ \sim \ \mu \ \ll \ 1,\]
the flow can be accurately approximated as a decomposition into two counterpropagating waves, each component satisfying a KdV equation (or more generally, with the same order of accuracy, a {\em Benjamin-Bona-Mahony equation}~\cite{BenjaminBonaMahony72,Peregrine66}). 
\begin{equation}\tag{BBM}\label{BBM}
\partial_t u \ + \ c\partial_x u \ + \ \alpha u\partial_x u \ + \ \nu_x\partial_x^3 u \ - \nu_t\partial_x^2\partial_t u \ = \ 0,
\end{equation}
The situation is similar in the case of two layers of immiscible fluids, with a rigid lid (as well as in the case of two layers with a free surface, except the flow is then decomposed into four propagating waves), and corresponding results are presented in~\cite{Duchene11a}.

Of course the coefficients ($c,\alpha,\nu$, {\em etc.}) depend on the situation, and a striking difference between the water-wave case and the case of internal waves is that in the latter case, there exists a critical ratio for the two layers of fluid (depending on the ratio of the mass densities) for which the nonlinearity coefficient $\alpha$ vanishes. In that case, heuristic arguments support the inclusion of the next order (cubic) nonlinearity, which yields the {\em modified KdV equation} (see~\cite{OstrovskyStepanyants89} and references therein):
\begin{equation}\tag{mKdV}\label{mKdV}
\partial_t u \ + \ c\partial_x u \ + \ \alpha_2 u^2\partial_x u \ + \ \nu_x\partial_x^3 u \ - \nu_t\partial_x^2\partial_t u \ = \ 0.
\end{equation}
For the cubic nonlinearity to be of the same order of magnitude as the dispersive terms, this urges to consider the so-called (refering to~\cite{ConstantinLannes09} for this non-standard denomination) {\em Camassa-Holm regime}:
\[ \epsilon^2 \ \sim \ \mu \ \ll \ 1.\]
Note that this regime is interesting by itself as it allows larger amplitude waves than the long wave regime, and in particular yields models developing finite time breaking wave singularities~\cite{Constantin01,ConstantinLannes09}. Thus we ask:
{\em Can we extend to the Camassa-Holm regime the rigorous justification of the small amplitude (long wave) one?}
\medskip

 As we shall see, the answer is not straightforward, as several evolution equations are in competition, with an accuracy depending on the situation, and in particular the criticality of the depth ratio, and localization in space of the initial data.
 In addition to the equations already presented above, we consider the {\em extended KdV} (or {\em Gardner}) equation
\begin{equation}\tag{eKdV}\label{eKdV}
\partial_t u \ + \ c\partial_x u \ + \ \alpha_1 u\partial_x u \ + \ \alpha_2 u^2\partial_x u\ + \ \nu_x\partial_x^3 u \ - \nu_t\partial_x^2\partial_t u \ = \ 0,
\end{equation}
as presented in~\cite{KakutaniYamasaki78,DjordjevicRedekopp78,Miles81}, studied in~\cite{Grimshaw04,OstrovskyStepanyants05,SakaiRedekopp07} (among other works), and tested against experiments in~\cite{KoopButler81,HelfrichMelvilleMiles84,MichalletBarth'elemy98}. More generally, we
work with the higher order model 
\begin{align}
\partial_t u \ + \ c\partial_x u \ + \ \alpha_1 u\partial_x u \ + \ \alpha_2 u^2\partial_x u \ + \ \alpha_3 u^3\partial_x u \ + \nu_x\partial_x^3 u \ - \nu_t\partial_x^2\partial_t u & \nn\\
 + \ \partial_x\big(\kappa_1 u\partial_x^2u \ + \ \kappa_2(\partial_x u)^2\big) & \ = \ 0.
\tag{CL}\label{CL} \end{align}
Such a model has been presented and justified (in the water-wave setting, and in the Camassa-Holm regime) in~\cite{ConstantinLannes09}; we therefore designate it as {\em Constantin-Lannes} equation. 
This equation is related, though slightly different from the {\em Camassa-Holm} equation obtained for example in~\cite{CamassaHolm93,Johnson02} or higher order models (see various such models formally derived in~\cite{Whitham,Fokas95,FokasLiu96,ChoiCamassa99,
GrimshawPelinovskyPoloukhina02,OstrovskyGrue03,DullinGottwaldHolm03,CraigGuyenneKalisch05}, concerning the water wave or internal wave problem). However, each of these works is limited to the restrictive assumption that only one direction of propagation is non-zero.
\medskip

We propose to study in detail the justification of the decomposition of the flow (although the unidirectional case is also treated) as presented above, and in the case of internal waves with a rigid lid. Our results are expressed in a very general setting: assuming only
\[ \mu \ \ll\ 1 \quad \text{ and } \quad \epsilon \ \ll \ 1,\]
and expressing the accuracy of the different models as a function of these parameters. Although this complicates the expression of our results, and renders the proofs fairly technical, such choice allows to cover general regimes, including the long-wave as well as the Camassa-Holm regimes, described above. We recover in particular the relevance of the KdV approximation in the long wave regime. In the Camassa-Holm regime, the conclusion not as definite, but depends on the criticality of the depth ratio, the localization in space of the initial data, as well as the time-scale which is considered.

\subsection{Main results and outline of the paper}
As motivated above, the main of this article is to construct and rigorously justify asymptotic models for the system of equations, presented in figure~\ref{fig:SketchOfDomain}, and which describe the behavior of two layers of immiscible, homogeneous, ideal, incompressible fluids under the sole influence of gravity. The derivation of such equations is not new; we briefly recall it in Section~\ref{ssec:FullEulerSystem}, and refer to~\cite{BonaLannesSaut08} for more details.

The so-called {\em full Euler system} can be written as two evolution equations using Zakharov's canonical variables, namely $\zeta$ the deformation of the interface from its rest position, and the trace of a velocity potential at the interface. Such a formulation is built upon the so-called {\em Dirichlet-to-Neumann} operators, solving Laplace's equation on the two domains of fluid, with suitable Neumann or Dirichlet boundary conditions.

Lastly, we non-dimensionalize the system in order to put forward the relevant dimensionless parameters of the system, and in particular $\epsilon$, the nonlinearity parameter, and $\mu$, the shallowness parameter. We restrict our study to the following set of parameters:
\begin{equation} \label{eqn:defRegimeFNL}
\P \ \equiv \ \big\{ (\mu,\epsilon,\delta,\gamma),\ \ 0 \leq \ \mu \ \leq \ \mu_{\text{max}}, \quad 0 \ \leq\ \epsilon \ \leq \ \epsilon_{\text{max}} , \quad \delta \in (\delta_{\text{min}},\delta_{\text{max}}), \quad 0\ \leq \ \gamma\ <\ 1 \ \big\},\end{equation}
where $\mu_{\text{max}},\epsilon_{\text{max}},\delta_{\text{min}},\delta_{\text{max}}$ are positive, finite. $\delta\equiv\frac{d_1}{d_2}$ is the depth ratio, and $\gamma\equiv\frac{\rho_1}{\rho_2}$ is the mass density ratio.

As expressed above, our asymptotic models rely on the assumption that
$\mu \ \ll\ 1 $ and $\epsilon \ \ll \ 1$. More precisely, our results hold for any $(\mu,\epsilon,\delta,\gamma)\in\P$, with $\mu_{\text{max}},\epsilon_{\text{max}}$ not necessarily small, but our approximate solutions are accurate if and only if both $\epsilon$ and $\mu$ are small. 

Of course, the choice of our regime inevitably restricts the scope of our results, and the validity of our approximate models. In particular, even if we attach a great attention to impose as weak a constraint as possible on the nonlinearity and shallowness parameters in~\eqref{eqn:defRegimeFNL}, we forbid the depth ratio to approach zero or infinity. In particular, this is incompatible with the intermediate long wave regime (ILW) or Benjamin-Ono regime (BO). It would be of great interest to extend our results to such regimes, since the equivalent model to our Green-Naghdi coupled model has been derived in~\cite{Xu12}, together with the proof of energy estimates allowing to handle the full justification process as described below. Let us also recall that one can find a very large family of coupled models in~\cite{BonaLannesSaut08}, associated with a wide range of regimes, and all justified in the sense of consistency.

Finally, we would like to emphasize that our results are especially adapted to parameters restricted in the so-called Camassa-Holm regime:
\begin{equation} \label{eqn:defRegimeCH}
 \P_{CH} \ \equiv \ \big\{ (\mu,\epsilon,\delta,\gamma),\ 0 \leq \ \mu \ \leq \ \mu_{\text{max}}, \quad 0 \ \leq \ \epsilon \ \leq \ M\sqrt\mu, \quad \delta \in (\delta_{\text{min}},\delta_{\text{max}}), \quad 0< \gamma < 1 \ \big\}, 
\end{equation}
with $M$ positive, finite.

\paragraph{The Green-Naghdi coupled model.}
The first step of our analysis is to construct a shallow-water ($\mu\ll1$) high-order, coupled asymptotic model: the so-called {\em Green-Naghdi model}.

The key ingredient is to obtain an expansion of the aforementioned Dirichlet-to-Neumann operators, with respect to the shallowness parameter, $\mu$. Such an expansion is provided in the (one layer) water-wave case in~\cite[Proposition 3.8]{Alvarez-SamaniegoLannes08}, and a first-order expansion in the bi-fluidic case is obtained in~\cite[Section 2]{BonaLannesSaut08} (let us also note that the case of two layers of immiscible fluids with a free surface is given in~\cite[Section 2.2]{Duchene10}). We offer a second order expansion of the Dirichlet-to-Neumann operators, precisely disclosed in Proposition~\ref{prop:expGH}.

When replacing the Dirichlet-to-Neumann operators by their truncated expansion, and after straightforward computations, one is able to deduce
 the Green-Naghdi model, that we disclose below. Our system has two unknowns: $\zeta$ representing the deformation of the interface, and $\b v$ the shear layer-mean velocity, as defined by $\bar v \ \equiv \ \overline{u}_{2}\ -\ \gamma \overline{u}_{1}$, 
where $\overline{u}_{1}, \ \overline{u}_{2}$ are the horizontal velocities integrated across the vertical layer in each fluid:
\[
\overline{u}_{1}(t,x) \ = \ \frac{1}{h_1(t,x)}\int_{\epsilon\zeta(t,x)}^{1} \partial_x \phi_1(t,x,z) \ dz, \quad \text{ and }\quad \overline{u}_{2}(t,x) \ = \ \frac{1}{h_2(t,x)}\int_{-\frac1\delta}^{\epsilon\zeta(t,x)} \partial_x \phi_2(t,x,z) \ dz.
\]
 Precisely, our Green-Naghdi model is the following.
\begin{equation}\label{eqn:GreenNaghdiMeanI}
\left\{ \begin{array}{l}
\displaystyle\partial_{ t}{\zeta} \ + \ \partial_x \Big(\frac{h_1h_2}{h_1+\gamma h_2}\bar v\Big)\ =\ 0, \\ \\
\displaystyle\partial_{ t}\Big( \bar v \ + \ \mu\overline{\Q}[h_1,h_2]\bar v \Big) \ + \ (\gamma+\delta)\partial_x{\zeta} \ + \ \frac{\epsilon}{2} \partial_x\Big(\dfrac{{h_1}^2 -\gamma {h_2}^2 }{(h_1+\gamma h_2)^2} |\bar v|^2\Big) \ = \ \mu\epsilon\partial_x\big(\overline{\R}[h_1,h_2]\bar v \big) ,
\end{array} 
\right. 
\end{equation}
where $h_1=1-\epsilon\zeta$ and $h_2=\frac1\delta+\epsilon\zeta$ are the depth of, respectively, the upper and the lower layer, and where we define the following operators:
 \begin{align*} 
 \overline{\Q}[h_1,h_2]V \ &\equiv \ \frac{-1}{3h_1 h_2}\Bigg(h_1 \partial_x \Big({h_2}^3\partial_x\big(\frac{h_1\ V}{h_1+\gamma h_2} \big)\Big)\ +\ \gamma h_2\partial_x \Big( {h_1}^3\partial_x \big(\frac{h_2\ V}{h_1+\gamma h_2}\big)\Big)\Bigg), \\
 \overline{\R}[h_1,h_2]V \ &\equiv \ \frac12 \Bigg( \Big( h_2\partial_x \big( \frac{h_1\ V}{h_1+\gamma h_2} \big)\Big)^2\ -\ \gamma\Big(h_1\partial_x \big(\frac{h_2\ V}{h_1+\gamma h_2} \big)\Big)^2\Bigg)\\
 &\qquad + \frac13\frac{V}{h_1+\gamma h_2}\ \Bigg( \frac{h_1}{h_2}\partial_x\Big( {h_2}^3\partial_x\big(\frac{h_1\ V}{h_1+\gamma h_2} \big)\Big) \ - \ \gamma\frac{h_2}{h_1}\partial_x\Big({h_1}^3\partial_x \big(\frac{h_2\ V}{h_1+\gamma h_2}\big)\Big) \Bigg).
 \end{align*}
 This system is related to the one introduced by Choi and Camassa in~\cite{ChoiCamassa99}, and is the two-layers counterpart of the water-wave Green-Naghdi model introduced in~\cite{GreenNaghdi76}, and fully justified in~\cite{Alvarez-SamaniegoLannes08,Li06,Israwi11,Lannes}. We present here the first rigorous justification of the two-layer Green-Naghdi model, in the sense of consistency.
 \begin{Proposition}\label{prop:ConsGreenNaghdiMeanI}
Let $U^\p\equiv(\zeta^\p,\psi^\p)_{\p\in\P}$ be a family of solutions of the full Euler system~\eqref{eqn:EulerCompletAdim}, such that $\zeta^\p\in W^{1}([0,T);H^{s+11/2}),\partial_x\psi^\p\in W^{1}([0,T);H^{s+13/2})$ with $s\geq s_0+1/2$, $s_0>1/2$, and uniformly with respect to $\p\in\P$; see~\eqref{eqn:defRegimeFNL}. Moreover assume that there exists $h>0$ such that
\[
h_1 \ \equiv\ 1-\epsilon\zeta^\p \geq h>0, \quad h_2 \ \equiv \ \frac1\delta +\epsilon \zeta^\p\geq h>0.
\]
Define $\bar v^\p \ \equiv \ \overline{u}_{2}\ -\ \gamma \overline{u}_{1}$ as above.
Then $(\zeta^\p,\bar v^\p)$ satisfies~\eqref{eqn:GreenNaghdiMeanI}, up to a remainder $R$, bounded by
\[ \big\Vert R \big\Vert_{L^\infty([0,T);H^s)} \ \leq \ \mu^2\ C,\]
with $C=C(\frac1{s_0-1/2},\frac1{h},\epsilon_{\text{max}},\mu_{\text{max}},\frac1{\delta_{\text{min}}},\delta_{\text{max}},\big\Vert \zeta^\p \big\Vert_{W^{1}([0,T);H^{s+11/2})},\big\Vert \partial_x\psi^\p \big\Vert_{W^{1}([0,T);H^{s+13/2})} )$.
\end{Proposition} 
\bigskip

 Here, and in the following, we denote by $C(\lambda_1,\lambda_2,\ldots)$ any positive constant, depending on the parameters ${\lambda_1,\lambda_2,\ldots }$, and whose dependence on $\lambda_j$ is assumed to be nondecreasing. $C_0$ is a positive constant, which do not depend on all the other parameters involved, and we write $A=\O( B)$ if $A\leq C_0 B$, and $A\approx B$ if $A=\O( B)$ and $B=\O( A)$.
 
 For $0 < T \leq \infty$ and $f(t,x)$ a function defined on $[0,T]\times\RR$, we write $f\in L^\infty([0,T];H^s)$ if $f$ is uniformly (with respect to $t\in [0,T]$) bounded in $H^s=H^s(\RR)$ the $L^2$-based Sobolev space. Its norm is denoted $\big\Vert\cdot\big\Vert_{L^\infty([0,T); H^s)}$, as the Sobolev norms are denoted with simple bars: $\big\vert\cdot\big\vert_{H^s}$. Finally, $W^{1}([0,T);H^{s+1})$ is the space of functions $f(t,x)\in L^\infty([0,T];H^{s+1})$ such that $\partial_t f\in L^\infty([0,T];H^s)$, endowed with the norm
$\big\Vert f \big\Vert_{W^{1}([0,T);H^{s+1})} \equiv \big\Vert f \big\Vert_{L^\infty([0,T);H^{s+1})} + \big\Vert \partial_t f \big\Vert_{L^\infty([0,T);H^{s})} $.
 \medskip
 
 \paragraph{Consistency and full justification.}
Consistency results, such as Proposition~\ref{prop:ConsGreenNaghdiMeanI}, are part of the procedure that leads to a full justification of asymptotic models, as it has been achieved in the water wave case in~\cite{Alvarez-SamaniegoLannes08}. A model is said to be {\em fully justified} (following the terminology of~\cite{Lannes}) if the Cauchy problem of both the full Euler system and the asymptotic model is well-posed for a given class of initial data, and over the relevant time scale; and if the solutions with corresponding initial data remain close.
As described in~\cite[Section~6.3]{Lannes10}, the full justification of system~\eqref{eqn:GreenNaghdiMeanI} follows from:
\begin{itemize}
\item (Consistency) One proves that families of solutions of the full Euler system, existing and controlled over the relevant time scale (that is $\O(1/\epsilon)$ here), solves the Green-Naghdi model~\eqref{eqn:GreenNaghdiMeanI} up to a small residual. This is Proposition~\ref{prop:ConsGreenNaghdiMeanI}.
\item (Existence) One proves that families of solutions to the full Euler system as above do exist. This difficult step is ensured by Theorem 5.8 in~\cite{Lannes10}, provided that a small surface tension is added, and that an additional stability criterion is satisfied (see details therein).
\item (Convergence) One proves that the solutions of the full Euler system, and the ones of the Green-Naghdi model~\eqref{eqn:GreenNaghdiMeanI} with corresponding initial data remain close, over the relevant time scale. 
\end{itemize}
 The last step requires the well-posedness of the Cauchy problem for the Green-Naghdi model~\eqref{eqn:GreenNaghdiMeanI}, as well as the stability of its solutions with respect to perturbations. More precisely, we require that functions satisfying the Green-Naghdi model up to a small residual remain close to the exact solution with corresponding initial data, so that the first two steps of our procedure (consistency and existence) yield the conclusion (convergence), and therefore the full justification of the Green-Naghdi model.
See Theorem~6.1 in~\cite{Lannes10} for the application of such procedure for the full justification of the so-called {\em shallow-water/shallow-water} asymptotic model, which corresponds to~\eqref{eqn:GreenNaghdiMeanI}, when withdrawing $\O(\mu)$ terms.
\medskip

The different models throughout this work are justified through consistency results, with respect to the Green-Naghdi model~\eqref{eqn:GreenNaghdiMeanI}. Let us define precisely below what we designate by consistency in the core of this paper.
 \begin{Definition}[Consistency]\label{def:consistency}
Let $(\zeta^\p,v^\p)_{\p\in\P}$ be a family of pair of functions, uniformly bounded in $L^\infty([0,T);H^{s+\b s})$ ($\b s\geq0$ to be determined), depending on parameters $\p\in \P$; see~\eqref{eqn:defRegimeFNL}.

We say that $(\zeta^\p,v^\p)$ is {\em consistent} with Green-Naghdi system~\eqref{eqn:GreenNaghdiMeanI} at precision $\O(\eps^\p)$, of order $s$ and on $[0,T)$, if $(\zeta^\p,v^\p)$ satisfies, for $\eps^\p$ sufficiently small,
 \[
\left\{ \begin{array}{l}
\displaystyle\partial_{ t}{\zeta^\p} \ + \ \partial_x \Big(\frac{h_1h_2}{h_1+\gamma h_2} v^\p\Big)\ =\ \eps^\p\ r_1, \\ \\
\displaystyle\partial_{ t}\Big( v^\p + \mu\overline{\Q}[h_1,h_2] v^\p \Big) + (\gamma+\delta)\partial_x{\zeta^\p} + \frac{\epsilon}{2} \partial_x\Big(\dfrac{{h_1}^2 -\gamma {h_2}^2 }{(h_1+\gamma h_2)^2} | v^\p|^2\Big) - \mu\epsilon\partial_x\big(\overline{\R}[h_1,h_2]\bar v^\p \big) = \eps^\p\ r_2 ,
\end{array} 
\right. \]
where $h_1=1-\epsilon\zeta^\p$ and $h_2=\frac1\delta+\epsilon\zeta^\p$, and with
\[\big\Vert (r_1, r_2)\big\Vert_{L^\infty([0,T);H^s)^2}\ \leq \ C\Big(\mu_{\text{max}},\epsilon_{\text{max}},1/\delta_{\text{min}},\delta_{\text{max}},\big\Vert\zeta^\p\big\Vert_{L^\infty([0,T);H^{s+\b s})},\big\Vert v^\p\big\Vert_{L^\infty([0,T);H^{s+\b s})}\Big).\]
\end{Definition}

Of course, one can apply the procedure described above to any consistent approximate solutions of the Green-Naghdi system. Consequently, our approximate solutions, described in Propositions~\ref{prop:unidirI},~\ref{prop:decompositionI} and~\ref{prop:other}, are fully justified as approximate solutions of the Green-Naghdi model, {\em and therefore as approximate solutions of the full Euler system~\eqref{eqn:EulerCompletAdim}}, provided the Green-Naghdi system~\eqref{eqn:GreenNaghdiMeanI}, {\em or any equivalently consistent model}, enjoys the following property.
\begin{Hypothesis}[Well-posedness and stability]\label{conj:stab}
Let $(\epsilon,\mu,\delta,\gamma)=\p\in\P$, as defined in~\eqref{eqn:defRegimeFNL},
and $U_0^\p\in H^{s+\bar s}$, with $s$ and $\bar s$ sufficiently big. Then the following holds.
\begin{enumerate}
\item There exists $T>0$ and unique strong solutions of the Green-Naghdi system~\eqref{eqn:GreenNaghdiMeanI}, $U_{\text{GN}}^\p$, such that $U_{\text{GN}}^\p\id{t=0}=U_0^\p$ and $U_{\text{GN}}^\p$ is uniformly bounded on $W^1([0,T);H^s)$.
\item 
Let $U^\p$, satisfying $U^\p\id{t=0}={U_0^\p}\id{t=0}$, be consistent (in the sense of Definition~\ref{def:consistency}) with Green-Naghdi system~\eqref{eqn:GreenNaghdiMeanI} on $[0,T]$, of order $s$, at precision $\O(\eps)$. Then the difference between the two families of functions is estimated as
\[ \big\Vert U^\p-U_{\text{GN}}^\p \big\Vert_{L^\infty([0,t];H^s)} \ \leq \ C\ \eps\ t,\]
with $C=C\big(\big\Vert U_{\text{GN}}^\p\big\Vert_{L^\infty([0,T];H^s)} ,\big\Vert U^\p\big\Vert_{L^\infty([0,T];H^{s+\b s})},\frac{1}{\delta_{\text{min}}},\delta_{\text{max}},\epsilon_{\text{max}},\mu_{\text{max}} \big)$.
\end{enumerate}
\end{Hypothesis}
Both of the properties in Hypothesis~\ref{conj:stab} typically follow from energy estimates on the Green-Naghdi system. Such a result has been obtained for a well-chosen Green-Naghdi system in the case of the water-wave problem in~\cite{Li06,Alvarez-SamaniegoLannes08a,Israwi11}. In the two-layer setting, the result has been proved for an equivalent system (the so-called symmetric Boussinesq/Boussinesq model) in the long wave regime, $\epsilon=\O(\mu)$, in~\cite{Duchene11a}.\footnote{Although the results of~\cite{Duchene11a} are dedicated to the two-layer case with a free surface, the method can easily be adapted to the simpler rigid lid situation. Let us also mention~\cite{SautXu12} for an extensive study of the very similar Boussinesq models in the water-wave case.} Such result can be extended to the Camassa-Holm regime~\eqref{eqn:defRegimeCH}, and will be the object of a future publication. 
\medskip

Let us emphasize that Hypothesis~\ref{conj:stab} is not necessary for our results to hold, but only to complete the full justification procedure described above. However, such a result is very useful to perceive the accuracy of our approximate solutions (in the sense of convergence) that one expects, especially for the decoupled approximations in Propositions~\ref{prop:decompositionI} and~\ref{prop:other}. Thus in the discussion of our results in Sections~\ref{sec:unidirectional} and~\ref{ssec:discussion}, we assume that 
 Hypothesis~\ref{conj:stab} holds. Numerical simulations are in perfect agreement with the resulting estimates.

\paragraph{The scalar models.}
Let us describe the different scalar evolution equations which are considered in this paper. The higher-order model we study is the {\em Constantin-Lannes} equation (CL):
\begin{multline} \label{eq:CLI}
 (1- \ \mu\beta \partial_x^2)\partial_t v \ +\ \epsilon \alpha_1 v\partial_x v \ + \ \epsilon^2 \alpha_2 v^2\partial_x v\ + \ \epsilon^3 \alpha_3 v^3\partial_x v \\
 +\ \mu\nu \partial_x^3 v \ \ + \ \mu\epsilon\partial_x\big(\kappa_1 v\partial_x^2 v+\kappa_2(\partial_x v)^2\big) \ = \ 0, 
\end{multline}
where $\beta,\alpha_i$ ($i=1,2,3$), $\nu$, $\kappa_1$, $\kappa_2$, are fixed parameters . 
\medskip

Note that the Constantin-Lannes equation can be seen as a generalization of classical lower order models (obtained when neglecting higher-order terms, or equivalently setting some parameters to zero). In the following, we consider
\begin{itemize}
\item the {\em inviscid Burgers’} equation (iB): 
\[
\partial_t v \ + \ \epsilon\alpha_1 v\partial_x v \ \ = \ 0; 
\]
\item the {\em Korteweg-de Vries} (or more precisely {\em Benjamin-Bona-Mahony}~\cite{BenjaminBonaMahony72}) equation (KdV):
\[
(1- \ \mu\beta \partial_x^2)\partial_t v \ + \ \epsilon\alpha_1 v\partial_x v \
+\ \mu\nu \partial_x^3 v \ = \ 0; 
\]
\item the {\em extended Korteweg-de Vries} equation (eKdV):
\[
(1- \ \mu\beta \partial_x^2)\partial_t v \ + \ \epsilon\alpha_1 v\partial_x v \ + \ \epsilon^2 \alpha_2 v^2\partial_x v\ \ 
+\ \mu\nu \partial_x^3 v \ = \ 0. 
\]
\end{itemize}
\medskip

Our first result regarding these scalar evolution equations, concerns the long time well-posedness of the Cauchy problem, and the persistence in time of the localization in space of initial data. We treat simultaneously all scalar models by allowing parameters in~\eqref{eq:CLI} to vanish. However, we require that $\mu \beta>0$, as our proof relies heavily on {\em a priori} estimates of the solution in the following scaled Sobolev norm:
\[ \big\vert u \big\vert_{H^{s+1}_{\mu\beta}}^2 \ \equiv \ \big\vert u \big\vert_{H^{s}}^2 \ + \ \mu\beta \big\vert u \big\vert_{H^{s+1}}^2,\quad \text{ for some } s\geq0.\]
We also make use of the closed functional subspaces $X^s_{n,\mu}\subset H^{s+2n}$, endowed with the following weighted Sobolev norm:
\[ \big\vert u \big\vert_{X^s_{n,\mu}} \ = \ \sum_{j=0}^n\big\vert x^j {u}\big\vert_{H^{s+2(n-j)}_\mu} .\]
\begin{Proposition}[Well-posedness and persistence]\label{prop:WPI}
Let $u^0 \in H^{s+1}$, with $s\geq s_0>3/2$. Let the parameters be such that $\beta,\mu,\epsilon>0$, and define $M>0$ such that
\[ \beta+\frac{1}{\beta}+\mu+\epsilon +|\alpha_1|+|\alpha_2|+|\alpha_3|+|\nu|+|\kappa|+|\iota| \ \leq \ M. \]
Then 
 there exists $T=C\big(\frac1{s_0-3/2},M,\big\vert u^0 \big\vert_{H^{s+1}_\mu}\big)$ and a unique $u \in C^0([0,T/\epsilon); H^{s+1}_\mu)\cap C^1([0,T/\epsilon);H^{s}_\mu)$ such that $u$ satisfies~\eqref{eq:CLI} and initial condition $u\id{t=0}=u^0$.

 Moreover, $u$ satisfies the following energy estimate for $0\leq t\leq T/\epsilon$:
\[
\big\Vert \partial_t u \big\Vert_{L^\infty([0,T/\epsilon);H^{s}_\mu)} \ + \ \big\Vert u \big\Vert_{L^\infty([0,T/\epsilon);H^{s+1}_\mu)} \ \leq \ C(\frac1{s_0-3/2},M, \big\vert u^0 \big\vert_{H^{s+1}_\mu}) .
\]

\medskip

 Assume additionally that for fixed $n, k\in\NN$, the function $x^{j} {u^0}\in H^{s+\b s}$, with $0\leq j \leq n$ and $\b s=k+1+2(n-j)$. Then 
 there exists $T=C\big(\frac1{s_0-3/2},M,n,k,\sum_{j=0}^n\big\vert x^j {u^0}\big\vert_{H^{s+k+1+2(n-j)}_\mu}\big)$
such that for $0\leq t\leq T\times\min(1/\epsilon,1/\mu)$, one has
\[ \big\Vert x^n\partial_k \partial_t {u}\big\Vert_{L^\infty([0,t);H^{s}_\mu)} + \big\Vert x^n \partial_k {u}\big\Vert_{L^\infty([0,t);H^{s+1}_\mu)} \ \leq \ C\Big( \frac1{s_0-3/2},M,n,k,\sum_{j=0}^n\big\vert x^j {u^0}\big\vert_{H^{s+k+1+2(n-j)}_\mu}\Big) .\]
In particular, 
one has, for $0\leq t\leq T\times\min(1/\epsilon,1/\mu)$,
\[
 \big\Vert \partial_t {u}\big\Vert_{L^\infty([0,t);X^{s}_{n,\mu})} + \big\Vert {u}\big\Vert_{L^\infty([0,t);X^{s+1}_{n,\mu})} \ \leq \ C\Big(\frac1{s_0-3/2}, M,n,\big\vert u^0 \big\vert_{X^{s+1}_{n,\mu}} \Big) .
\]
\end{Proposition}
\begin{Remark}
The well-posedness of the Cauchy problem for inviscid Burgers’, KdV and eKdV equations are well-known (see for example~\cite{BonaSmith75,Kato75,Kato83,GinibreTsutsumi89,CollianderKeelStaffilaniEtAl03} and references therein), and actually do not require $\mu\beta>0$. The case of Constantin-Lannes equation is provided in Proposition~4 in~\cite{ConstantinLannes09}. 

The persistence of the solution in weighted Sobolev norms for the Constantin-Lannes equation is new, as far as we know. Similar results in the case of (eventually extended) Korteweg-de Vries equations are obtained in slightly different setting (for the most part using weighted $L^2$ spaces intersected with non-weighted Sobolev spaces $H^s$, $s>0$) in~\cite{Iorio90,BonaSaut93,SchneiderWayne00,Vera02,CarvajalGamboa10,Nahas10,NahasPonce11,Brandolese11}.
\end{Remark}

As mentioned previously, we consider two different approaches concerning the justification of scalar equations such as~\eqref{eq:CLI}, to construct approximate solutions of the Green-Naghdi system~\eqref{eqn:GreenNaghdiMeanI}, and therefore as asymptotic models for the propagation of internal waves. The first justification, that we call {\em unidirectional approximation}, consists in adjusting carefully the initial data so that solutions of~\eqref{eqn:GreenNaghdiMeanI} provide a good approximation of the flow. The second justification states that any initial perturbation of the flow can be approximately decomposed into the sum of two decoupled waves, each of them satisfying a scalar equation. Controlling the precision of such {\em decoupled approximation} is of course more difficult, and the unidirectional approximation gains in precision what it lacks in generality.

These two approaches have been successfully developed in the (one layer) water wave situation (see~\cite{Lannes} and references therein), and extended to the bifluidic case in the long wave regime in~\cite{Duchene11a}. Our results may therefore be considered as a continuation of these earlier works.

 \paragraph{Unidirectional propagation.} This result follows the strategy developed for the water-wave problem in~\cite{Johnson02,ConstantinLannes09}. We prove that if one chooses carefully the initial perturbation (deformation of the interface as well as shear layer-mean velocity), then the flow can be approximated as a solution of the Constantin-Lannes equation~\eqref{eq:CLI}, with a great accuracy.
\begin{Proposition} \label{prop:unidirI}
Set $\lambda,\theta\in\RR$, and $\zeta^0 \in H^{s+5}$ with $s\geq s_0>3/2$. For $(\epsilon,\mu,\delta,\gamma)=\p\in\P$, as defined in~\eqref{eqn:defRegimeFNL}, denote
$(\zeta^{\p})_{\p \in\P }$ the unique solution of the equation
\begin{align*}
\partial_t \zeta + \partial_x \zeta + \epsilon\alpha_1\zeta\partial_x \zeta + \epsilon^2 \alpha_2 \zeta^2\partial_x \zeta +\epsilon^3 \alpha_3 \zeta^3\partial_x \zeta+\mu\nu_x^{\theta,\lambda}\partial_x^3\zeta-\mu\nu_t^{\theta,\lambda}\partial_x^2\partial_t\zeta & \\
+\mu\epsilon\partial_x\left(\kappa_1^{\theta,\lambda} \zeta\partial_x^2\zeta +\kappa_2^{\theta} (\partial_x \zeta)^2\right) & =0\ ,
\end{align*}
where parameters $\alpha_i$ ($i=1,2,3$), $\nu_x^{\theta,\lambda}$, $\nu_t^{\theta,\lambda}$, $\kappa_1^{\theta,\lambda}$, $\kappa_2^{\theta}$, are precisely enclosed in Proposition~\ref{eq:CLuni}. 
For given $M_{s+5},h>0$, assume that there exists $T_{s+5}>0$ such that
\[ T_{s+5} \ = \ \max\big( \ T\geq0\quad \text{such that}\quad \big\Vert \zeta^{\p} \big\Vert_{L^\infty([0,T);H^{s+5})}\ \leq \ M_{s+5}\ \big) \ ,\]
and for any $(t,x)\in [0,T_{s+5})\times\RR$,
\[ h_1(t,x)=1-\epsilon\zeta^\p(t,x)>h>0, \quad h_2(t,x)=\frac1\delta+\epsilon\zeta^\p(t,x)>h>0.\]
Then define
$v^\p$ as $ v^\p=\frac{h_1+\gamma h_2}{h_1h_2}\underline{v}[\zeta^\p]$, with
\[
\underline{v}[\zeta] \ = \ \zeta + \epsilon\frac{\alpha_1}2\zeta^2 + \epsilon^2 \frac{\alpha_2}{3} \zeta^3 +\epsilon^3 \frac{\alpha_3}4 \zeta^4 +\mu\nu\partial_x^2\zeta+\mu\epsilon\left(\kappa_1 \zeta\partial_x^2\zeta +\kappa_2 (\partial_x \zeta)^2\right),
\]
where parameters $\alpha_1,\alpha_2,\alpha_3$ are as above, and
$ \nu=\nu^{0,0}_x$, $ \kappa_1=\kappa_1^{0,0}$, $ \kappa_2=\kappa_2^{0}$.

Then $(\zeta^\p,v^\p)$ is consistent with Green-Naghdi equations~\eqref{eqn:GreenNaghdiMeanI}, of order $s$ and on $[0,T_{s+5})$, with precision $\O(\eps)$, with
\[ \eps \ = \ C(M_{s+5},\frac1{s_0-3/2},{h}^{-1},\delta_{\text{min}}^{-1},\delta_{\text{max}},\epsilon_{\text{max}},\mu_{\text{max}},|\lambda|,|\theta|)\ \times \ \max(\epsilon^4,\mu^2)\ .\]
\end{Proposition}

 \begin{Remark}
 When $\gamma\to0$ and $\delta\to1$, one recovers the one-fluid model introduced by Constantin and Lannes in~\cite[Section~2.2]{ConstantinLannes09}, with $q=\frac{1-\theta}{6}$ and $\lambda=0$, using notations therein. In the bi-fluidic case, Choi and Camassa~\cite[Appendix A]{ChoiCamassa99} obtained a very similar result.
 \end{Remark} 
 \begin{Remark}
Our approximation consists in solving a scalar evolution equation for the deformation at the interface, followed by a reconstruction of the shear layer-mean velocity from the deformation (in particular, the initial shear velocity is determined by the initial deformation). Following~\cite{ConstantinLannes09}, a similar strategy consists in looking for an evolution equation for the shear layer-mean velocity, and reconstruct the deformation at the interface. We decide not to present the outcome of such strategy, as the result is very similar, and calculations are somewhat heavier in that case.
\end{Remark}
\begin{Remark}
As discussed in~\cite[Proposition~5]{ConstantinLannes09}, specific values of parameters in~\eqref{eq:CLuni} yield equations with different properties, especially concerning the behavior near the maximal time of definition (if it is finite). Indeed, the proof of~\cite[Proposition~5]{ConstantinLannes09} can easily be adapted to more general coefficients, and one obtains:
\begin{itemize}
\item If $\nu_t^{\theta,\lambda}>0$, $\kappa_1^{\theta,\lambda}=2\kappa_2^{\theta}>0$ and $\alpha_3>0$, then singularities can develop in finite time only in the form of {\em surging wave breaking}. In other words, if the maximal time of existence of $\zeta$ is finite, $T<\infty$, then
\[ \sup_{t\in[0,T), x\in\RR} \{|\zeta(t,x)|\}<\infty \quad \text{ and } \quad \sup_{x\in\RR}\{\partial_x \zeta(t,x)\}\uparrow\infty \ \text{ as } t\uparrow T .\]
 \item If $\nu_t^{\theta,\lambda}>0$, $\kappa_1^{\theta,\lambda}=2\kappa_2^{\theta}<0$ and $\alpha_3<0$, then singularities can develop in finite time only in the form of {\em plunging wave breaking}. In other words, if the maximal time of existence of $\zeta$ is finite, $T<\infty$, then
\[ \sup_{t\in[0,T), x\in\RR} \{|\zeta(t,x)|\}<\infty \quad \text{ and } \quad \inf_{x\in\RR}\{\partial_x \zeta(t,x)\}\downarrow -\infty \ \text{ as } t\uparrow T .\] 
\end{itemize}
Identity $\kappa_1^{\theta,\lambda}=2\kappa_2^{\theta}$ holds in the line $\theta-\lambda=1/2$, and in that case, $\nu_t^{\theta,\lambda}>0$ if and only if $\theta>1/4$. 
Restricting to $\theta\leq1$ as natural values for the use of BBM trick, one can easily check that if $\gamma=0$, then singularities may occur only as surging wave breaking, as it is the case in the one-layer situation. On the contrary, if $\gamma\sim 1$, as is the case in the stratified ocean (small variation of densities) then singularities will occur as surging wave breaking if $\delta>1$ (thicker upper layer), and plunging wave breaking will occur for $\delta<1$ (thicker lower layer).
\end{Remark}

\paragraph{The decoupled approximation.} We now turn to the case of a generic initial perturbation of the flow. When neglecting any term of size $\O(\epsilon+\mu)$ in~\eqref{eqn:GreenNaghdiMeanI}, one obtains a simple wave equation for $\zeta,\bar v$, which predicts that the flow splits into two counterpropagating waves, each one moving at velocity $c=\pm1$. Our aim is to provide a higher precision model, by allowing each of these waves to satisfy a scalar evolution equation. The strategy presented here has been used in~\cite{BonaColinLannes05}, where the authors present a similar rigorous justification of the KdV equation as asymptotic model for the (one fluid) water wave problem in the long wave regime (see~\cite{Duchene11a} for the bi-fluidic case). As discussed earlier on, a major difference of the bifluidic case is the existence of a {\em critical ratio}, for which the quadratic nonlinearity in our models vanish. This is our main motivation for the present work, which extends the previous results to more general regimes, allowing greater nonlinearities, and higher order scalar equations.
More precisely, we define the following Constantin-Lannes approximation.
\begin{Definition}[Constantin-Lannes decoupled approximation] \label{def:CL}
Let $\zeta^0,v^0$ be given scalar functions, and set parameters $(\epsilon,\mu,\gamma,\delta)\in\P$, as defined in~\eqref{eqn:defRegimeFNL}, and $(\lambda,\theta)\in\RR^2$. The Constantin-Lannes decoupled approximation is then
\[ U_{\text{CL}}\ \equiv \ \Big(v_+(t,x-t)+v_-(t,x+t),(\gamma+\delta)\big(v_+(t,x-t)-v_-(t,x+t)\big)\Big),\]
 where ${v_\pm}\id{t=0} \ = \ \frac12(\zeta^0\pm\frac{v^0}{\gamma+\delta})\id{t=0}$ and $v_\pm=(1\pm\mu\lambda \partial_x^2)^{-1}v_\pm^\lambda$ with $v_\pm^\lambda$ satisfying
\begin{multline} \label{eq:CL}
 \partial_t v_\pm^\lambda \ \pm \ \epsilon\alpha_1 v_\pm^\lambda\partial_x v_\pm^\lambda \ \pm \ \epsilon^2 \alpha_2 (v_\pm^\lambda)^2\partial_x v_\pm^\lambda\ \pm \ \epsilon^3 \alpha_3^{\theta,\lambda}(v_\pm^\lambda)^3\partial_x v_\pm \\
\pm\ \mu\nu^{\theta,\lambda}_x \partial_x^3 v_\pm^\lambda \ - \ \mu\nu^{\theta,\lambda}_t \partial_x^2\partial_t v_\pm^\lambda \ \pm \ \mu\epsilon\partial_x\big(\kappa_1^{\theta,\lambda} v_\pm^\lambda\partial_x^2 v_\pm^\lambda+\kappa_2^{\theta}(\partial_x v_\pm^\lambda)^2\big) \ = \ 0, 
\end{multline}
with parameters disclosed precisely in~\eqref{eqn:parameters}.
\end{Definition}
As mentioned above, the Constantin-Lannes equation is the higher order model we consider. It is not obvious that such precise model offers a significantly better approximation than lower-order models. As a matter of fact, this is generically not the case, as discussed in Section~\ref{ssec:discussion} (see Remark~\ref{rem:differentmodels}, below).
Thus we also consider models with formally lower order accuracy.
\begin{Definition}[lower order decoupled approximations] \label{def:other}
Let $\zeta^0,v^0$ be given scalar functions, and set parameters $(\epsilon,\mu,\gamma,\delta)\in\P$, $(\lambda,\theta)\in\RR^2$. A decoupled approximate solution of the system~\eqref{eqn:GreenNaghdiMeanI} is \[ U\ \equiv \ \Big(v_+(t,x-t)+v_-(t,x+t),(\gamma+\delta)\big(v_+(t,x-t)-v_-(t,x+t)\big)\Big),\]
 where ${v_\pm}\id{t=0} \ = \ \frac12(\zeta^0\pm\frac{v^0}{\gamma+\delta})$ and $v_\pm=(1\pm\mu\lambda \partial_x^2)^{-1}v_\pm^\lambda$ with $v_\pm^\lambda$ satisfying a scalar evolution equation. In what follows, we consider
\begin{itemize}
\item the inviscid Burgers’ equation: 
\[
\partial_t v_\pm^\lambda \ \pm \ \epsilon\alpha_1 v_\pm^\lambda\partial_x v_\pm^\lambda \ \ = \ 0; 
\]
\item the Korteweg-de Vries (or more precisely Benjamin-Bona-Mahony) equation:
\[
\partial_t v_\pm^\lambda \ \pm \ \epsilon\alpha_1 v_\pm^\lambda\partial_x v_\pm^\lambda \
\pm\ \mu\nu^{\theta,\lambda}_x \partial_x^3 v_\pm^\lambda \ - \ \mu\nu^{\theta,\lambda}_t \partial_x^2\partial_t v_\pm^\lambda \ \ = \ 0; 
\]
\item the extended Korteweg-de Vries equation:
\[
\partial_t v_\pm^\lambda \ \pm \ \epsilon\alpha_1 v_\pm^\lambda\partial_x v_\pm^\lambda \ \pm \ \epsilon^2 \alpha_2 (v_\pm^\lambda)^2\partial_x v_\pm^\lambda\ \ 
\pm\ \mu\nu^{\theta,\lambda}_x \partial_x^3 v_\pm^\lambda \ - \ \mu\nu^{\theta,\lambda}_t \partial_x^2\partial_t v_\pm^\lambda \ = \ 0; 
\]
\end{itemize}
 where the parameters are the same as in Definition~\ref{def:CL}.
\end{Definition}

Let us now state our main results, concerning the justification of the decoupled scalar approximations.
\begin{Proposition}[Consistency]\label{prop:decompositionI}
Let $\zeta^0,v^0\in H^{s+6}$, with $s\geq s_0> 3/2$. For $(\epsilon,\mu,\delta,\gamma)=\p\in\P$, as defined in~\eqref{eqn:defRegimeFNL}, we denote $U_{\text{CL}}^\p$ the unique solution of the CL approximation, as defined in Definition~\ref{def:CL}. For some given $M^\star_{s+6}>0$, sufficiently big, we assume that there exists $T^\star>0$ and a family $(U_{\text{CL}}^\p)_{\p\in\P}$ such that
\[ T^\star \ = \ \max\big( \ T\geq0\quad \text{such that}\quad \big\Vert U_{\text{CL}}^\p \big\Vert_{L^\infty([0,T);H^{s+6})}+\big\Vert \partial_t U_{\text{CL}}^\p \big\Vert_{L^\infty([0,T);H^{s+5})} \ \leq \ M^\star_{s+6}\ \big) \ .\]

Then there exists $U^c=U^c[U_{\text{CL}}^\p]$ such that $U\equiv U_{\text{CL}}^\p+U^c$ is consistent with Green-Naghdi equations~\eqref{eqn:GreenNaghdiMeanI} of order $s$ on $[0,t]$ for $t<T^\star$, at precision $\O(\eps^\star_{\text{CL}})$ with
\[ \eps^\star_{\text{CL}} \ = \ C\ \max(\epsilon^2(\delta^2-\gamma)^2,\epsilon^4,\mu^2)\ (1+\sqrt t) ,\]
with $C=C(\frac1{s_0-3/2},M^\star_{s+6},\frac1{\delta_{\text{min}}},\delta_{\text{max}},\epsilon_{\text{max}},\mu_{\text{max}},|\lambda|,|\theta|)$, and the corrector term $U^c$ is estimated as
\[ \big\Vert U^c \big\Vert_{L^\infty([0,T^\star];H^{s})}+\big\Vert \partial_t U^c \big\Vert_{L^\infty([0,T^\star];H^{s})} \leq C\ \max(\epsilon(\delta^2-\gamma),\epsilon^2,\mu) \min( t, \sqrt t) .\]
\medskip

Additionally, if there exists $\alpha>1/2$, $M^\sharp_{s+6},\ T^\sharp>0$ such that 
\[ \sum_{k=0}^6\big\Vert (1+x^2)^\alpha \partial_x^k U_{\text{CL}}^\p \big\Vert_{L^\infty([0,T);H^{s})} +\sum_{k=0}^5\big\Vert (1+x^2)^\alpha \partial_x^k\partial_t U_{\text{CL}}^\p \big\Vert_{L^\infty([0,T);H^{s})} \ \leq \ M^\sharp_{s+6}\ ,\]
then $U\equiv U_{\text{CL}}^\p+U^c$ is consistent with Green-Naghdi equations~\eqref{eqn:GreenNaghdiMeanI} of order $s$ on $[0,t]$ for $t<T^\sharp$, at precision $\O(\eps_{\text{CL}}^\sharp)$ with
\[ \eps^\sharp_{\text{CL}} \ = \ C\ \max(\epsilon^2(\delta^2-\gamma)^2,\epsilon^4,\mu^2),\]
with $C=C(\frac1{s_0-3/2},M^\sharp_{s+6},\frac1{\delta_{\text{min}}},\delta_{\text{max}},\epsilon_{\text{max}},\mu_{\text{max}},|\lambda|,|\theta|)$ and $U^c$ is uniformly estimated as
\[ \big\Vert U^c \big\Vert_{L^\infty([0,T^\sharp];H^{s})}+\big\Vert \partial_t U^c \big\Vert_{L^\infty([0,T^\sharp];H^{s})} \leq C\ \max(\epsilon(\delta^2-\gamma),\epsilon^2,\mu) \min( t, 1). \]
\end{Proposition}

Equivalent results can be obtained when following the exact same strategy, but using lower order equations to describe the evolution of the decoupled waves, $v_\pm$.
We detail below the accuracy of such approximations.
\begin{Proposition}\label{prop:other}
Assume that the hypotheses of Proposition~\ref{prop:decompositionI} hold. Denote $U_{\text{eKdV}}^\p$, $U_{\text{KdV}}^\p$ and $U_{\text{iB}}^\p$, respectively, the solutions of the eKdV, KdV and iB approximations, as defined in Definition~\ref{def:other}. 
In each case, we assume that the decoupled approximation is uniformly estimated in $[0,T^\star]$, as in Proposition~\ref{prop:decompositionI}.
Then 
\begin{enumerate}[i.]
\item there exists $U^c=U^c[U_{\text{eKdV}}^\p]$ such that $U\equiv U_{\text{eKdV}}^\p +U^c$ is consistent with Green-Naghdi equations~\eqref{eqn:GreenNaghdiMeanI} of order $s$ on $[0,t]$ for $t<T^\star$, at precision $\O(\eps^\star_{\text{eKdV}})$, with
\[ \eps^\star_{\text{eKdV}} \ = \ C\ \times\ \Big(\ \max(\epsilon^2(\delta^2-\gamma)^2,\epsilon^4,\mu^2)\ (1+\sqrt t) \ + \ \max(\epsilon^3,\mu\epsilon) \ \Big)\ ;\]
\item there exists $U^c=U^c[U_{\text{KdV}}^\p]$ such that $U\equiv U_{\text{KdV}}^\p+U^c$ is consistent with Green-Naghdi equations~\eqref{eqn:GreenNaghdiMeanI} of order $s$ on $[0,t]$ for $t<T^\star$, at precision $\O(\eps^\star_{\text{KdV}})$ with
\[ \eps^\star_{\text{KdV}} \ = \ C\ \times\ \Big(\ \max(\epsilon^2(\delta^2-\gamma)^2,\epsilon^4,\mu^2)\ (1+\sqrt t) \ + \ \epsilon^2 \ \Big)\ ;\]
\item there exists $U^c=U^c[U_{\text{iB}}^\p]$ such that $U\equiv U_{\text{iB}}^\p+U^c$ is consistent with Green-Naghdi equations~\eqref{eqn:GreenNaghdiMeanI} of order $s$ on $[0,t]$ for $t<T^\star$, at precision $\O(\eps^\star_{\text{iB}})$ with
\[ \eps^\star_{\text{iB}} \ = \ C\ \times\ \Big(\ \max(\epsilon^2(\delta^2-\gamma)^2,\epsilon^4,\mu^2)\ (1+\sqrt t) \ + \ \max(\epsilon^2,\mu) \ \Big)\ ;\]
\end{enumerate}
where $C=C(\frac1{s_0-3/2},M^\sharp_{s+6},\frac1{\delta_{\text{min}}},\delta_{\text{max}},\epsilon_{\text{max}},\mu_{\text{max}},|\lambda|,|\theta|)$. Each time, the corrector term $U^c$ is estimated as follows:
\[ \big\Vert U^c \big\Vert_{L^\infty([0,T^\star];H^{s})}+\big\Vert \partial_t U^c \big\Vert_{L^\infty([0,T^\star];H^{s})} \leq C\ \max(\epsilon(\delta^2-\gamma),\epsilon^2,\mu) \min( t, \sqrt t) .\]

Moreover, if the decoupled approximation is sufficiently localized in space, then all the estimates are improved as in the second part of the Proposition~\ref{prop:decompositionI} (that is replacing $\sqrt t$ by $1$).
\end{Proposition}
\begin{Remark} \label{rem:WC}
The function $U^c$, which depends only on the decoupled waves $v_\pm$, is a first order corrector which allows to take into account the leading order coupling effects, and therefore reach the desired accuracy. Its construction and precise definition is explicitly displayed in the proof of the Proposition; see Section~\ref{sec:prop:decompositionI}, and in particular Definition~\ref{def:corrector}.
The approximate solution given by $U\equiv U_{\text{CL}}^\p+U^c$ can be seen as a weakly coupled model of independent interest, in the spirit of~\cite{Wright05}.
\end{Remark}
\begin{Remark}\label{rem:differentmodels}
In the estimates presented in Propositions~\ref{prop:decompositionI} and~\ref{prop:other}, the terms growing in $\O(\sqrt t)$ come from coupling effects between the two propagating waves, that are neglected in our decoupled models; this is why the accuracy is significantly better if the initial data is localized in space, as the two counterpropagating waves will be located far away from each other after some time. 

Uniformly bounded terms are the contribution of unidirectional errors, generated by the different manipulations on the equation ({\em e.g.} BBM trick), and eventually neglected terms in lower order approximations for Proposition~\ref{prop:other}, below. 

The magnitude of each contribution, and therefore the accuracy of the decoupled approximation, depends on \begin{enumerate}[i.]
\item the evolution equation considered (CL,eKdV,KdV, {\em etc.});
\item the size of the parameters $\epsilon,\mu$ as well as $\delta^2-\gamma$ (critical ratio);
\item the localization in space of the initial data;
\item the time-scale considered.
\end{enumerate}
In that respect, it is not obvious to decide which approximation is the best to consider, that is what is the simplest equation leading to the highest accuracy. We discuss several important cases (long-wave regime, Camassa-Holm regime with critical or non-critical ratio) in Section~\ref{ssec:discussion}, with several numerical simulations to support our conclusions.
\end{Remark}

\paragraph{Long time behavior} A notable difference in the statement of our results, when compared with previous work in~\cite{BonaColinLannes05,Duchene11a,ConstantinLannes09}, is that our estimates in Propositions~\ref{prop:decompositionI} and~\ref{prop:other} (and similarly Proposition~\ref{prop:unidirI}) are valid as long as the solution of the approximate scalar model is defined and uniformly estimated.
Let us note first that Proposition~\ref{prop:WPI} ensures that these results are not empty, but on the contrary are valid for long times (provided that $\nu_t^{\theta,\lambda}>0$ and the initial data sufficiently smooth). More precisely, one has straightforwardly the following result.
\begin{Corollary}[Existence and magnitude of $T^\star$, $T^\sharp$]
Let $(\zeta^0,v^0)=U^0\in H^{s+7},s\geq s_0>3/2$, and let $\p\in\P$, with additional restriction $\nu_t^{\theta,\lambda}>\nu_0>0$. 

Then there exists $C_1,C_2=C(\frac1{s_0-3/2},\mu_{\text{max}},\epsilon_{\text{max}},\delta_{\text{min}}^{-1},\delta_{\text{max}},|\lambda|,|\theta|,\nu_0^{-1},\big\vert U^0 \big\vert_{H^{s+7}_\mu})$, independent of $\p$, such that for any $M^\star_{s+6}\geq C_1$, the decoupled approximate solutions defined in
Propositions~\ref{prop:decompositionI} and~\ref{prop:other} are uniquely defined and satisfy the uniform bound of the Proposition, with
\[T^\star\geq C_2/\epsilon.\]
If $U^0\in X^{s+7}_2$, then there exists $C_1,C_2=C(\frac1{s_0-3/2},\mu_{\text{max}},\epsilon_{\text{max}},\delta_{\text{min}}^{-1},\delta_{\text{max}},|\lambda|,|\theta|,\nu_0^{-1},\big\vert U^0 \big\vert_{X^{s+7}_{2,\mu}})$ such that for any $M^\sharp_{s+6}\geq C_1$, one has
\[ T^\sharp\geq C_2/\max(\epsilon,\mu).\]
\end{Corollary}
One question, which is essential in the discussion of Section~\ref{ssec:discussion}, is whether these above estimates are optimal, or in the contrary can be extended to longer times. Let us discuss below some elements of answer.

 First, it is well-known that inviscid Burgers’ equation $\partial_t u + u\partial_x u$ will generate a shock in finite time, for any non-trivial, decreasing at infinity initial data. A simple scaling arguments shows that the inviscid Burgers decoupled approximation will therefore generate shocks in finite time $T\approx 1/\epsilon$. On the contrary, using conservation laws of the KdV equation, one can extend local well-posedness result inherited from the hyperbolic energy method, to global well-posedness in sufficiently regular spaces $H^s, \ s\geq1$ (see~\cite{KenigPonceVega93} among many other works). The same result holds for the modified KdV (when the power of the nonlinearity is non-quadratic but sub-critical), and is therefore valid for extended Korteweg-de Vries equations in Definition~\ref{def:other}, provided $\nu_t^{\theta,\lambda}=0$~\cite{Alejo12}.

Concerning higher order models, it is known that Camassa-Holm family of equations, related to~\eqref{eq:CLI}, can develop singularities in finite time in the form of wave breaking (see~\cite{Constantin01,Brandolese11}). In~\cite{ConstantinLannes09}, the authors show that wave breaking of solutions to~\eqref{eq:CLI} occurs for a specific set of parameters, at hyperbolic time $T\approx 1/\epsilon$, provided that the initial data is sufficiently big in $L^\infty$ norm. However, as the justification of our model assumes that the initial deformation is bounded in Sobolev norm, uniformly with respect to the parameters, assumptions of~\cite[Proposition 6]{ConstantinLannes09} cannot be justified. Thus the problem of the well-posedness of Constantin-Lannes equation for longer time than the one expressed in Proposition~\ref{prop:WPI} is still open, as far as we know. 

As for the localization in space, it is clear that our models, as they include dispersive terms, cannot be uniformly controlled in weighted Sobolev norms, globally in time. However, as the two counterpropagating waves move away from each other while their spatial localization weakens, it does not seem out of reach to extend Propositions~\ref{prop:decompositionI} and~\ref{prop:other} in order to uniformly control the coupling effects for very long time. On the other hand, in order to complete the full justification of our asymptotic models as described above, one has to obtain Hypothesis~\ref{conj:stab} over times greater than the typical hyperbolic time scale $\O(1/\epsilon)$. Thus the dispersive properties of the scalar models we consider, which are largely overlooked in this work but have been extensively studied in the literature, will play a predominant role in the behavior of the system at very long time. To conclude, let us note that, despite all the aforementioned difficulties, numerical simulations in Section~\ref{ssec:discussion} indicate that our uniform estimate in Proposition~\ref{prop:decompositionI} remains valid for times of order $\O(\epsilon^{-3/2})$.

\paragraph{Outline of the paper}
The precise presentation and justification of the governing (full Euler) equations of our system is introduced in Section~\ref{sec:models}. Using the shallowness assumption ($\mu\ll1$), we then introduce the so-called Green-Naghdi model~\eqref{eqn:GreenNaghdiMeanI}, as written out above. 
 Several equivalently precise models are constructed, using different variables; a short analysis on the linear well-posedness of these systems supports the choice of the shear layer-mean velocity, which yields~\eqref{eqn:GreenNaghdiMeanI}. 
 
 The well-posedness of the scalar models we consider, as well as the persistence of spatial decay of its solutions (Proposition~\ref{prop:WPI}), is proved in Section~\ref{sec:prop:WPI}.

 In Section~\ref{sec:unidirectional}, we turn to the case of the unidirectional approximation, and prove Proposition~\ref{prop:unidirI}. We then numerically investigate if the quite restricting condition on the initial data arises naturally, that is if after some time, the flow generated by any initial perturbation will eventually decompose into two almost purely unidirectional waves.
 
Finally, Section~\ref{sec:decomposition} is dedicated to the study of the decoupled models.
In Section~\ref{ssec:formal}, we present a formal argument which allows to derive the decoupled approximations, as defined precisely in~Definitions~\ref{def:CL} and~\ref{def:other}. As an intermediary step, we introduce a coupled asymptotic models, which we believe is of independent interest:~\eqref{eq:dSerrel}-\eqref{eq:dSerrer} is a simplified version of the Green-Naghdi equation, with same precision in the Camassa-Holm regime~\eqref{eqn:defRegimeCH}, as stated in Proposition~\ref{prop:ConsSerre}.
The main result, Proposition~\ref{prop:decompositionI}, is proved in Section~\ref{sec:prop:decompositionI}. Finally, Section~\ref{ssec:discussion} contains a discussion concerning the competition between the different scalar models, in various scenari, supported by numerical simulations.

 \section{Derivation of the Green-Naghdi system}
 \label{sec:models}

This section is dedicated to the construction and justification of the Green-Naghdi model~\eqref{eqn:GreenNaghdiMeanI}, which is the groundwork of our study. We first briefly recall the so-called {\em full Euler system}~\eqref{eqn:EulerComplet}, governing the behavior of two layers of immiscible, homogeneous, ideal, incompressible fluids under the sole influence of gravity. Following Craig-Sulem~\cite{CraigSulem93}, the system can be written as two evolution equations~\eqref{eqn:EulerCompletDN} coupling Zakharov's canonical variables~\cite{Zakharov68}, thanks to the use of Dirichlet-to-Neumann operators (Definition~\ref{def:D2N}).

After non-dimensionalizing our system of equations in order to make appear the dimensionless parameters at stake, we expand the Dirichlet-to-Neumann operators with respect to the shallowness, parameter, $\mu$ (Proposition~\ref{prop:expGH}).
Green-Naghdi models are obtained when replacing the Dirichlet-to-Neumann operators by their truncated expansion, and one obtains successively ~\eqref{eqn:GreenNaghdi1},~\eqref{eqn:GreenNaghdi2} and~\eqref{eqn:GreenNaghdiMeanA}. These systems are equivalent, but handle different velocity variables as unknowns. The latter is the system we base our study on, {\em i.e.~\eqref{eqn:GreenNaghdiMeanI}}, and considers the shear layer-mean velocity, obtained after integrating the velocity potential across the vertical layer in each fluid.
\footnote{Note that one cannot write the full Euler system in the simple form of two evolution equations using layer-mean velocity variables, as the pressure cannot be eliminated from the equation. The nice formulation of the Green-Naghdi system with the shear layer-mean velocity relies on the assumption of shallow water, $\mu\ll1$, which allows to approximate Zakharov's canonical variables in terms of layer-mean velocity variables. In~\cite{ChoiCamassa96,ChoiCamassa99}, Choi and Camassa formally construct similar Green-Naghdi models using layer-mean velocity variables from start to finish.}

All of these asymptotic models are justified by a consistency result, stating that solutions of the full Euler system satisfy Green-Naghdi asymptotic models up to a small remainder, of size $\O(\mu^2)$.

\subsection{The Full Euler system}
\label{ssec:FullEulerSystem}

The system we study consists in two layers of immiscible fluids, compelled below by a flat bottom and by a flat, rigid lid (see Figure~\ref{fig:SketchOfDomain}). We restrict ourselves to the two-dimensional case, {\em i.e.} to horizontal dimension $d=1$. The domains of the two fluids are infinite in the horizontal dimension, and the fluids are at rest at infinity. The depth of the upper and lower layers are, respectively, $d_1$ and $d_2$.

We assume that the interface between the two fluids is given as the graph of a function $\zeta(t,x)$ which expresses the deviation from its rest position $\{(x,z),z=0\}$ at the horizontal spatial coordinate $x$ and at time $t$. 
Therefore, at each time $t\ge 0$, the domains of the upper and lower fluid (denoted, respectively, $\Omega_1^t$ and $\Omega_2^t$), are given by
\begin{align*}
 \Omega_1^t \ &= \ \{\ (x,z)\in\RR^{d}\times\RR, \quad \zeta(t,x)\ \leq\ z\ \leq \ d_1\ \}, \\
 \Omega_2^t \ &= \ \{\ (x,z)\in\RR^{d}\times\RR, \quad -d_2 \ \leq\ z\ \leq \ \zeta(t,x)\ \}.
\end{align*}
We assume that the two domains are strictly connected, that is
\[ d_1+\zeta(t,x) \geq h>0, \qquad d_2+\zeta(t,x)\geq h>0.\]

We denote by $(\rho_1,{\mathbf v}_1)$ and $(\rho_2,{\mathbf v}_2)$ the mass density and velocity fields of, respectively, the upper and the lower fluid. The two fluids are assumed to be homogeneous and incompressible, so that the mass densities $\rho_1,\ \rho_2$ are constant, and the velocity fields ${\mathbf v}_1,\ {\mathbf v}_2$ are divergence free.
As we assume the flows to be irrotational, one can express the velocity field as gradients of a potential: ${\mathbf v}_i=\nabla\phi_i$ $(i=1,2)$, and the velocity potentials satisfy Laplace's equation
\[\partial_x^2 \phi_i \ + \ \partial_z^2 \phi_i \ = \ 0.\] 

 The fluids being ideal, they satisfy the Euler equations; the momentum equations can be integrated, which yields Bernoulli's equation:
 \[ \partial_t \phi_i+\frac{1}{2} |\nabla_{x,z} \phi_i|^2=-\frac{P}{\rho_i}-gz \quad \text{ in }\Omega^t_i \quad (i=1,2),\]
 where $P$ denotes the pressure inside the fluids.
 
 From the assumption that no fluid particle crosses the surface, the bottom or the interface, one deduces kinematic boundary conditions, and the set of equations is closed by the continuity of the pressure at the interface, assuming that there is no surface tension. \footnote{The surface tension effects should be included for our system to be well-posed. However, the surface tension is very small in practice, and does not play any role in our asymptotic analysis. See~\cite{Lannes10} for in depth study of this phenomenon.} 
 
Altogether, the governing equations of our problem are the following:
 \begin{equation} \label{eqn:EulerComplet}
\left\{\begin{array}{ll}
 \partial_x^2 \phi_i \ + \ \partial_z^2 \phi_i \ = \ 0 & \mbox{ in }\Omega^t_i, \ i=1,2,\\
 \partial_t \phi_i+\frac{1}{2} |\nabla_{x,z} \phi_i|^2=-\frac{P}{\rho_i}-gz & \mbox{ in }\Omega^t_i, \ i=1,2, \\
 \partial_{z}\phi_1 \ = \ 0 & \mbox{ on } \Gamma_t\equiv\{(x,z),z=d_1\}, \\
 \partial_t \zeta \ = \ \sqrt{1+|\partial_x\zeta|^2}\partial_{n}\phi_1 \ = \ \sqrt{1+|\partial_x\zeta|^2}\partial_{n}\phi_2 & \mbox{ on } \Gamma \equiv\{(x,z),z=\zeta(t,x)\},\\ 
 \partial_{z}\phi_2 \ = \ 0 & \mbox{ on } \Gamma_b\equiv\{(x,z),z=-d_2\}, \\
 P \text{ continous } & \mbox{ on } \Gamma ,
 \end{array}
\right.
\end{equation}
where $n$ is the unit upward normal vector at the interface.
\medskip

\paragraph{Rewriting the system as evolution equations.}
The system~\eqref{eqn:EulerComplet} can be reduced into two evolution equations coupling Zakharov's canonical variables, namely the deformation of the free interface from its rest position, $\zeta$, and the trace of the upper potential at the interface, $\psi$:
\[ \psi \ \equiv \ \phi_1(t,x,\zeta(t,x)).\]
The potentials $\phi_1$ and $\phi_2$ are uniquely deduced from $(\zeta,\psi)$ as the unique solutions of the following Laplace's problems:\footnote{The solution of the second Laplace's problem is defined up to a constant, which does not play any role in our analysis.}
\begin{equation}\label{eqn:Laplace}
\left\{\begin{array}{ll}
 (\ \partial_x^2 \ + \ \partial_z^2\ )\ \phi_1=0 & \mbox{in } \Omega_1, \\
\phi_1 =\psi & \mbox{on } \Gamma, \\
 \partial_{z}\phi_1 =0 & \mbox{on } \Gamma_t, \\

\end{array}
\right. \quad \text{and} \quad 
\left\{\begin{array}{ll}
 (\ \partial_x^2 \ + \ \partial_z^2\ )\ \phi_2=0 & \mbox{in } \Omega_2, \\
 \partial_{n}\phi_2 =\partial_{n}\phi_1 & \mbox{on } \Gamma_t,\\
 \partial_{z}\phi_2 =0 & \mbox{on } \Gamma_b.
\end{array}\right.
\end{equation}
More precisely, we define the so-called Dirichlet-Neumann operators.
\begin{Definition}[Dirichlet-Neumann operators]\label{def:D2N}
Let $\zeta\in W^{1,\infty}(\RR)$, and $\partial_x \psi\in H^{1/2}(\RR)$. Then we define 
\begin{align*}
G[\zeta]\psi \ &= \ \sqrt{1+|\partial_x\zeta|^2}\big(\partial_n \phi_1 \big)\id{z=\zeta} \ = \ \big(\partial_z \phi_1\big)\id{z=\zeta} -(\partial_x \zeta)\big(\partial_x \phi_1\big)\id{z=\zeta}, \\
H[\zeta]\psi \ &= \ \partial_x \big( {\phi_2}\id{z=\zeta}\big) \ = \ \partial_x \big( \phi_2(t,x,\zeta(t,x))\big),
\end{align*}
where $\phi_1$ and $\phi_2$ are uniquely defined (up to a constant for $\phi_2$) as the solutions in $H^2(\RR)$ of~\eqref{eqn:Laplace}.
\end{Definition}
The well-posedness of Laplace's problem~\eqref{eqn:Laplace}, and therefore the Dirichlet-Neumann operators, follow from classical arguments detailed, for example, in~\cite[Proposition 2.1]{Duchene10}.
\medskip

One can then rewrite the conservation of momentum equations in~\eqref{eqn:EulerComplet} at the interface, thanks to chain rule:
\[ \partial_t ( {\phi_i}\id{z=\zeta})+g \zeta + \frac12 \big|\partial_x \big({\phi_2}\id{z=\zeta}\big) \big|^2 +\frac{\big(G[\zeta]\psi +(\partial_x\zeta)(\partial_x \big({\phi_2}\id{z=\zeta}) \big)^2}{2(1+|\partial_x\zeta|^2)} \ = \ \frac{P\id{z=\zeta}}{\rho_i}.\]
 Using the continuity of the pressure at the interface, one deduces from the identities above
\[ \partial_t (\rho_2 H[\zeta]\psi -\rho_1 \partial_x \psi)+g(\rho_2-\rho_1)\partial_x \zeta + \frac12\partial_x \Big( \rho_2 |H[\zeta]\psi|^2 -\gamma |\partial_x \psi|^2\Big)\ =\ \partial_x \N(\zeta,\psi) ,\]
where $\N$ is defined as
\[ \N(\zeta,\psi) \ \equiv \ \frac{\rho_1\big(G[\zeta]\psi +(\partial_x\zeta)(\partial_x\psi)\big)^2 -\rho_2\big(G[\zeta]\psi +(\partial_x\zeta)H[\zeta]\psi \big)^2}{2(1+|\partial_x\zeta|^2)}.\]

The kinematic boundary condition at the interface is obvious, and the system~\eqref{eqn:EulerComplet} is therefore rewritten as
 \begin{equation} \label{eqn:EulerCompletDN}
\left\{\begin{array}{l}
 \partial_t (\rho_2 H[\zeta]\psi -\rho_1 \partial_x \psi)+g(\rho_2-\rho_1)\partial_x \zeta + \frac12\partial_x \Big( \rho_2 |H[\zeta]\psi|^2 -\rho_1 |\partial_x \psi|^2\Big)=\partial_x \N(\zeta,\psi) , \\
 \partial_t \zeta \ = \ G[\zeta]\psi .
 \end{array}
\right.
\end{equation}
which is exactly system (9) in~\cite{BonaLannesSaut08}.
\medskip

\paragraph{Nondimensionalization of the system.}
Thanks to an appropriate scaling, the two-layer full Euler system~\eqref{eqn:EulerCompletDN} can be written in dimensionless form.
The study of the linearized system (see~\cite{Lannes10}, for example), which can be solved explicitly, leads to a well-adapted rescaling. 

\medskip

Let $a$ be the maximum amplitude of the deformation of the interface. We denote by $\lambda$ a characteristic horizontal length, say the wavelength of the interface. Then the typical velocity of small propagating internal waves (or wave celerity) is given by
\[c_0 \ = \ \sqrt{g\frac{(\rho_2-\rho_1) d_1 d_2}{\rho_2 d_1+\rho_1 d_2}}.\]
Consequently, we introduce the dimensionless variables
\footnote{We choose $d_1$ as the reference vertical length. By doing so, we implicitly assume that the two layers of fluid have comparable depth, and therefore that the depth ratio $\delta$ do not approach zero or infinity. Note also that $c_0\to0$ as $\rho_2\to \rho_1$.}
\[\begin{array}{cccc}
 \t z \ \equiv\ \dfrac{z}{d_1}, \quad\quad & \t x\ \equiv \ \dfrac{x}{\lambda}, \quad\quad & \t t\ \equiv\ \dfrac{c_0}{\lambda}t,
\end{array}\]
the dimensionless unknowns
\[\begin{array}{cc}
 \t{\zeta}(\t x)\ \equiv\ \dfrac{\zeta(x)}{a}, \quad\quad& \t{\psi}(\t x)\ \equiv\ \dfrac{d_1}{a\lambda c_0}\psi(x),
\end{array}\]
and the four independent dimensionless parameters
\[
 \gamma\ =\ \dfrac{\rho_1}{\rho_2}, \quad \epsilon\ \equiv\ \dfrac{a}{d_1},\quad \mu\ \equiv\ \dfrac{d_1^2}{\lambda^2}, \quad \delta\ \equiv \ \dfrac{d_1}{d_2}.
\]
With this rescaling, the system~\eqref{eqn:EulerCompletDN} becomes (we withdraw the tildes for the sake of readability)
\begin{equation}\label{eqn:EulerCompletAdim}
\left\{ \begin{array}{l}
\displaystyle\partial_{ t}{\zeta} \ -\ \frac{1}{\mu}G^{\mu,\epsilon}\psi\ =\ 0, \\ \\
\displaystyle\partial_{ t}\Big(H^{\mu,\epsilon}\psi-\gamma \partial_x{\psi} \Big)\ + \ (\gamma+\delta)\partial_x{\zeta} \ + \ \frac{\epsilon}{2} \partial_x\Big(|H^{\mu,\epsilon}\psi|^2 -\gamma |\partial_x {\psi}|^2 \Big) \ = \ \mu\epsilon\partial_x\N^{\mu,\epsilon} \ ,
\end{array} 
\right. 
\end{equation}with
\[ \N^{\mu,\epsilon} \ \equiv \ \dfrac{\big(\frac{1}{\mu}G^{\mu,\epsilon}\psi+\epsilon(\partial_x{\zeta})H^{\mu,\epsilon}\psi \big)^2\ -\ \gamma\big(\frac{1}{\mu}G^{\mu,\epsilon}\psi+\epsilon(\partial_x{\zeta})(\partial_x{\psi}) \big)^2}{2(1+\mu|\epsilon\partial_x{\zeta}|^2)}, 
 \]
 and the dimensionless Dirichlet-to-Neumann operators defined by
 \begin{align*}
 &G^{\mu,\delta}\psi \ \equiv \ \sqrt{1+|\epsilon\partial_x\zeta|^2}\big(\partial_n \phi_1 \big)\id{z=\epsilon\zeta} \ = \ -\mu\epsilon(\partial_x\zeta) (\partial_x\phi_1)\id{z=\epsilon\zeta}+(\partial_z\phi_1)\id{z=\epsilon\zeta},\\
 &H^{\mu,\delta}\psi\ \equiv\ \partial_x \big(\phi_1\id{z=\epsilon\zeta}\big) \ = \ (\partial_x\phi_1)\id{z=\epsilon\zeta}+\epsilon(\partial_x \zeta)(\partial_z\phi_1)\id{z=\epsilon\zeta},
\end{align*}
where $\phi_1$ and $\phi_2$ are the solutions of the rescaled Laplace problems
\begin{align}
\label{eqn:Laplace1} &\left\{
\begin{array}{ll}
 \left(\ \mu\partial_x^2 \ +\ \partial_z^2\ \right)\ \phi_1=0 & \mbox{ in } \Omega_1\equiv \{(x,z)\in \RR^{2}, \epsilon{\zeta}(x)<z<1\}, \\
\partial_z \phi_1 =0 & \mbox{ on } \Gamma_t\equiv \{z=1\}, \\
 \phi_2 =\psi & \mbox{ on } \Gamma\equiv \{z=\epsilon \zeta\}, 
\end{array}
\right.\\ 
\label{eqn:Laplace2}&\left\{
\begin{array}{ll}
 \left(\ \mu\partial_x^2\ + \ \partial_z^2\ \right)\ \phi_2=0 & \mbox{ in } \Omega_2\equiv\{(x,z)\in \RR^{2}, -\frac{1}{\delta}<z<\epsilon\zeta\}, \\
\partial_{n}\phi_1 = \partial_{n}\phi_2 & \mbox{ on } \Gamma_2, \\
 \partial_{z}\phi_2 =0 & \mbox{ on } \Gamma_b\equiv \{z=-\frac{1}{\delta}\}.
\end{array}
\right.
\end{align}
Similarly as in Definition~\ref{def:D2N}, the Dirichlet-Neumann operators are well defined, provided that $\zeta\in W^{1,\infty}(\RR)$, $\partial_x \psi\in H^{1/2}(\RR)$, and 
the following condition holds: there exists $h>0$ such that
\begin{equation} \label{eqn:connected}
h_1 \ \equiv\ 1-\epsilon\zeta \geq h>0, \quad h_2 \ \equiv \ \frac1\delta +\epsilon \zeta\geq h>0.
\end{equation}
Any reference of the {\em full Euler system} in the following concerns system~\eqref{eqn:EulerCompletAdim} with operators defined as above.

\subsection{Asymptotic models}
Our aim is now to obtain asymptotic models for the full Euler system~\eqref{eqn:EulerCompletAdim}, using smallness of dimensionless parameters. In the following, we will consider the case of shallow water, namely
\[ \mu \ \ll \ 1.\]
The key ingredient comes from the expansions of the Dirichlet-Neumann operators, in terms of $\mu$. Replacing the operators by the leading order terms of these expansions allow to obtain the desired asymptotic models, which are consequently justified in the sense of consistency. 

As a second step, we rewrite the equations using the shear layer-mean velocity as unknown.
One benefit of such a choice is that it yields to a much better behavior concerning the linear well-posedness, as we discuss at the end of this section.

Our method has been used by Alvarez-Samaniego and Lannes~\cite{Alvarez-SamaniegoLannes08} in the case of the water-wave problem (one layer of fluid, with free surface), and lead the authors to a complete rigorous justification of the so-called Green-Naghdi equations~\cite{GreenNaghdi76}. In the case of two layers with a free surface, a shallow water model (first order) and Boussinesq-type models (in the long wave regime) have been derived and justified in the sense of consistency in~\cite{BonaLannesSaut08}; the analysis below is therefore an extension of their work. Similar models as our Green-Naghdi system have been formally obtained in~\cite{ChoiCamassa99}, as well as in~\cite{OstrovskyGrue03} (with the additional assumption of $\gamma\approx 1$) and in~\cite{CraigGuyenneKalisch05}, but it is the first time to our knowledge that a rigorous justification is provided. Let us also mention the work concerning the case of two layers of fluids with an interface and a free surface: Green-Naghdi-type models have been derived in~\cite{BarrosGavrilyuk07,BarrosGavrilyukTeshukov07}, and justified in the sense of consistency in~\cite{Duchene10}. One could formally recover our models from~\cite[(44) and (60)]{Duchene10} by forcing the surface to be flat ($\alpha\equiv0$ using notation therein).
\medskip

\noindent{\bf Expansion of the Dirichlet-Neumann operators.}\\
The main ingredients of the following Proposition are given in~\cite{BonaLannesSaut08}, and we extend their result one order further.
\begin{Proposition}[Expansion of the Dirichlet-Neumann operators]\label{prop:expGH}
Let $s_0>1/2$ and $s\geq s_0 + 1/2$. Let $\psi$ be such that $\partial_x \psi \in H^{s+11/2}(R)$, and $\zeta \in H^{s+9/2}(\RR)$. Let $h_1=1-\epsilon\zeta$ and $h_2=1/\delta+\epsilon\zeta$ such that~\eqref{eqn:connected} is satisfied. 
Then
\begin{align}\label{eqn:expG0}
\big\vert \frac1\mu G^{\mu,\epsilon}\psi \ -\ \partial_x(h_1\partial_x\psi) \big\vert_{H^s} \ \leq \ \mu\ C_0,\\
\label{eqn:expG}
\big\vert \frac1\mu G^{\mu,\epsilon}\psi \ -\ \partial_x(h_1\partial_x\psi)\ -\ \mu\frac13\partial_x^2({h_1}^3\partial_x^2\psi)\big\vert_{H^s} \ \leq \ \mu^2\ C_1,\\
\label{eqn:expH0}
 \big\vert H^{\mu,\epsilon}\psi +\frac{h_1}{h_2}\partial_x\psi \big\vert_{H^s} \ \leq \ \mu\ C_0,\\
 \label{eqn:expH}
 \big\vert H^{\mu,\epsilon}\psi +\frac{h_1}{h_2}\partial_x\psi -\frac\mu{3h_2}\partial_x\Big({h_2}^3\partial_x\big(\frac{h_1}{h_2}\partial_x\psi\big)-{h_1}^3\partial_x^2\psi\Big) \big\vert_{H^s} \ \leq \ \mu^2\ C_1,
 \end{align}
with $C_j=C(\frac1{h},\epsilon_{\text{max}},\mu_{\text{max}},\frac1{\delta_{\text{min}}},\delta_{\text{max}},\big\vert \zeta \big\vert_{H^{s+5/2+2j}},\big\vert \partial_x\psi \big\vert_{H^{s+7/2+2j}} )$. The estimates are uniform with respect to the parameters $(\mu,\epsilon,\gamma,\delta)\in\P$, as defined in~\eqref{eqn:defRegimeFNL}.
\end{Proposition}
\begin{proof}
As remarked in~\cite{BonaLannesSaut08}, the operator $G^{\mu,\epsilon}\psi$ can be deduced from similar operator in the (one layer) water wave case with flat bottom:
\[ G^{\mu,\epsilon}\psi \ = \ -\mathcal G[-\epsilon\zeta]\psi,\]
where $\mathcal G$ is defined in~\cite[Section 3]{Alvarez-SamaniegoLannes08}, and estimates~\eqref{eqn:expG0},\eqref{eqn:expG} follows from Proposition 3.8 therein.
\medskip

Estimate~\eqref{eqn:expH0} is given in~\cite[Section 2.2.2]{BonaLannesSaut08}, and we obtain~\eqref{eqn:expH} using the same method, expanding one order further. Let us detail the strategy.
\medskip

The first step consists in rewriting the scaled Laplace problem~\eqref{eqn:Laplace2} into a variable-coefficient, boundary-value problem on the flat strip $\S := \RR \times (-1, 0)$ using the diffeomorphism
\[ \begin{array}{rcl}
& \S &\to \Omega_2, \\
\sigma: & (x,z) & \mapsto \sigma (x,z)\equiv \big(x,(1/\delta+\epsilon\zeta )z +\epsilon\zeta \big) .
\end{array}\]
Now, one can check (see~\cite{BonaLannesSaut08} or~\cite[Proposition 2.1]{Duchene10}) that $\phi_2$ solves~\eqref{eqn:Laplace2} if and only if $\underline{\phi_2}\equiv \phi_2 \circ \sigma$ satisfies
\begin{equation}\label{eqn:FlatLaplace}
\left\{ \begin{array}{ll}
\nabla_{x,z} \cdot Q^\mu[\epsilon\zeta ] \nabla_{x,z} \underline{\phi_2} \ = \ 0 & \text{ in } \S \\
\partial_n \underline{\phi_2}\id{z=0}=G^{\mu,\epsilon}\psi & \partial_n \underline{\phi_2}\id{z=-1}=0,
\end{array} \right.
\end{equation}
with
\[ Q^\mu[\epsilon\zeta] \ \equiv \ \begin{pmatrix} 
\mu\partial_z\sigma & -\mu\partial_x\sigma \\
-\mu\partial_x\sigma & \frac{1+\mu|\partial_x \sigma|^2}{\partial_z \sigma}
\end{pmatrix} \ = \ \frac1{1/\delta+\epsilon\zeta}\begin{pmatrix} 
0 & 0 \\
0 & 1 \end{pmatrix}+\mu \begin{pmatrix} 
1/\delta+\epsilon\zeta & -(z+1)\epsilon\partial_x\zeta \\
-(z+1)\epsilon\partial_x\zeta & \dfrac{|(z+1)\epsilon\partial_x\zeta|^2}{1/\delta+\epsilon\zeta}
\end{pmatrix},\]
and where $\partial_n \underline{\phi_2}$ stands for the upward conormal derivative associated to the elliptic operator involved:
\[ \partial_n \underline{\phi_2}\id{z=z_0} \ = \ \begin{pmatrix} 0\\1 \end{pmatrix} \cdot Q\nabla_{x,z} \underline{\phi_2}\id{z=z_0}.\]
The asymptotic expansion of the Dirichlet-Neumann operator $H^{\mu,\epsilon}$ is deduced from the identity
\[ H^{\mu,\epsilon}\psi \ = \ \partial_x \big( \underline{\phi_2}\id{z=0} \big).\]

The second step consists in computing the formal expansion of $\underline{\phi_2}$ the solution of~\eqref{eqn:FlatLaplace}, as
\[ \underline{\phi_2} \ = \ \phi^{(0)} \ + \ \mu \phi^{(1)} \ + \ \mu^2 \phi^{(2)}\ + \ \mu^3 \phi_r,\]
One can solve~\eqref{eqn:FlatLaplace} at each order, using the obvious expansion of $Q^\mu$, as well as the known expansion of the operator $G^{\mu,\epsilon}\psi$. This yields explicit formulas for $\phi^{(i)}$, and the estimate follows from adequate control of the residual $\phi_r$.
\medskip

At first order, one has
\[
\left\{ \begin{array}{ll}
\partial_z \big( \frac1{1/\delta+\epsilon\zeta}\partial_z \phi^{(0)} \big) \ = \ 0 & \text{ in } \S \\
\frac1{1/\delta+\epsilon\zeta}\partial_z \phi^{(0)}\id{z=0}=0 & \frac1{1/\delta+\epsilon\zeta}\partial_z \phi^{(0)}\id{z=-1}=0,
\end{array} \right.
\]
so that
\[ \phi^{(0)}(x,z) \ = \ \phi^{(0)}(x),\]
independent of $z$. At next order, one has (denoting $h_2(x)=1/\delta+\epsilon\zeta(x)$)
\[
\left\{ \begin{array}{l}
\dfrac1{h_2}\partial_z^2 \phi^{(1)}\ = \ -\nabla_{x,z} \cdot\begin{pmatrix} 
h_2 & -(z+1)\epsilon\partial_x\zeta \\
-(z+1)\epsilon\partial_x\zeta & \dfrac{|(z+1)\epsilon\partial_x\zeta|^2}{h_2}
\end{pmatrix} \nabla_{x,z} \phi^{(0)} \ = \ -h_2 \partial_x^2\phi^{(0)} \quad \text{ in } \S \\
\frac1{h_2}\partial_z \phi^{(1)}\id{z=0}=(\epsilon\partial_x\zeta)(\partial_x\phi^{0})+\partial_x(h_1\partial_x\psi) \hfill \frac1{h_2}\partial_z \phi^{(0)}\id{z=-1}=0 ,
\end{array} \right.
\]
where we used that $\phi^{(0)}$ is independent of $z$. The above system is a second order ordinary differential equation, which is solvable under the condition
\[ \partial_x(h_1\partial_x\psi) \ = \ -\partial_x(h_2\partial_x \phi^{(0)}) ,\]
and whose solution is then
\[ \phi^{(1)}(x,z) \ = \ -\frac12(z+1)^2 {h_2}^2 \partial_x^2\phi^{(0)}+\phi^{(1)}_0(x) ,\]
with $\phi^{(1)}_0(x)$ being a function independent of $z$, to be determined later. Note that since the horizontal dimension is one,\footnote{In the 2d case, one should introduce the non-local operator of orthogonal projection onto the gradient vector fields, as in~\cite{BonaLannesSaut08}, to ensure that the right hand side is a gradient.} and using the fact that the fluids are at rest at infinity, integrating the compatibility conditions yields
\[ \partial_x \phi^{(0)} \ = \ -\frac{h_1}{h_2}\partial_x\psi.\]

\medskip

Let us turn to the next order. One has 
\[
\left\{ \begin{array}{l}
\dfrac1{h_2}\partial_z^2 \phi^{(2)}\ = \ -h_2\partial_x^2 \phi^{(1)}_0 +(z+1)^2F(x) \quad \text{ in } \S \\
\frac1{h_2}\partial_z \phi^{(1)}\id{z=0}=G(x)+(\partial_x h_2)(\partial_x \phi^{(1)}_0)+\frac13\partial_x^2({h_1}^3\partial_x^2\psi) \qquad \frac1{h_2}\partial_z \phi^{(0)}\id{z=-1}=0 ,
\end{array} \right.
\]
with
\[
F(x) \ = \ \frac{{h_2}^3}2\partial_x^4 \phi^{(0)}\quad \text{ and } \quad 
G(x) \ = \ -\frac16\partial_x({h_2}^3)\partial_x^3 \phi^{(0)}.
\]
Solving the first identity with boundary condition $\partial_z \phi^{(0)}\id{z=-1}=0 $ yields
\[ \phi^{(2)}(x,z) \ = \ -\frac{(z+1)^2}2 {h_2}^2\partial_x^2 \phi^{(1)}_0 F(x)+\frac{(z+1)^4}{12}h_2F(x)+\phi^{(2)}_0(x),\]
with $\phi^{(2)}_0(x)$ independent of $z$ (and which can be set to zero for simplicity), and solving the boundary condition at $z=0$ yields the compatibility condition:
\[-h_2\partial_x^2 \phi^{(1)}_0+\frac13 F(x)=G(x)+(\partial_x h_2)(\partial_x \phi^{(1)}_0)+\frac13\partial_x^2({h_1}^3\partial_x^2\psi), \]
or, equivalently,
\[ \partial_x(h_2\partial_x \phi^{(1)}_0) \ = \ \frac16 \partial_x( {h_2}^3\partial_x^3 \phi^{(0)})-\frac13\partial_x^2({h_1}^3\partial_x^2\psi). \]
Finally, integrating this identity and using the expression of $\partial_x \phi^{(0)}$ obtained above, one deduces
\[ h_2\partial_x \phi^{(1)}_0 \ = \ -\frac13\partial_x({h_1}^3\partial_x^2\psi)-\frac16 {h_2}^3\partial_x^2( \frac{h_1}{h_2}\partial_x\psi).\]

The final step is as follows. Let us define 
\[ \underline{\phi_{2,\text{app}}} \ \equiv \ \phi^{(0)} \ + \ \mu \phi^{(1)} , \]
where $\phi^{(0)}$ and $\phi^{(1)}$ have been obtained by the above calculations. Note that 
\begin{align*}
 H_\text{app} \ = \ \partial_x \big( \underline{\phi_{2,\text{app}}} \id{z=0} \big) \ &= \ \partial_x \phi^{(0)}(x)+\mu\partial_x\Big( \ - \frac12 {h_2}^2 \partial_x^2\phi^{(0)}+\phi^{(1)}_0(x)\Big) \\
 &= \ -\frac{h_1}{h_2}\partial_x\psi +\frac\mu{3h_2}\partial_x\Big(-{h_1}^3\partial_x^2\psi+{h_2}^3\partial_x\big(\frac{h_1}{h_2}\partial_x\psi\big)\Big) ,
 \end{align*} 
which is exactly the expansion in ~\eqref{eqn:expH}. Therefore, the result follows from an adequate estimate on 
\[ u \ \equiv \ \partial_x \big( \underline{\phi_{2}} \id{z=0}\big) -\partial_x \big( \underline{\phi_{2,\text{app}}} \id{z=0}\big) .\]

Let us first note that one has straightforwardly
\[ \big\vert \phi^{(2)} \big\vert_{H^s} \ \leq \ C(\frac1{h},\epsilon_{\text{max}},\mu_{\text{max}},\big\vert \zeta \big\vert_{H^{s+9/2}},\big\vert \partial_x\psi \big\vert_{H^{s+11/2}} ).\]
Then, using previous calculations, $v \equiv \underline{\phi_{2}} - \underline{\phi_{2,\text{app}}} - \mu^2\phi^{(2)}$ satisfies the system
\begin{equation}\label{eqn:FlatLaplacerhs}
\left\{ \begin{array}{ll}
\nabla_{x,z} \cdot Q^\mu[\epsilon\zeta ] \nabla_{x,z} v \ = \ \mu^2\nabla_{x,z}\cdot \mathbf h & \text{ in } \S \\
\partial_n \underline{\phi_2}\id{z=0}=V +\mu^2 \begin{pmatrix} 0\\1 \end{pmatrix} \cdot {\mathbf h}\id{z=0}& \partial_n \underline{\phi_2}\id{z=-1}=\mu^2 \begin{pmatrix} 0\\1 \end{pmatrix} \cdot {\mathbf h}\id{z=-1},
\end{array} \right.
\end{equation}
with $\mathbf h \ = \ Q^\mu \nabla_{x,z} \phi^{(2)}$, and
$V=G^{\mu,\epsilon}\psi+\mu \partial_x(h_1\partial_x\psi)+\mu^2\frac13\partial_x^2({h_1}^3\partial_x^2\psi)$, so that~\eqref{eqn:expG} yields
\[ \big\vert V \big\vert_{H^s} \ \leq \ \mu^3 C(\frac1{h},\epsilon_{\text{max}},\mu_{\text{max}},\big\vert \zeta \big\vert_{H^{s+9/2}},\big\vert \partial_x\psi \big\vert_{H^{s+11/2}} ).\]
One can now apply Proposition 3 of~\cite{BonaLannesSaut08}, after straightforward adjustments, and deduce
\[ \big\vert \partial_x v\id{z=0} \big\vert_{H^s} \ \leq \ \mu^2 C(\frac1{h},\epsilon_{\text{max}},\mu_{\text{max}},\big\vert \zeta \big\vert_{H^{s+9/2}},\big\vert \partial_x\psi \big\vert_{H^{s+11/2}} ).\]
The Proposition is proved.
\end{proof}

\paragraph{The Green-Naghdi models.} Let us now plug the expansions of Proposition~\ref{prop:expGH} into the full Euler system~\eqref{eqn:EulerCompletAdim}, and withdraw $\O(\mu^2)$ terms. One obtains
\begin{equation}\label{eqn:GreenNaghdi1}
\left\{ \begin{array}{l}
\displaystyle\partial_{ t}{\zeta} \ - \ \partial_x(h_1\partial_x\psi)\ -\ \mu\frac13\partial_x^2({h_1}^3\partial_x^2\psi) \ =\ 0, \\ \\
\displaystyle\partial_{ t}\Big(-\frac{h_1+\gamma h_2}{h_2}\partial_x{\psi}+\mu\frac1{h_2}\partial_x\big( \P\partial_x\psi\big) \Big)\ + \ (\gamma+\delta)\partial_x{\zeta} \\
\qquad \qquad \displaystyle + \ \dfrac{\epsilon}{2} \partial_x\Big(|-\frac{h_1}{h_2}\partial_x\psi + \mu\frac1{h_2}\partial_x\big(\P\partial_x\psi\big) |^2 -\gamma |\partial_x {\psi}|^2 \Big) = \ \mu\epsilon\partial_x\big( \N[h_1,h_2]\partial_x\psi\big) \ ,
\end{array} 
\right. 
\end{equation}
where we denote $\P\partial_x\psi \ = \ \P[h_1,h_2]\partial_x\psi$, and with the following operators
\begin{align*}
 \P[h_1,h_2]V \ &\equiv \ \frac13 \Big({h_2}^3\partial_x\big(\frac{h_1}{h_2}V \big)-{h_1}^3\partial_x V\Big),\\
\N[h_1,h_2]V \ &\equiv \ \frac12 \left( \big(\partial_x(h_1 V) -\epsilon(\partial_x{\zeta})\frac{h_1}{h_2}V \big)^2\ -\ \gamma\big(h_1\partial_x^2\psi \big)^2\right).
 \end{align*} 
 
 Our model is justified by the following consistency result.
\begin{Proposition}\label{prop:ConsGreenNaghdi1}
 Let $U^\p\equiv(\zeta^\p,\psi^\p)_{\p\in\P}$ be a family of solutions of the full Euler system~\eqref{eqn:EulerCompletAdim}, such that~\eqref{eqn:connected} holds, and $\zeta^\p\in W^{1}([0,T);H^{s+11/2}),\partial_x\psi^\p\in W^{1}([0,T);H^{s+13/2})$ with $s\geq s_0+1/2$, $s_0>1/2$, and uniformly with respect to $\p\in\P$; see~\eqref{eqn:defRegimeFNL}. Then $U^\p$ satisfies~\eqref{eqn:GreenNaghdi1}, up to a remainder $R$, bounded by
\[ \big\Vert R \big\Vert_{L^\infty([0,T);H^s)} \ \leq \ \mu^2\ C,\]
with $C=C(\frac1{s_0-1/2},\frac1{h},\epsilon_{\text{max}},\mu_{\text{max}},\frac1{\delta_{\text{min}}},\delta_{\text{max}},\big\Vert \zeta^\p \big\Vert_{W^{1}([0,T);H^{s+11/2})},\big\Vert \partial_x\psi^\p \big\Vert_{W^{1}([0,T);H^{s+13/2})} )$.
\end{Proposition}
\begin{proof}
When plugging $U^\p$ into~\eqref{eqn:GreenNaghdi1}, and after some straightforward computation, one can clearly estimate the remainder using the expansions of the Dirichlet-Neuman operators in Proposition~\ref{prop:expGH}. The only non-trivial term comes from
$\partial_{ t}\big(H^{\mu,\epsilon}\psi^\p-\gamma \partial_x{\psi}^\p \big)$, which requires the corresponding expansion of $\partial_{ t}\big(H^{\mu,\epsilon}\psi^\p\big)$. This can be obtained as in Proposition~\ref{prop:expGH}, using the elliptic problems satisfied by the time derivative of the potentials. The expansion follows in the same way, provided that $\zeta^\p\in W^{1}([0,T);H^{s+11/2}),\partial_x\psi^\p\in W^{1}([0,T);H^{s+13/2})$.
See the proof of~\cite[Proposition 2.12]{Duchene10} for more details.
\end{proof}

In~\cite{BonaLannesSaut08}, the authors use the shear velocity as for the velocity variable:
\begin{equation}\label{eqn:defv}
v \ \equiv\ \partial_x\left(\big(\phi_2 - \gamma\phi_1 \big)\id{z=\eps\zeta}\right)\ =\ H^{\mu,\epsilon}\psi-\gamma \partial_x\psi.\end{equation}
The expansion of $H^{\mu,\epsilon}$ in~\eqref{eqn:expH} allows to obtain approximately $\partial_x \psi$ as a function of $v$:
\begin{equation}\label{eqn:psitov}
 \partial_x \psi \ = \ -\frac{h_2}{h_1+\gamma h_2}v+\mu\frac1{h_1+\gamma h_2}\partial_x\Big( \P[h_1,h_2]\big(-\frac{h_2}{h_1+\gamma h_2}v\big) \Big) \ + \ \O(\mu^2).\end{equation}
Here and in the following, we use the notation $\O(\cdot)$ for estimates as in Proposition~\ref{prop:expGH}.

Plugging~\eqref{eqn:defv} into~\eqref{eqn:GreenNaghdi1}, and again withdrawing $\O(\mu^2)$ terms, yields
\begin{equation}\label{eqn:GreenNaghdi2}
\left\{ \begin{array}{l}
\displaystyle\partial_{ t}{\zeta} \ + \ \partial_x \Big(\frac{h_1h_2}{h_1+\gamma h_2}v\Big)\ -\ \mu\partial_x\Big( \Q[h_1,h_2]v \Big) \ =\ 0, \\ \\
\displaystyle\partial_{ t}v \ + \ (\gamma+\delta)\partial_x{\zeta} \ + \ \frac{\epsilon}{2} \partial_x\Big(\dfrac{{h_1}^2 -\gamma {h_2}^2 }{(h_1+\gamma h_2)^2} |v|^2\Big) \ = \ \mu\epsilon\partial_x\big(\R[h_1,h_2]v \big) \ ,
\end{array} 
\right. 
\end{equation}
with the following operators:
 \begin{align*} 
 \Q[h_1,h_2]V \ &\equiv \ \frac{-1}{3(h_1+\gamma h_2)}\Bigg(h_1 \partial_x \Big({h_2}^3\partial_x\big(\frac{h_1\ V}{h_1+\gamma h_2} \big)\Big)\ +\ \gamma h_2\partial_x \Big( {h_1}^3\partial_x \big(\frac{h_2\ V}{h_1+\gamma h_2}\big)\Big)\Bigg), \\
 \R[h_1,h_2]V \ &\equiv \ \frac12 \Bigg( \Big( h_2\partial_x \big( \frac{h_1\ V}{h_1+\gamma h_2} \big)\Big)^2\ -\ \gamma\Big(h_1\partial_x \big(\frac{h_2\ V}{h_1+\gamma h_2} \big)\Big)^2\Bigg)\\
 &\qquad + \gamma\frac{h_1+h_2}{3(h_1+\gamma h_2)^2}V\ \partial_x\Big( {h_2}^3\partial_x\big(\frac{h_1\ V}{h_1+\gamma h_2} \big)-{h_1}^3\partial_x \big(\frac{h_2\ V}{h_1+\gamma h_2}\big)\Big) . 
 \end{align*}

System~\eqref{eqn:GreenNaghdi2} is justified as an asymptotic model for the full Euler system~\eqref{eqn:EulerCompletAdim}, by a consistency result, with the same precision as~\eqref{eqn:GreenNaghdi1}. 
\begin{Proposition}\label{prop:ConsGreenNaghdi2}
Let $U^\p\equiv(\zeta^\p,\psi^\p)_{\p\in\P}$ be a family of solutions of the full Euler system~\eqref{eqn:EulerCompletAdim}, such that~\eqref{eqn:connected} holds, and $\zeta^\p\in W^{1}([0,T);H^{s+11/2}),\partial_x\psi^\p\in W^{1}([0,T);H^{s+13/2})$ with $s\geq s_0+1/2$, $s_0>1/2$, and uniformly with respect to $\p\in\P$; see~\eqref{eqn:defRegimeFNL}. Define $v^\p$ as in~\eqref{eqn:defv}. Then $(\zeta^\p,v^\p)$ satisfies~\eqref{eqn:GreenNaghdi2}, up to a remainder $R$, bounded by
\[ \big\Vert R \big\Vert_{L^\infty([0,T);H^s)} \ \leq \ \mu^2\ C,\]
with $C=C(\frac1{s_0-1/2},\frac1{h},\epsilon_{\text{max}},\mu_{\text{max}},\frac1{\delta_{\text{min}}},\delta_{\text{max}},\big\Vert \zeta^\p \big\Vert_{W^{1}([0,T];H^{s+11/2})},\big\Vert \partial_x\psi^\p \big\Vert_{W^{1}([0,T];H^{s+13/2})} )$.
\end{Proposition} 
The proof of Proposition~\ref{prop:ConsGreenNaghdi2} is identical to the proof of Proposition~\ref{prop:ConsGreenNaghdi1}, once one obtains a rigorous statement of~\eqref{eqn:psitov} in $W^1([0,T];H^{s+1})$ norm, thus we omit it.
 \medskip
 
\noindent{\bf Using layer-mean velocities.}\\
As we have seen, several different velocity variables are natural when expressing the Green-Naghdi equations. In the following, we choose to use, as in~\cite{ChoiCamassa96,ChoiCamassa99} for example,
the {\em shear layer-mean velocity}, defined by
\begin{equation}\label{eqn:defbarv}
\bar v \ \equiv \ \overline{u}_{2}\ -\ \gamma \overline{u}_{1}, 
\end{equation}
where $\overline{u}_{1}, \ \overline{u}_{2}$ are the layer-mean velocities integrated across the vertical layer in each fluid:
\[
\overline{u}_{1}(t,x) \ = \ \frac{1}{h_1(t,x)}\int_{\epsilon\zeta(t,x)}^{1} \partial_x \phi_1(t,x,z) \ dz, \quad \text{ and }\quad \overline{u}_{2}(t,x) \ = \ \frac{1}{h_2(t,x)}\int_{-\frac1\delta}^{\epsilon\zeta(t,x)} \partial_x \phi_2(t,x,z) \ dz.
\]
We see two main benefits for such a choice. First, the equation describing the evolution of the deformation of the interface is an exact equation, and not an $\O(\mu^2)$ approximation. What is more, the system obtained using layer-mean velocities have a better behavior as for the linear well-posedness. These two facts shall be discussed in more details below.
\medskip

When integrating Laplace's equation in~\eqref{eqn:EulerComplet} against a test function $\tilde \varphi(x,z)=\varphi(x)$ on the lower domain $\Omega_2$, and using boundary kinematic equations, one has
\begin{align*}
0 \ &= \ \iint_{\Omega_2} \tilde\varphi (\mu\partial_x^2+\partial_z^2 )\phi_2 \ dx \ dz \ = \ -\iint_{\Omega_2} \nabla^\mu_{x,z} \tilde\varphi \cdot\nabla^\mu_{x,z} \phi_2 \ dx \ dz \ + \ \int_{\Gamma} \varphi\ \partial_n \phi_2 \ - \ \int_{\Gamma_b} \varphi\ \partial_z\phi_2\\
 &= \ -\int_\RR dx \sqrt\mu\partial_x \varphi \int_{-1/\delta}^{\epsilon\zeta} \sqrt\mu\partial_x \phi_2(x,z) \ dz \ \ - \ \int_\RR dx G^{\mu,\epsilon}\psi,
\end{align*}
where $\nabla\mu_{x,z}=(\sqrt\mu\partial_x,\partial_z)^T$, so that one deduces that for any $x\in\RR$,
\[ -\partial_x \big(h_2 \overline{u}_{2}\big) \ = \ \frac1\mu G^{\mu,\epsilon}\psi.\]
In the same way, one has
\begin{align*}
0 \ &= \ \iint_{\Omega_1} \tilde\varphi (\mu\partial_x^2+\partial_z^2 )\phi_1 \ dx \ dz \ = \ -\iint_{\Omega_1} \nabla^\mu_{x,z} \tilde\varphi \cdot\nabla^\mu_{x,z} \phi_1 \ dx \ dz \ - \ \int_{\Gamma} \varphi\ \partial_n \phi_1 \ + \ \int_{\Gamma_t} \varphi\ \partial_z\phi_1\\
 &= \ -\int_\RR dx\ \partial_x \sqrt\mu\varphi \int_{\epsilon\zeta}^1 \sqrt\mu\partial_x \phi_1(x,z) \ dz \ \ + \ \int_\RR dx G^{\mu,\epsilon}\psi,
\end{align*}
so that \[ \partial_x \big(h_1 \overline{u}_{1}\big) \ = \ \frac1\mu G^{\mu,\epsilon}\psi.\]

One deduces from the two above identities, imposing zero boundary conditions at infinity (the fluid is at rest at infinity), 
\[h_2 \overline{u}_{2} \ = \ -h_1 \overline{u}_{1}, \quad \text{ and }\quad \bar v \ = \ \frac{ h_1+\gamma h_2}{h_1}\overline{u}_{2} \ = \ -\frac{ h_1+\gamma h_2}{h_2}\overline{u}_{1},\]

It follows that the first equation in~\eqref{eqn:EulerCompletAdim} becomes
\[ \partial_t\zeta \ = \ \frac1\mu G^{\mu,\epsilon}\psi \ = \ -\partial_x \Big(\frac{h_1h_2}{h_1+\gamma h_2}\bar v\Big).\]
Let us emphasize again that this identity is exact, as opposed to the $\O(\mu^2)$ approximations in previous asymptotic models.

One also deduces an expansion of $\bar v$ in terms of $v$, using Proposition~\ref{prop:expGH} (or, equivalently, identifying the above identity with the first line of~\eqref{eqn:GreenNaghdi2}). It follows
\[ v \ = \ \bar v \ + \ \mu \frac{h_1+\gamma h_2}{h_1h_2} \Q[h_1,h_2]\bar v \ + \ \O(\mu^2).\]

System~\eqref{eqn:GreenNaghdi2} therefore becomes
\begin{equation}\label{eqn:GreenNaghdiMeanA}
\left\{ \begin{array}{l}
\displaystyle\partial_{ t}{\zeta} \ + \ \partial_x \Big(\frac{h_1h_2}{h_1+\gamma h_2}\bar v\Big)\ =\ 0, \\ \\
\displaystyle\partial_{ t}\Bigg( \bar v \ + \ \mu\overline{\Q}[h_1,h_2]\bar v \Bigg) \ + \ (\gamma+\delta)\partial_x{\zeta} \ + \ \frac{\epsilon}{2} \partial_x\Big(\dfrac{{h_1}^2 -\gamma {h_2}^2 }{(h_1+\gamma h_2)^2} |\bar v|^2\Big) \ = \ \mu\epsilon\partial_x\big(\overline{\R}[h_1,h_2]\bar v \big) \ ,
\end{array} 
\right. 
\end{equation}
with the following operators:
 \begin{align*} 
 \overline{\Q}[h_1,h_2]V \ &\equiv \ \frac{-1}{3h_1 h_2}\Bigg(h_1 \partial_x \Big({h_2}^3\partial_x\big(\frac{h_1\ V}{h_1+\gamma h_2} \big)\Big)\ +\ \gamma h_2\partial_x \Big( {h_1}^3\partial_x \big(\frac{h_2\ V}{h_1+\gamma h_2}\big)\Big)\Bigg), \\
 \overline{\R}[h_1,h_2]V \ &\equiv \ \frac12 \Bigg( \Big( h_2\partial_x \big( \frac{h_1\ V}{h_1+\gamma h_2} \big)\Big)^2\ -\ \gamma\Big(h_1\partial_x \big(\frac{h_2\ V}{h_1+\gamma h_2} \big)\Big)^2\Bigg)\\
 &\qquad + \frac13\frac{V}{h_1+\gamma h_2}\ \Bigg( \frac{h_1}{h_2}\partial_x\Big( {h_2}^3\partial_x\big(\frac{h_1\ V}{h_1+\gamma h_2} \big)\Big) \ - \ \gamma\frac{h_2}{h_1}\partial_x\Big({h_1}^3\partial_x \big(\frac{h_2\ V}{h_1+\gamma h_2}\big)\Big) \Bigg).
 \end{align*}
 This system is exactly our Green-Naghdi system~\eqref{eqn:GreenNaghdiMeanI}, as presented in the Introduction. One obtains Proposition~\ref{prop:ConsGreenNaghdiMeanI}, namely the consistency of the solutions of the full Euler system towards our Green-Naghdi model, in the same way as Propositions~\ref{prop:ConsGreenNaghdi1} and~\ref{prop:ConsGreenNaghdi2} above, after several technical but straightforward computations. 
 
 We prove below that system~\eqref{eqn:GreenNaghdiMeanA} is linearly well-posed, as opposed to system~\eqref{eqn:GreenNaghdi2}.

\paragraph{Linear dispersion relations.}
The linearized system from~\eqref{eqn:GreenNaghdi2} is
 \[\left\{ \begin{array}{l}
\partial_{ t}\zeta +\frac1{\gamma+\delta} \partial_x v \ + \ \mu\frac{1+\gamma\delta}{3\delta (\delta+\gamma)^2}\partial_x^3 v \ =\ 0,\\ 
\partial_{ t} v \ +\ (\gamma+\delta)\partial_x \zeta \ = \ 0,
\end{array} \right. \]
Let us look for solutions of the form $\zeta=\zeta^0e^{i(kx-\omega t)}, \ v = v^0e^{i(kx-\omega t)}$. This leads to the algebraic system
\[\left\{ \begin{array}{l}
-i\omega \zeta^0 +\frac{ik}{\gamma+\delta} v^0 \ -\ \mu i k^3 \frac{1+\gamma\delta}{3\delta(\gamma+\delta)^2} v^0 \ =\ 0,\\ 
-i\omega \b v^0\ + ik(\gamma+\delta) \zeta^0 \ = \ 0,
\end{array} \right. \]
which yields the dispersion relation (with $\omega v^0 \ = \ k(\gamma+\delta)\zeta^0$)
\[\omega^2 \ = \ k^2 \ -\ \mu k^4 \frac{1+\gamma\delta}{3\delta(\gamma+\delta)} .\]
This equation does not have any real solution $\omega(k)$ if $\mu k^2 \frac{1+\gamma\delta}{3\delta(\gamma+\delta)} > 1$, thus {\em the system~\eqref{eqn:GreenNaghdi2} is linearly ill-posed}.
\medskip

The linearized system from~\eqref{eqn:GreenNaghdiMeanA} is
 \[\left\{ \begin{array}{l}
\partial_{ t}\zeta +\frac1{\gamma+\delta} \partial_x\b v \ =\ 0,\\ 
\partial_{ t} \left(\b v\ -\ \mu \frac{1+\gamma\delta}{3\delta(\gamma+\delta)}\partial_x^2 \b v \right) + (\gamma+\delta)\partial_x \zeta \ = \ 0,
\end{array} \right. \]
Same calculations as above yield the algebraic system
\[\left\{ \begin{array}{l}
-i\omega \zeta^0 +\frac{ik}{\gamma+\delta} \b v^0 \ =\ 0,\\ 
-i\omega \left(\b v^0\ +\ \mu k^2 \frac{1+\gamma\delta}{3\delta(\gamma+\delta)} \b v^0 \right) + ik(\gamma+\delta) \zeta^0 \ = \ 0,
\end{array} \right. \]
so that $\omega\zeta^0 \ = \ \frac{k}{\gamma+\delta}\b v^0$ and the dispersion relation is
\[ \omega^2 \left(1\ +\ \mu k^2 \frac{1+\gamma\delta}{3\delta(\gamma+\delta)} \right) \ = \ k^2 .\]
This equations has always solutions: $\omega_\pm(k) \ = \ \pm \frac{k}{\sqrt{1\ +\ \mu k^2 \frac{1+\gamma\delta}{3\delta(\gamma+\delta)} }}$. 

Thus
{\em the system using layer-mean velocity variables(namely~\eqref{eqn:GreenNaghdiMeanA}, or identically~\eqref{eqn:GreenNaghdiMeanI}) is linearly well-posed.}

\section{Well-posedness of our scalar evolution equations; proof of Proposition~\ref{prop:WPI}}\label{sec:prop:WPI}
This section is dedicated to the study of the well-posedness of the Cauchy problem as well as the persistence of spatial localization (expressed by weighted Sobolev norms), for the following equation:
\begin{align}
\label{eq:u02}
(1- \mu\ \beta \partial_{x}^2)\partial_t u \ + \ \epsilon\alpha_1 u\partial_x u \ + \ \epsilon^2\ \alpha_2 u ^2 \partial_x u\ + \ \epsilon^3\ \alpha_3u^3\partial_x u\ & \\
+ \ \mu\ \nu \partial_x^3 u \ + \ \mu\epsilon\ \partial_x\big(\kappa u\partial_x^2 u+\iota (\partial_x u)^2\big) \ &= \ 0, \nn
\end{align}
where $\alpha_i$ ($i=1,2,3$), $\nu$, $\kappa$, $\iota$ are fixed parameters (possibly zero) and $\beta,\mu,\epsilon>0$.
More precisely, we prove the assertions of Proposition~\ref{prop:WPI}, recalled below
\begin{Proposition}[Well-posedness and persistence]\label{prop:WP}
Let $u^0 \in H^{s+1}$, with $s\geq s_0>3/2$. Let the parameters be such that $\beta,\mu,\epsilon>0$, and define $M>0$ such that
\[ \beta+\frac{1}{\beta}+\mu+\epsilon +|\alpha_1|+|\alpha_2|+|\alpha_3|+|\nu|+|\kappa|+|\iota| \ \leq \ M. \]
Then 
 there exists $T=C\big(\frac1{s_0-3/2},M,\big\vert u^0 \big\vert_{H^{s+1}_\mu}\big)$ and a unique $u \in C^0([0,T/\epsilon); H^{s+1}_\mu)\cap C^1([0,T/\epsilon);H^{s}_\mu)$ such that $u$ satisfies~\eqref{eq:u02} and initial condition $u\id{t=0}=u^0$.

 Moreover, $u$ satisfies the energy estimate for $0\leq t\leq T/\epsilon$:
\begin{equation}\label{eq:estCL}
\big\Vert \partial_t u \big\Vert_{L^\infty([0,T/\epsilon);H^{s}_\mu)} \ + \ \big\Vert u \big\Vert_{L^\infty([0,T/\epsilon);H^{s+1}_\mu)} \ \leq \ C(\frac1{s_0-3/2},M, \big\vert u^0 \big\vert_{H^{s+1}_\mu}) .
\end{equation}

\medskip

 Assume additionally that for fixed $n, k\in\NN$, the function $x^{j} {u^0}\in H^{s+\b s}$, with $0\leq j \leq n$ and $\b s=k+1+2(n-j)$. Then 
 there exists $T=C\big(\frac1{s_0-3/2},M,n,k,\sum_{j=0}^n\big\vert x^j {u^0}\big\vert_{H^{s+k+1+2(n-j)}_\mu}\big)$
such that for $0\leq t\leq T\times\min(1/\epsilon,1/\mu)$, one has
\[ \big\Vert x^n\partial_k \partial_t {u}\big\Vert_{L^\infty([0,t);H^{s}_\mu)} + \big\Vert x^n \partial_k {u}\big\Vert_{L^\infty([0,t);H^{s+1}_\mu)} \ \leq \ C\Big( \frac1{s_0-3/2},M,n,k,\sum_{j=0}^n\big\vert x^j {u^0}\big\vert_{H^{s+k+1+2(n-j)}_\mu}\Big) .\]
In particular, 
one has, for $0\leq t\leq T\times\min(1/\epsilon,1/\mu)$,
 \begin{equation} \label{eq:estCLweight2} 
 \big\Vert \partial_t {u}\big\Vert_{L^\infty([0,t);X^{s}_{n,\mu})} + \big\Vert {u}\big\Vert_{L^\infty([0,t);X^{s+1}_{n,\mu})} \ \leq \ C\Big(\frac1{s_0-3/2}, M,n,\big\vert u^0 \big\vert_{X^{s+1}_{n,\mu}} \Big) .
\end{equation} 
\end{Proposition}
\begin{proof} The existence and uniqueness of $u \in C^0([0,T/\epsilon); H^{s+1}_\mu)\cap C^1([0,T/\epsilon);H^{s}_\mu)$ such that $u$ satisfies~\eqref{eq:u02} and initial condition $u\id{t=0}=u^0\in H^s$ $s\geq s_0>3/2$ has been obtained in~\cite[Proposition~4]{ConstantinLannes09} (where the authors used slightly different parameters).

The proof is based on an iterative scheme, which relies heavily on the following energy estimate:
\[ \frac1{1+\beta^{-1}}\big\vert u\big\vert_{H^{s+1}_\mu} \ \leq \ \left(\big\vert u( t, \cdot )\big\vert_{H^s}^2+\beta\mu \big\vert u( t,\cdot )\big\vert_{H^{s+1}}^2\right)^{1/2} \ \equiv \ E^{s}(u)( t) \ \leq \ C(M, \big\vert u^0 \big\vert_{H^{s+1}_\mu}),\]
which is proved to be valid for $t\leq T_\epsilon\equiv \epsilon^{-1}\ C(\frac1{s_0-3/2},M,\big\vert u^0 \big\vert_{H^{s+1}_\mu})$. One proves in the same way
\[ \frac1{1+\beta^{-1}}\big\vert \partial_t u\big\vert_{H^{s}_\mu} \ \leq \ E^{s-1}(\partial_t u)( t) \ \leq \ C(\frac1{s_0-3/2},M,E^s(u)) \ \leq \ C(\frac1{s_0-3/2},M, \big\vert u^0 \big\vert_{H^{s+1}_\mu}),\]
 so that the estimate~\eqref{eq:estCL} follows.
 \medskip
 
In order to obtain the estimates in the weighted Sobolev norms, we introduce the function $v_n \ = \ x^n u$, where $u$ satisfies~\eqref{eq:u02}. We will prove that estimate~\eqref{eqn:weight}, below, holds for any $n,k\in\NN$, and for $0\leq t\leq T_{\epsilon,\mu}\equiv C(\frac1{s_0-3/2},M,n,k,\sum_{j=0}^n\big\vert x^{n-j} {u^0}\big\vert_{H^{s+k+2j+1}_\mu})\times \min(1/\epsilon,1/\mu)$ by induction on $n\in\NN$. Proposition~\ref{prop:WP} is a straightforward consequence.
 \begin{equation} \label{eqn:weight} 
 \big\Vert x^n\partial_k \partial_t {u}\big\Vert_{L^\infty([0,t);H^{s}_\mu)} + \big\Vert x^n \partial_k {u}\big\Vert_{L^\infty([0,t);H^{s+1}_\mu)} \ \leq \ C\Big( M,n,k,\sum_{j=0}^n\big\vert x^{n-j} {u^0}\big\vert_{H^{s+k+2j+1}_\mu}\Big) .
\end{equation} 

\noindent {\em The case $n=1$.} One can easily check that $v_1\equiv x\ u$ satisfies the equation
\begin{align}
\label{eq:v1}
(1- \mu\ \beta \partial_{x}^2)\partial_t v_1 \ + \ \epsilon\alpha_1 v_1\partial_x u \ + \ \epsilon^2\ \alpha_2 v_1 u \partial_x u\ + \ \epsilon^3\ \alpha_3 v_1 u^2\partial_x u\ & \\
+ \ \mu\ \nu \partial_x^3 v_1 \ + \ \mu\epsilon\ \partial_x\big(\kappa v_1\partial_x^2 u+\iota (\partial_x v_1)(\partial_x u)\big) & \ = \ R[u] ,\nn \end{align}
with
\[ R[u] \ \equiv \ - 2\mu\ \beta \partial_{x}\partial_t u \ + \ 3\mu\nu \partial_x^2 u \ + \ \mu\epsilon\ \big((\kappa +\iota)u\partial_x^2 u+2\iota (\partial_x u)^2\big) . \]

When taking the ($L^2-$)inner product of~\eqref{eq:v1} with $\Lambda^{2s}v_1$, one obtains (using that the operator $\Lambda^s$ is symmetric for the ($L^2-$)inner product, and $\partial_x$ is anti-symmetric)
\begin{align*} \frac12\frac{d}{dt}\big( E^s(v_1)\big)^2 + \big(\ \Lambda^s (\epsilon\alpha_1 v_1\partial_x u + \epsilon^2\ \alpha_2 v_1 u \partial_x u + \epsilon^3\ \alpha_3 v_1 u^2\partial_x u) \ , \ \Lambda^s v_1 \ \big) &\\
- \mu\epsilon\ \big(\ \big(\kappa v_1\partial_x^2 u+\iota (\partial_x v_1)(\partial_x u)\big) \ , \ \Lambda^s\partial_x v_1 \ \big) & \ = \ \big( \ \Lambda^s R[u] \ , \ \Lambda^s v_1 \ \big).
\end{align*}

Now, we use that $u$ is uniformly bounded through~\eqref{eq:estCL} for $t\leq T_\epsilon$.
\begin{itemize}
\item Using Cauchy-Schwarz inequality, and Moser estimates: $\big\vert fg \big\vert_{H^s}\leq C(\frac1{s_0-1/2}) \big\vert f \big\vert_{H^s}\big\vert g \big\vert_{H^s}$ for $s\geq s_0>1/2$, one has
\begin{multline*} 
\left\lvert \big(\ \Lambda^s (\epsilon\alpha_1 v_1\partial_x u \ + \ \epsilon^2\ \alpha_2 v_1 u \partial_x u\ + \ \epsilon^3\ \alpha_3 v_1 u^2\partial_x u) \ , \ \Lambda^s v_1 \ \big)\right\rvert \\
 \leq \ \big\vert v_1\big\vert_{H^s}^2 C(\frac1{s_0-1/2},\big\vert u\big\vert_{H^{s+1}} ) 
\ \leq \ C(\frac1{s_0-3/2},M, \big\vert u^0 \big\vert_{H^{s+2}_\mu})\big( E^s(v_1)\big)^2 . 
\end{multline*}
\item In the same way, and using the definition $\big( E^s(\cdot)\big)^2 \ = \ \big\vert \cdot \big\vert_{H^s}^2+\beta\mu \big\vert\cdot \big\vert_{H^{s+1}}^2 \ \geq \ \frac1{C(M)}\big\vert \cdot \big\vert_{H^{s+1}_\mu}^2$,
\begin{align*} 
\left\lvert \big(\ \Lambda^s\big(\kappa v_1\partial_x^2 u+\iota (\partial_x v_1)(\partial_x u)\big) \ , \ \Lambda^s\partial_x v_1 \ \big)\right\rvert \ &\leq \ C\big\vert v_1\big\vert_{H^{s+1}}\left(\big\vert v_1\big\vert_{H^s} \big\vert u\big\vert_{H^{s+2}} +\big\vert v_1\big\vert_{H^{s+1}} \big\vert u\big\vert_{H^{s+1}} \right) \\
&\leq \frac1\mu C(\frac1{s_0-3/2},M, \big\vert u^0 \big\vert_{H^{s+2}_\mu})\big( E^s(v_1)\big)^2 .
\end{align*}
\item Finally, one can check
\[ \left\lvert \big(\ \Lambda^s R[u] \ , \ \Lambda^s v_1 \ \big)\right\rvert \ \leq \ \mu \big\vert v_1\big\vert_{H^s}C(\frac1{s_0-1/2},\big\vert u\big\vert_{H^{s+2}}+\big\vert \partial_t u\big\vert_{H^{s+1}} ) \ \leq \ \mu C(\frac1{s_0-3/2},M, \big\vert u^0 \big\vert_{H^{s+3}_\mu}) E^s(v_1).\]
\end{itemize}
Altogether, one obtains the following differential inequality, valid for all $t\leq T_\epsilon$:
\[ \frac12\frac{d}{dt}\big( E^s(v_1)\big)^2 \ \leq \ \epsilon C(\frac1{s_0-3/2},M, \big\vert u^0 \big\vert_{H^{s+2}_\mu})\big( E^s(v_1)\big)^2 \ + \ \mu C(\frac1{s_0-3/2},M, \big\vert u^0 \big\vert_{H^{s+3}_\mu}) E^s(v_1).\]

Gronwall-Bihari’s inequality (see~\cite{TaylorI} for example) allows to deduce:
\[
E^s(v_1)(t) \ \leq \ E^s(v_1)\id{t=0} \ + \ \frac{\mu}{\epsilon} C(\frac1{s_0-3/2},M, \big\vert u^0 \big\vert_{H^{s+3}_\mu})\left(\exp(C(\frac1{s_0-3/2},M, \big\vert u^0 \big\vert_{H^{s+2}_\mu})\epsilon t)-1\right).
\]
When restricting $t$ to $t\leq T_{\epsilon,\mu}\equiv C(\frac1{s_0-3/2},M,\big\vert u^0 \big\vert_{H^{s+2}_\mu})\times \min(1/\epsilon,1/\mu)$, it follows that for any $s\geq s_0> 3/2$,
\[
 \big\vert x {u}\big\vert_{H^{s+1}_\mu} \ \leq \ \big\vert x {u^0}\big\vert_{H^{s+1}_\mu} \ + \ C(\frac1{s_0-3/2},M,\big\vert u^0 \big\vert_{H^{s+3}_\mu}) .
 \]

In order to control the time derivative of $v_1$, we take the ($L^2-$)inner product of~\eqref{eq:v1} with $\Lambda^{2s-2}(\partial_t v_1)$. Estimating each term as above, one obtains
\begin{align*} \big( E^{s-1}(\partial_t v_1)\big)^2 \ &\leq \ \epsilon \big\vert v_1\big\vert_{H^{s-1}}\big\vert \partial_t v_1\big\vert_{H^{s-1}} C(\frac1{s_0-1/2},\big\vert u\big\vert_{H^{s}} ) \ + \ \mu\nu \left\lvert \big(\ \Lambda^{s-1} \partial_x^2 v \ , \ \Lambda^{s-1}\partial_x \partial_t v_1 \ \big)\right\rvert\ \\ 
&\qquad + \mu\epsilon C\big\vert \partial_t v_1\big\vert_{H^{s}}\left(\big\vert v_1\big\vert_{H^{s-1}} \big\vert u\big\vert_{H^{s+1}} +\big\vert v_1\big\vert_{H^{s+1}} \big\vert u\big\vert_{H^{s}} \right) + \left\lvert \big( \Lambda^{s-1} R[u] , \Lambda^{s-1}\partial_t v_1 \big)\right\rvert \\
&\leq \ E^{s-1}(\partial_t v_1) \left( C(\frac1{s_0-3/2},M, \big\vert u^0 \big\vert_{H^{s+1}_\mu}) E^s(v_1)+ \mu C(\frac1{s_0-3/2},M, \big\vert u^0 \big\vert_{H^{s+2}_\mu}) \right). \end{align*}
Estimate~\eqref{eqn:weight} thus follows for $n=1$ and $k=0$.
\medskip

We know turn to $x\partial_x^k u$, for $k\in\NN$, $k\geq1$ . Note that
\[ x\partial_x^k u \ = \ \partial_x^k (xu) \ - \ k\partial_x^{k-1} u,\]
so that
\[ \big\vert x\partial_x^k u \big\vert_{H^{s+1}_\mu} \ \leq \ \big\vert xu \big\vert_{H^{s+k+1}_\mu} \ + \ k \big\vert u \big\vert_{H^{s+k}_\mu}\]
One deduces, for $ 0\leq t\leq T(M)\min(1/\epsilon,1/\mu)$,
 \[
 \big\vert x \partial_k {u}\big\vert_{H^{s+1}_\mu} \ \leq \ \big\vert x {u}\big\vert_{H^{s+k+1}_\mu} \ + \ C_1(\frac1{s_0-3/2}.M,\big\vert u^0 \big\vert_{H^{s+k+3}_\mu}) 
 \]
In the same way, one has $x\partial_x^k\partial_t u \ = \ \partial_x^k (x\partial_t u) \ - \ k\partial_x^{k-1}\partial_t u,$
so that estimate~\eqref{eqn:weight} hold for $n=1$ and any $k\in\NN$.
\bigskip

\noindent {\em The case $n\geq2$.} Let us assume that~\eqref{eqn:weight} holds for any $k\in\NN$ and $n\leq m-1$ (with $m\geq2$). One can easily check that $v_m$ satisfies the equation
\begin{align}
\label{eq:vn}
(1- \mu\ \beta \partial_{x}^2)\partial_t v_m \ + \ \epsilon\alpha_1 v_m\partial_x u \ + \ \epsilon^2\ \alpha_2 v_m u \partial_x u\ + \ \epsilon^3\ \alpha_3 v_m u^2\partial_x u\ & \\
+ \ \mu\ \nu \partial_x^3 v_m \ + \ \mu\epsilon\ \partial_x\big(\kappa v_m\partial_x^2 u+\iota (\partial_x v_m)(\partial_x u)\big) & \ = \ R_m[u] ,\nn \end{align}
with
\begin{align*}R_m[u] &\equiv - 2\mu\ \beta m x^{m-1}\partial_{x}\partial_t u- 2\mu\ \beta m (m-1)x^{m-2}\partial_t u + 3\mu\nu mx^{m-1}\partial_x^2 u + 3\mu\nu m(m-1)x^{m-2}\partial_x u\\
&\ + \mu\nu m(m-1)(m-2)x^{m-3} u + \mu\epsilon m x^{m-1} \big((\kappa +\iota)u\partial_x^2 u+2\iota (\partial_x u)^2\big) +\mu\epsilon m(m-1)x^{m-2}\iota u\partial_x u. 
\end{align*}
Taking the ($L^2-$)inner product of~\eqref{eq:vn} with $\Lambda^{2s}(v_m)$, one obtains as above
\[ \frac12\frac{d}{dt}\big( E^s(v_m)\big)^2 \ \leq \ \epsilon C(\frac1{s_0-3/2},M, \big\vert u^0 \big\vert_{H^{s+2}_\mu})\big( E^s(v_m)\big)^2 \ + \ \left\lvert \big(\ \Lambda^s R_m[u] \ , \ \Lambda^s v_m \ \big)\right\rvert .\]
Now, one can easily check
\begin{align*} \left\lvert \big(\ \Lambda^s R_m[u] \ , \ \Lambda^s v_m \ \big)\right\rvert \ & \leq \ \mu C(\frac1{s_0-1/2},m) \big\vert v_m\big\vert_{H^s}C\Big(\big\vert x^{m-1}\partial_x \partial_t u\big\vert_{H^{s}}+\big\vert x^{m-2} \partial_t u\big\vert_{H^{s}} \\
& \qquad +\big\vert x^{m-1} \partial_x^2 u\big\vert_{H^{s}}+\big\vert x^{m-2} \partial_x u\big\vert_{H^{s}} +\big\vert x^{m-3} u\big\vert_{H^{s}} \Big) \\ 
&\leq \ \mu E^s(v_m)C\Big(\frac1{s_0-3/2}, M,m,\sum_{j=0}^{m-1}\big\vert x^{m-1-j} {u^0}\big\vert_{H^{s+2j+1}_\mu}\Big) .\end{align*}

As above, Gronwall-Bihari’s inequality allows to deduce
\[
E^s(v_m)(t) \ \leq \ E^s(v_m)\id{t=0} \ + \ C\Big( \frac1{s_0-3/2}, M,m,\sum_{j=0}^{m-1}\big\vert x^{m-1-j} {u^0}\big\vert_{H^{s+2j+1}_\mu}\Big) ,
\]
for any $0\leq t\leq T_{\epsilon,\mu}\equiv C\big(\frac1{s_0-3/2}, M,m,\sum_{j=0}^{m-1}\big\vert x^{m-1-j} {u^0}\big\vert_{H^{s+2j+1}_\mu}\big)\times\min(1/\epsilon,1/\mu)$ and for any $s\geq s_0>3/2$. Therefore,
\begin{equation}\label{eq:xmu}
 \big\vert x^m {u}\big\vert_{H^{s+1}_\mu} \ \leq \ C\Big( \frac1{s_0-3/2},M,m,\sum_{j=0}^m\big\vert x^{m-j} {u^0}\big\vert_{H^{s+2j+1}_\mu}\Big) \quad \text{ for } \quad 0\leq t\leq T_{\epsilon,\mu}.
\end{equation}
The similar estimate on the time derivative of $v_m$, is obtained as above by taking the ($L^2-$)inner product of~\eqref{eq:vn} with $\Lambda^{2s-2}(\partial_t v_m)$. Estimate~\eqref{eqn:weight} follows for any $n=m$ and $k=0$.

\bigskip

We know turn to $x^m\partial_x^k u$ . Note that for any $k\in\NN$, $k\geq1$,
\[ x^m\partial_x^k u \ = \ \partial_x^k (x^m u) \ - \ \sum_{j=0}^{k-1}\ {k \choose j} (\partial_x^{k-j} x^m) (\partial_x^j u )\,\]
so that
\[ \big\vert x\partial_x^k u \big\vert_{H^{s+1}_\mu} \ \leq \ \big\vert x^m {u}\big\vert_{H^{s+k+1}_\mu} \ + \ C(k,m)\sum_{j=\min(0,k-m)}^{k-1}\big\vert x^{m-k+j} \partial_x^ju \big\vert_{H^{s+1}_\mu} .\]
Using~\eqref{eq:xmu}, one deduces by induction on $k$ that
 \begin{align*}
 \big\vert x^m\partial_k {u}\big\vert_{H^{s+1}_\mu} \ &\leq \ C\Big(\frac1{s_0-3/2}, M,m,\sum_{j=0}^m\big\vert x^{m-j} {u^0}\big\vert_{H^{s+k+2j+1}_\mu}\Big) \\
 & \qquad + \ \sum_{j=\min(0,k-m)}^{k-1}C\Big(\frac1{s_0-3/2}, M,m-k+j,j,\sum_{i=0}^{m-k+j}\big\vert x^{m-k+j-i} {u^0}\big\vert_{H^{s+j+2i+1}_\mu}\Big) \\
 &\leq \ C\Big( \frac1{s_0-3/2},M,m,k,\sum_{j=0}^m\big\vert x^{m-j} {u^0}\big\vert_{H^{s+k+2j+1}_\mu}\Big) \quad \text{ for } \quad 0\leq t\leq T_{\epsilon,\mu}.
 \end{align*}
In the same way, one estimates $x^m\partial_x^k\partial_t u$ by induction, using Leibniz rule of differentiation. 

Estimate~\eqref{eqn:weight} therefore holds for $n=m$ and any $k\in\NN$.
\medskip

By induction, we proved that~\eqref{eqn:weight} holds for any $k\in\NN$ and $n\in\NN$. Estimate~\eqref{eq:estCLweight2} follows as a direct consequence (using the case $k=0$), and Proposition~\ref{prop:WP} is proved.
\end{proof}

 \section{Unidirectional propagation}\label{sec:unidirectional}
In this section, we show that if one chooses carefully the initial perturbation (deformation of the interface as well as shear layer-mean velocity),
then the flow is unidirectional, in the sense that one can construct an extremely precise approximate solution, driven by a simple unidirectional scalar equation. This study follows the strategy developed for the water-wave problem in~\cite{ConstantinLannes09,Johnson02}.

The precise result, which has been displayed in the Introduction (Proposition~\ref{prop:unidirI}), is recalled below, followed by its proof and a brief discussion.
\begin{Proposition} \label{prop:unidir}
Set $\lambda,\theta\in\RR$, and $\zeta^0 \in H^{s+5}$ with $s\geq s_0>3/2$. For $(\epsilon,\mu,\delta,\gamma)=\p\in\P$, as defined in~\eqref{eqn:defRegimeFNL}, denote
$(\zeta^{\p})_{\p \in\P }$ the unique solution of the equation
\begin{align}\label{eq:CLuni}
\partial_t \zeta + \partial_x \zeta + \epsilon\alpha_1\zeta\partial_x \zeta + \epsilon^2 \alpha_2 \zeta^2\partial_x \zeta +\epsilon^3 \alpha_3 \zeta^3\partial_x \zeta+\mu\nu_x^{\theta,\lambda}\partial_x^3\zeta-\mu\nu_t^{\theta,\lambda}\partial_x^2\partial_t\zeta & \\
+\mu\epsilon\partial_x\left(\kappa_1^{\theta,\lambda} \zeta\partial_x^2\zeta +\kappa_2^{\theta} (\partial_x \zeta)^2\right) & =0\ , \nn 
\end{align}
with parameters
\[ \begin{array}{rlrl}
\alpha_1 &= \ \frac32\frac{\delta^2-\gamma}{\gamma+\delta}, 
\qquad 
\alpha_2 = \ \frac{21(\delta^2-\gamma)^2}{8(\gamma+\delta)^2}-3\frac{\delta^3+\gamma}{\gamma+\delta},
\qquad & \alpha_3 &= \ \frac{71(\delta^2-\gamma)^3}{16(\gamma+\delta)^3} -\frac{37(\delta^2-\gamma)(\delta^3+\gamma)}{4(\gamma+\delta)^2}+\frac{5(\delta^4-\gamma)}{\gamma+\delta}, 
\\
\nu_x^{\theta,\lambda} &= \ (1-\theta-\lambda)\frac{1+\gamma\delta}{6\delta(\gamma+\delta)} , &

 \nu_t^{\theta,\lambda} &= \ (\theta+\lambda)\frac{1+\gamma\delta}{6\delta(\gamma+\delta)} ,
\\
\kappa_1^{\theta,\lambda} &= \ \frac{(14-6(\theta+\lambda))(\delta^2-\gamma)(1+\gamma\delta)}{24\delta(\gamma+\delta)^2}-\frac{1-\gamma}{6(\gamma+\delta)},
&
\kappa_2^{\theta} &= \ \frac{(17-12\theta)(\delta^2-\gamma)(1+\gamma\delta)}{48\delta(\gamma+\delta)^2}-\frac{1-\gamma}{12(\gamma+\delta)}.
\end{array}
\]
For given $M_{s+5},h>0$, assume that there exists $T_{s+5}>0$ such that
\[ T_{s+5} \ = \ \max\big( \ T\geq0\quad \text{such that}\quad \big\Vert \zeta^{\p} \big\Vert_{L^\infty([0,T);H^{s+5})}\ \leq \ M_{s+5}\ \big) \ ,\]
and for any $(t,x)\in [0,T_{s+5})\times\RR$,
\[ h_1(t,x)=1-\epsilon\zeta^\p(t,x)>h>0, \quad h_2(t,x)=\frac1\delta+\epsilon\zeta^\p(t,x)>h>0.\]
Then define
$v^\p$ as $ v^\p=\frac{h_1+\gamma h_2}{h_1h_2}\underline{v}[\zeta^\p]$, with
\begin{equation} \label{eq:ztov}
\underline{v}[\zeta] \ = \ \zeta + \epsilon\frac{\alpha_1}2\zeta^2 + \epsilon^2 \frac{\alpha_2}{3} \zeta^3 +\epsilon^3 \frac{\alpha_3}4 \zeta^4 +\mu\nu\partial_x^2\zeta+\mu\epsilon\left(\kappa_1 \zeta\partial_x^2\zeta +\kappa_2 (\partial_x \zeta)^2\right),
\end{equation}
where parameters $\alpha_1,\alpha_2,\alpha_3$ are as above, and $\nu=\nu_x^{0,0},\kappa_1=\kappa_1^{0,0},\kappa_2=\kappa_2^{0}$.

Then $(\zeta^\p,v^\p)$ is consistent with Green-Naghdi equations~\eqref{eqn:GreenNaghdiMeanI}, of order $s$ and on $[0,T_{s+5})$, with precision $\O(\eps)$, with
\[ \eps \ = \ C(M_{s+5},\frac1{s_0-3/2},{h}^{-1},\frac{1}{\delta_{\text{min}}},\delta_{\text{max}},\epsilon_{\text{max}},\mu_{\text{max}},|\lambda|,|\theta|)\ \times \ \max(\epsilon^4,\mu^2)\ .\]
\end{Proposition}

\begin{proof} In order to simplify the calculations, we use the Green-Naghdi system~\eqref{eqn:GreenNaghdiMeanI} expressed using the variables $(\zeta,\underline{v})$
 where we define $\underline{v} \ = \ \frac{h_1h_2}{h_1+\gamma h_2}\bar v$. The system reads
 \begin{equation}\label{eqn:GreenNaghdiMean3}
\left\{ \begin{array}{l}
\displaystyle\partial_{ t}{\zeta} \ + \ \partial_x \underline{v} \ =\ 0, \\ \\
\displaystyle\partial_{ t}\Bigg( \frac{h_1+\gamma h_2}{h_1h_2}\underline{v} + \mu\underline{\Q}[h_1,h_2]\underline{v} \Bigg) \ + \ (\gamma+\delta)\partial_x{\zeta} + \frac{\epsilon}{2} \partial_x\Big(\frac{{h_1}^2 -\gamma {h_2}^2}{(h_1h_2)^2} |\underline{v}|^2\Big) = \mu\epsilon\partial_x\big(\underline{\R}[h_1,h_2]\underline{v} \big) ,
\end{array} 
\right. 
\end{equation}
with the following operators:
 \begin{align*} 
 \underline{\Q}[h_1,h_2]V \ &\equiv \ \frac{-1}{3h_1 h_2}\Bigg(h_1 \partial_x \Big({h_2}^3\partial_x\big( \frac{V}{h_2} \big)\Big)\ +\ \gamma h_2\partial_x \Big( {h_1}^3\partial_x \big( \frac{V}{h_1}\big)\Big)\Bigg), \\
 \underline{\R}[h_1,h_2]V \ &\equiv \ \frac12 \Bigg( \Big( h_2\partial_x \big( \frac{V}{h_2} \big)\Big)^2 - \gamma\Big(h_1\partial_x \big(\frac{V}{h_1} \big)\Big)^2\Bigg)\\
 &\qquad + \frac13 V\ \Bigg( \frac{h_1}{h_2}\partial_x\Big( {h_2}^3\partial_x\big(\frac{V}{h_2} \big)\Big) - \gamma\frac{h_2}{h_1}\partial_x\Big({h_1}^3\partial_x \big(\frac{V}{h_1}\big)\Big) \Bigg).
 \end{align*}
 
Using this system simplifies considerably the analysis. Indeed, it is clear that if $\zeta$ is to satisfy the following scalar evolution equation, 
\begin{align}
\partial_t \zeta + \partial_x \zeta + \epsilon\alpha_1\zeta\partial_x \zeta + \epsilon^2 \alpha_2 \zeta^2\partial_x \zeta +\epsilon^3 \alpha_3 \zeta^3\partial_x \zeta \qquad & \nn \\
+\mu\nu\partial_x^3\zeta +\mu\epsilon\partial_x\left(\kappa_1 \zeta\partial_x^2\zeta +\kappa_2 (\partial_x \zeta)^2\right)&=0, \label{eq:1} \end{align}
then $\underline{v}$ shall satisfy (using the first equation of~\eqref{eqn:GreenNaghdiMean3}, and the fact that the system is at rest at infinity: $\zeta,\underline{v}\to 0$ when $x\to\pm\infty$)
\begin{align}
&\underline{v} \ = \ \zeta + \epsilon\frac{\alpha_1}2\zeta^2 + \epsilon^2 \frac{\alpha_2}{3} \zeta^3 +\epsilon^3 \frac{\alpha_3}4 \zeta^4 +\mu\nu_x\partial_x^2\zeta+\mu\epsilon\left(\kappa_1 \zeta\partial_x^2\zeta +\kappa_2 (\partial_x \zeta)^2\right). \label{eq:2}\end{align}
Now we will show that one can choose coefficients $\alpha_1,\alpha_2$, {\em etc.} such that the second equation of~\eqref{eqn:GreenNaghdiMean3} is satisfied up to a small remainder.

Indeed, when plugging~\eqref{eq:1} and~\eqref{eq:2}, and expanding in terms of $\epsilon$ and $\mu$, one obtains
\begin{align}
& \epsilon\big( 3( {\delta}^{2}- \gamma)-2 ( \delta+\gamma ) \alpha_1 \big) \zeta \partial_x\zeta \nn \\
&\quad - \epsilon^2 \big(6( \delta^3+\gamma)-5
 ({\delta}^{2}-\gamma)\alpha_1+ ( {\alpha_1}^2+2\alpha_2 ) (\gamma+\delta) \big) \zeta^{2}\partial_x\zeta \nn \\
&\quad+ \epsilon^3 \big( 10( {\delta}^{4}-\gamma)-9( \delta^3+\gamma)\alpha_1+ (2 {\alpha_1}^{2}+14/3\ \alpha_2 ) (\delta^2-\gamma)-2 (\alpha_3+\alpha_1\alpha_2 ) (\delta+\gamma) \big) 
 \zeta^3\partial_x\zeta \nn \\
&\quad+\mu \big( \frac{1+\gamma\delta}{3\delta} -2\nu_x (\delta+\gamma) \big) \partial_x^3\zeta \nn\\
&\quad +	 \mu\epsilon \big( 4(\delta^2-\gamma)\nu_x- 2 ( \alpha_1 \nu + \kappa_1 )( \delta+\gamma) -\frac{1-\gamma}{3} +\frac{2(1+\gamma\delta)}{3\delta} \alpha_1 \big) \zeta
 \partial_x^3\zeta \nn \\
&\quad+
 \mu\epsilon ( 
 2 ({\delta}^{2}-\gamma)\nu_x- ( 3\alpha_1\nu_x +2 \kappa_1+4 \kappa_2 ) (\delta+\gamma)
 -\frac{2(1-\gamma)}{3}+\frac{2(1+\gamma\delta)}{\delta} \alpha_1 ) \partial_x\zeta\partial_x^2
\zeta \nn \\
& \qquad = \ \R \label{eq:remainder}
\end{align}
where the remainder $\R$ can be estimated, provided $h_1\geq h>0$ and $h_2\geq h>0$, as
\[ \big\vert \R \big\vert_{H^s} \ \leq \ C(\big\vert \zeta \big\vert_{H^{s+5}},\frac1{s_0-3/2},h^{-1},\delta_{\text{min}}^{-1},\delta_{\text{max}},\epsilon_{\text{max}},\mu_{\text{max}})\times\max(\epsilon^4,\mu^2). \]

The left hand side of~\eqref{eq:remainder} vanishes when we set
\[
\begin{array}{c}
 \alpha_1 \ = \ \frac32\frac{\delta^2-\gamma}{\gamma+\delta}, 
\qquad
\alpha_2 \ = \ \frac{5(\delta^2-\gamma)\alpha_1-6(\delta^3+\gamma)}{2(\gamma+\delta)}-\frac{\alpha_1^2}{2}\qquad \alpha_3= \frac{10(\delta^4-\gamma)-9(\delta^3+\gamma)\alpha_1+(\frac{14}3 \alpha_2+2\alpha_1^2)(\delta^2-\gamma)}{2(\gamma+\delta)}-\alpha_1\alpha_2, 
\\
 \nu=\frac{1+\gamma\delta}{6\delta(\gamma+\delta)} \qquad \kappa_1 \ = \ \frac{4(\delta^2-\gamma)\nu_x-\frac{1-\gamma}{3}+\alpha_1\frac{2(1+\gamma\delta)}{3\delta}}{2(\gamma+\delta)}- \alpha_1\nu,
\qquad 
\kappa_2 \ = \ \frac{(\delta^2-\gamma)\nu_x-\frac{1-\gamma}{3}+\alpha_1\frac{1+\gamma\delta}{\delta}}{2(\gamma+\delta)}-\frac{3\alpha_1\nu}{4}-\frac{\kappa_1}{2}.
\end{array}
\]
Such a choice corresponds to parameters of Proposition~\ref{prop:unidir}, with $\theta=\lambda=0$.
 \bigskip
 
The cases $\theta\neq0$ and $\lambda\neq0$ are obtained using the so-called BBM trick and near-identity change of variables, as precisely described in Section~\ref{ssec:formal}. We detail the calculations below, using for simplicity the notation $\O_{\b s}(\eps)$ for any term bounded by $\eps C(\big\vert \zeta \big\vert_{H^{s+\b s}})$.
\medskip

\noindent {\em BBM trick}.
We make use of the first order approximation in~\eqref{eq:1}:
\[ \partial_t \zeta \ + \ \partial_x \zeta \ + \ \epsilon\alpha_1 \zeta\partial_x \zeta \ = \ \O_3(\max(\mu,\epsilon^2)),\]
so that one has, for any $\theta\in\RR$,
\[ \partial_x \zeta \ = \ (1-\theta)\partial_x \zeta-\theta\big(\partial_t\zeta+ \epsilon\alpha_1 \zeta\partial_x \zeta\big) \ + \ \O_3(\max(\mu,\epsilon^2)).\] 
Plugging this identity into~\eqref{eq:1} yields
\begin{align}
\partial_t \zeta + \partial_x \zeta + \epsilon\alpha_1\zeta\partial_x \zeta + \epsilon^2 \alpha_2 \zeta^2\partial_x \zeta +\epsilon^3 \alpha_3 \zeta^3\partial_x \zeta 
+\mu(1-\theta)\nu\partial_x^3\zeta-\mu\theta\nu\partial_x^2\partial_t\zeta & \nn \\ +\mu\epsilon\partial_x\left((\kappa_1-\theta\nu\alpha_1 )\zeta\partial_x^2\zeta +(\kappa_2-\theta\nu\alpha_1) (\partial_x \zeta)^2\right)&=\O_5(\max(\mu^2,\mu\epsilon^2)). \label{eq:1theta} \end{align}
Conversely, $\zeta_\theta$, a solution of~\eqref{eq:1theta} (with zero on the right hand side), satisfies~\eqref{eq:1} with a reaminder bounded by $\O_5(\max(\mu^2,\mu\epsilon^2))$. One can easily check that, defining $\underline{v}_\theta$ as a function of $\zeta_\theta$ through~\eqref{eq:2}, $(\zeta_\theta,\underline{v}_\theta)$ satisfies~\eqref{eq:remainder} up to a remainder $\R_\theta=\O_5(\max(\mu^2,\mu\epsilon^2))$. Proposition~\ref{prop:unidir} is now proved for $\theta\in\RR$ and $\lambda=0$.
 \medskip
 
\noindent {\em Near identity change of variable}. Let us consider 
\[ \zeta_{\theta,\lambda} \ \equiv \ \zeta_\theta - \mu\nu \lambda \partial_x^2 \zeta_\theta\ , \]
where we recall $\nu=\frac{1+\gamma\delta}{6\delta(\gamma+\delta)}$, and with $\zeta_\theta$ satisfying~\eqref{eq:1theta} (with zero on the right hand side). Then $\zeta_{\theta,\lambda}$ satisfies
\begin{align}
\partial_t \zeta + \partial_x \zeta + \epsilon\alpha_1\zeta\partial_x \zeta + \epsilon^2 \alpha_2 \zeta^2\partial_x \zeta +\epsilon^3 \alpha_3 \zeta^3\partial_x \zeta 
+\mu((1-\theta)\nu-\lambda\nu)\partial_x^3\zeta-\mu(\theta\nu+\lambda\nu)\partial_x^2\partial_t\zeta & \nn \\ +\mu\epsilon\partial_x\left((\kappa_1-\theta\nu\alpha_1-\lambda\nu\alpha_1) \zeta\partial_x^2\zeta +(\kappa_2-\theta\nu\alpha_1) (\partial_x \zeta)^2\right)=\O_5(\max(\mu^2,\mu\epsilon^2)),& \label{eq:1thetalambda} \end{align}
Again, it is now straightforward though technical to check that, denoting $\zeta_{\theta,\lambda} $ the solution of~\eqref{eq:1thetalambda} and defining $\underline{v}_{\theta,\lambda}$ as a function of $\zeta_{\theta,\lambda} $ through~\eqref{eq:2}, then $(\zeta_{\theta,\lambda} ,\underline{v}_{\theta,\lambda} )$ satisfies~\eqref{eq:remainder} up to a remainder $\R_{\theta,\lambda}=\O_5(\max(\mu^2,\mu\epsilon^2))$. Proposition~\ref{prop:unidir} is now proved for any $\theta,\lambda\in\RR$.
\end{proof}
\medskip

 \paragraph{Discussion.} 
Note that the accuracy of the unidirectional approximate solution, described in Proposition~\ref{prop:unidir} is considerably better than the one of the decoupled model; see Proposition~\ref{prop:decompositionI}. As a matter of fact, the accuracy of the unidirectional approximation is as good as the solution of the coupled Green-Naghdi model, in the Camassa-Holm regime $\epsilon^2=\O(\mu)$. More precisely, provided Hypothesis~\ref{conj:stab} is valid, one obtains the following result (to compare with Corollary~\ref{cor:convergence}, below).
 \begin{Corollary}[Convergence of unidirectional approximation]\label{cor:convergenceuni}For $(\epsilon,\mu,\delta,\gamma)=\p\in\P$, as defined in~\eqref{eqn:defRegimeFNL}, let $U_{\text{GN}}^\p$ be a solution of Green-Naghdi equations~\eqref{eqn:GreenNaghdiMeanI} such that the family $(U_{\text{GN}}^\p)$ is uniformly bounded on $H^s$, $s$ sufficiently big, over time interval $[0,T_{\text{GN}}]$, and with initial data satisfying~\eqref{eq:ztov}. Assume that hypotheses of Proposition~\ref{prop:unidir} hold, and denote $U_{\text{CL}}^\p$ the unidirectional approximation defined therein. Then if Hypothesis~\ref{conj:stab} is true, one has for any $t\leq \min(T_{\text{GN}},T_{s+5})$, 
\[ \big\Vert U_{\text{CL}}-U_S\big\Vert_{L^\infty([0,t];H^s)} \ \leq \ C\ \max(\epsilon^4,\mu^2)\ t ,\]
with $C=C\big(\big\Vert U_{\text{GN}}^\p\big\Vert_{L^\infty([0,T];H^s)} ,M_{s+5},h^{-1},\delta_{\text{min}}^{-1},\delta_{\text{max}},\epsilon_{\text{max}},\mu_{\text{max}},|\lambda|,|\theta| \big)$.
\end{Corollary} 
Such a result is supported by numerical simulations. In Figures~\ref{fig:Uni201} and~\ref{fig:Uni102}, we compute the decoupled Constantin-Lannes approximation of Definition~\ref{def:CL}, as well as the unidirectional approximation described in Proposition~\ref{prop:unidir}, and compare them with the the solution of the Green-Naghdi system~\eqref{eqn:GreenNaghdiMeanI}, in the Camassa-Holm regime $\epsilon^2=\mu$ and for a unidirectional initial data ({\em i.e.} such that~\eqref{eq:ztov} is satisfied at $t=0$). Figure~\ref{fig:Uni201} deals with the case of a critical ratio $\delta^2-\gamma$, whereas the ratio is non-critical in Figure~\ref{fig:Uni102}.

 Each time, we represent the difference between the Green-Naghdi model and scalar (unidirectional and decoupled) approximations, with respect to time and for $\epsilon=0.1,0.05,0.035$. Values at times $t=10$ and $1/\epsilon$ are marked. In the two right-hand-side panels, we plot the difference in a log-log scale for more values of $\epsilon$ (the markers reveal the positions which have been simulated), at aforementioned times. The pink triangles express the convergence rate. The bottom panel reproduces the difference with respect to the space variable, at final time $t=1/\epsilon$.
\medskip

 \begin{figure}[!p] 
 \subfigure[behavior of the error with respect to time]{
\includegraphics[width=0.5\textwidth]{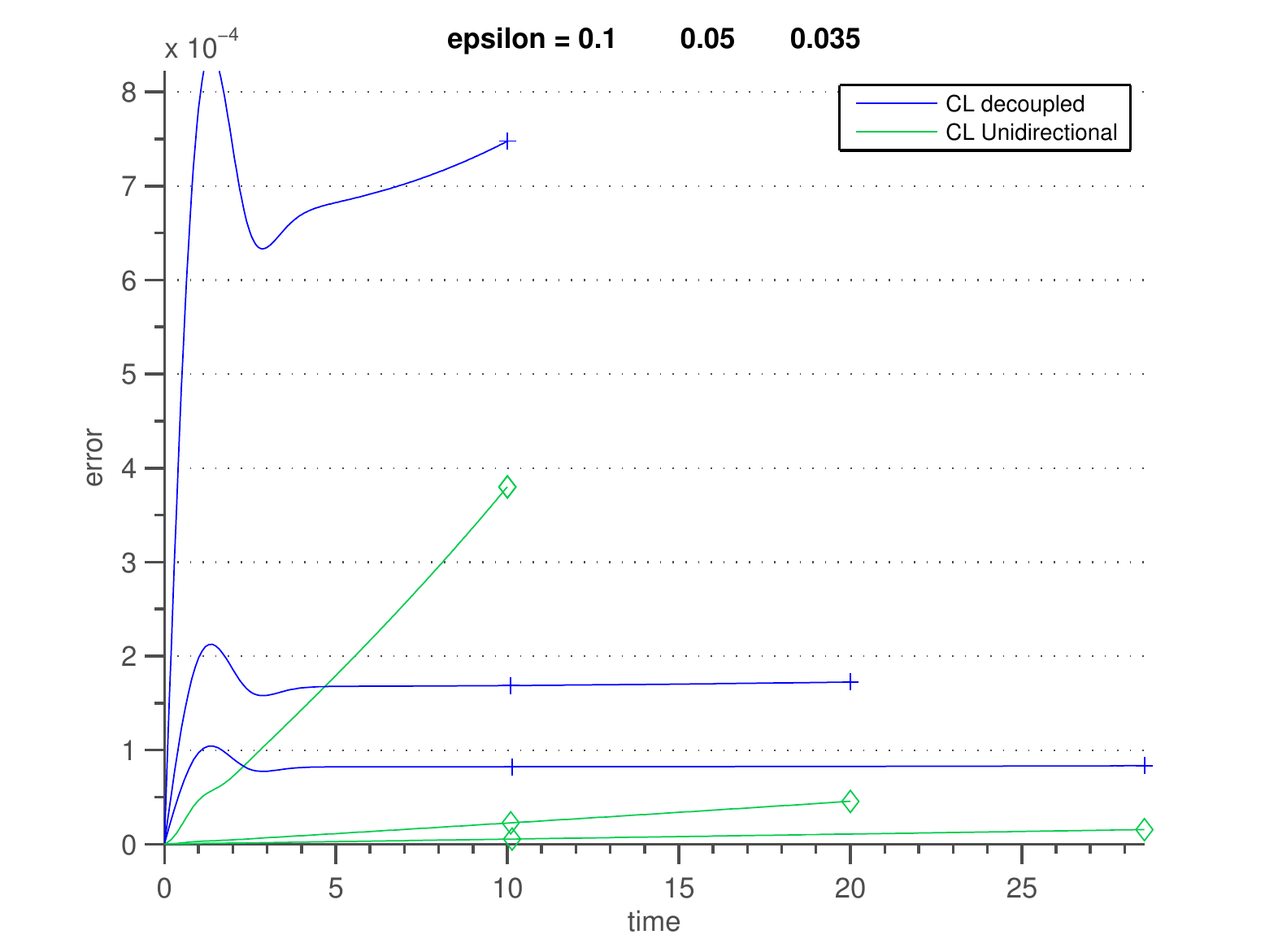}
\label{fig:Uni201a}
}\hspace{-1cm}
 \subfigure[behavior of the error with respect to $\epsilon=\sqrt\mu$]{
\includegraphics[width=0.5\textwidth]{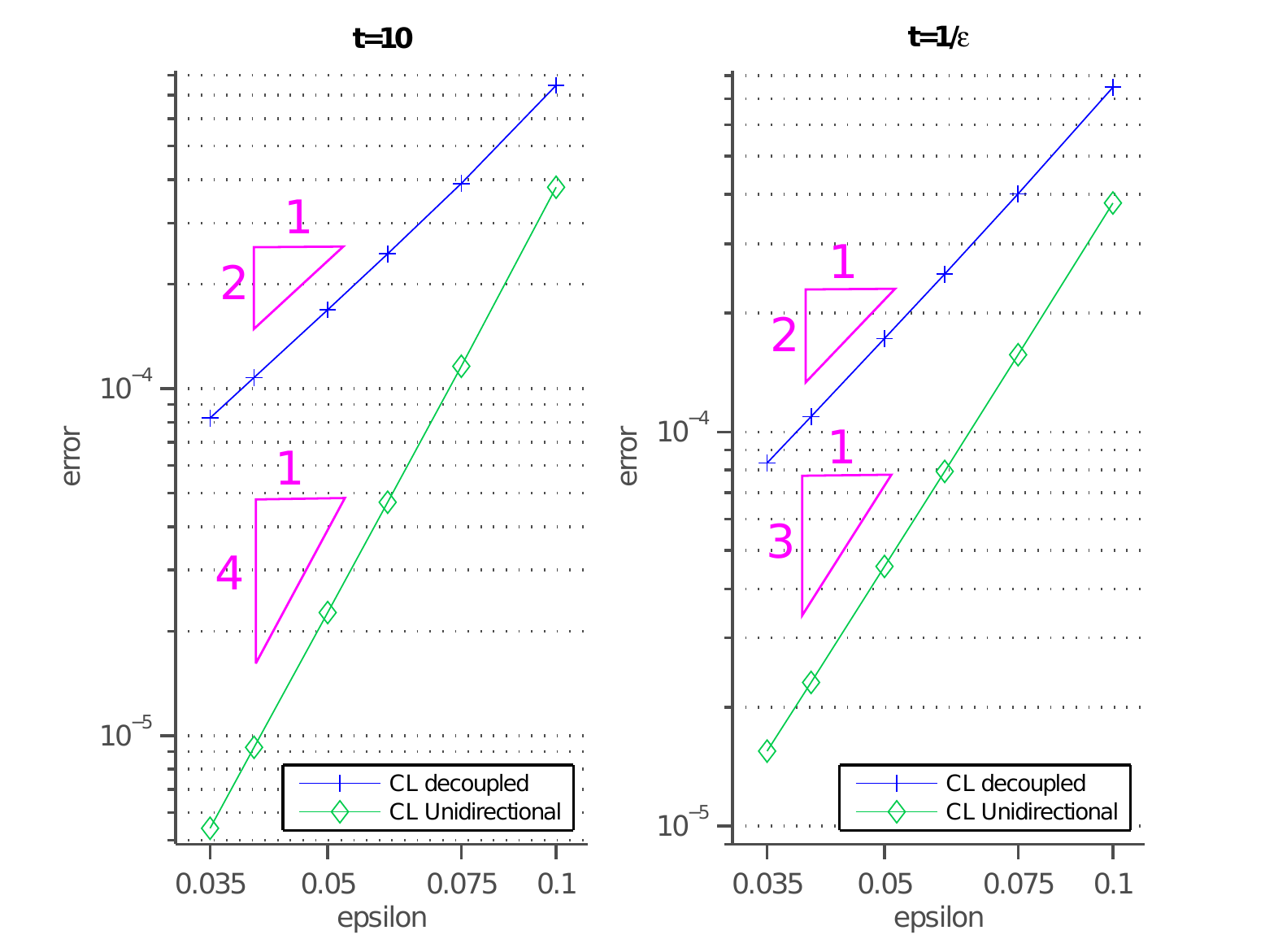}
\label{fig:Uni201b}
}\begin{center}
 \subfigure[Error at time $t=20$ for $\epsilon=0.05$]{
\includegraphics[width=0.8\textwidth]{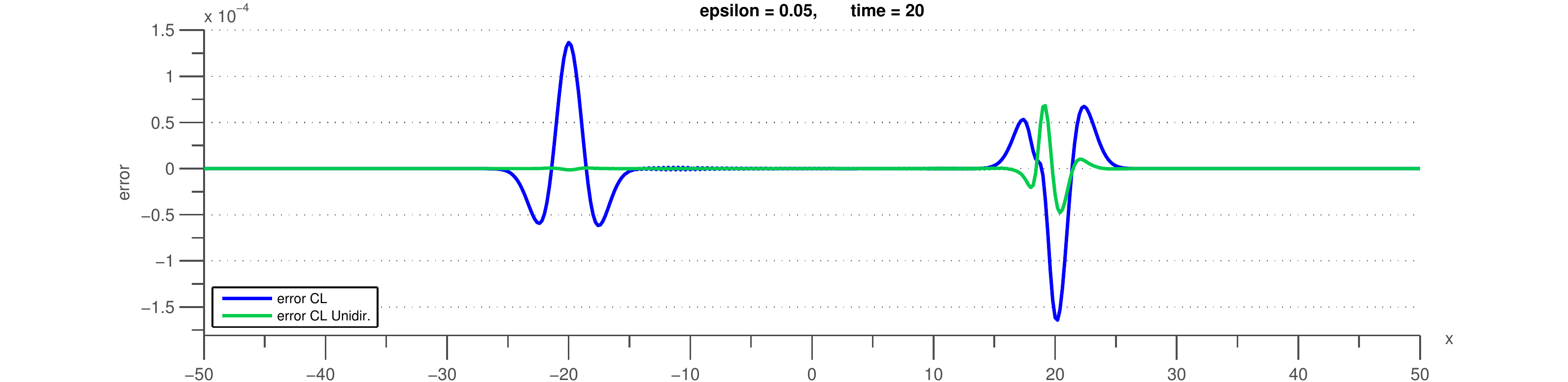}
\label{fig:Uni201c}
}\end{center}
\caption{Unidirectional, localized initial data, critical ratio, Camassa-Holm regime}
\label{fig:Uni201}
\end{figure}
\begin{figure}[!hp] 
 \subfigure[behavior of the error with respect to time]{
\includegraphics[width=0.5\textwidth]{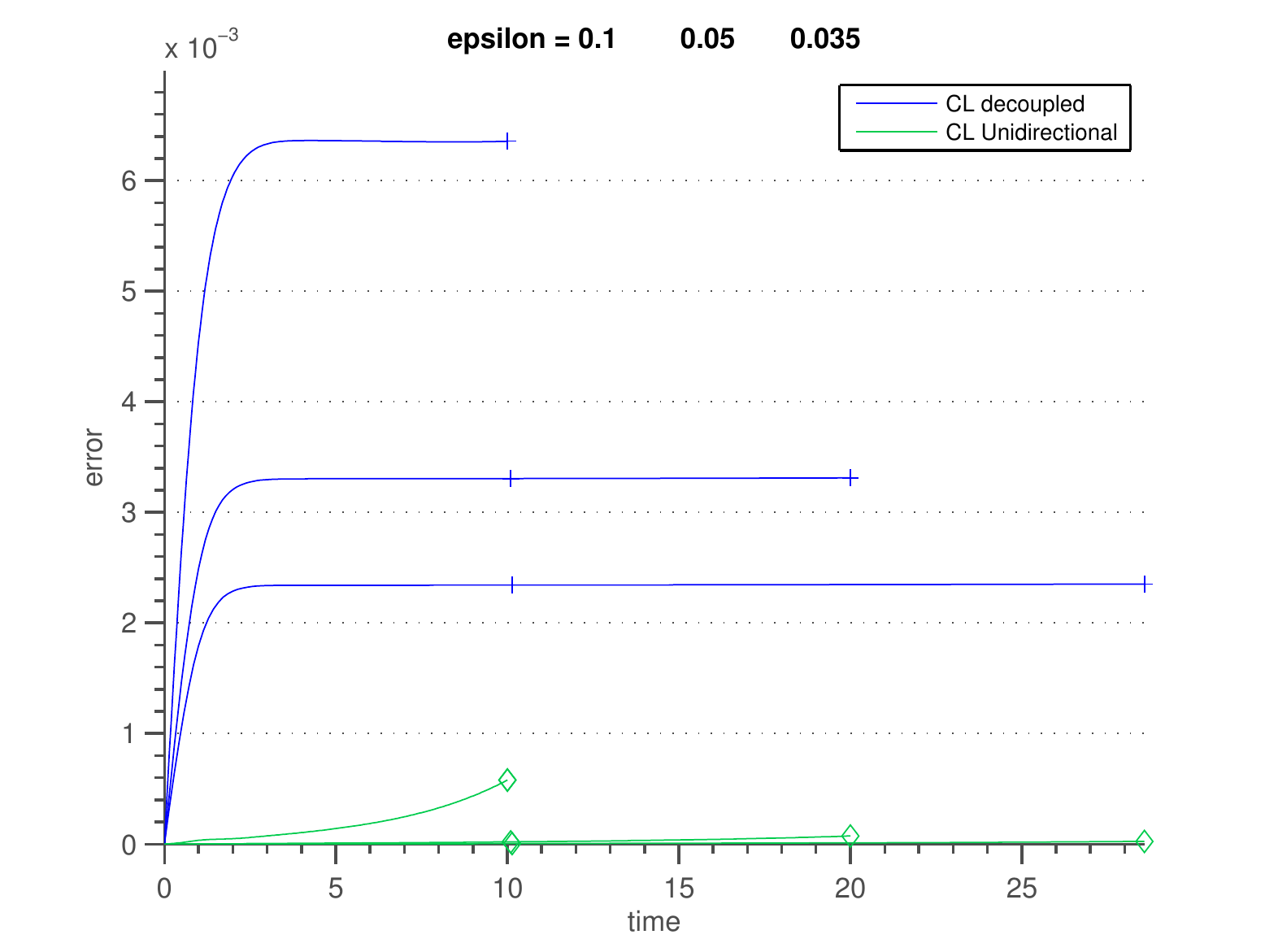}
\label{fig:Uni102a}
}\hspace{-1cm}
 \subfigure[behavior of the error with respect to $\epsilon=\sqrt\mu$]{
\includegraphics[width=0.5\textwidth]{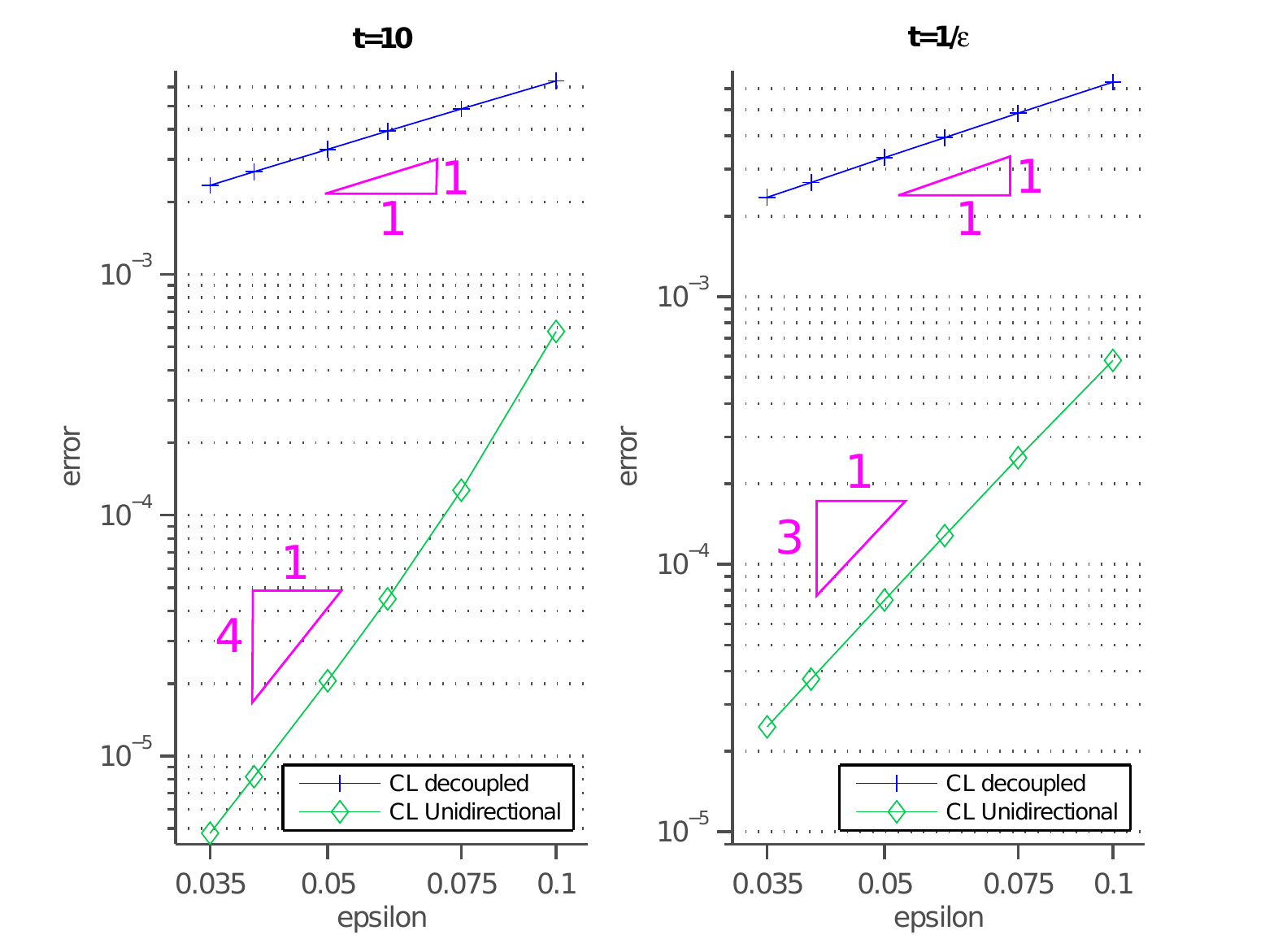}
\label{fig:Uni102b}
}\begin{center}
 \subfigure[Error at time $t=20$ for $\epsilon=0.05$]{
\includegraphics[width=0.8\textwidth]{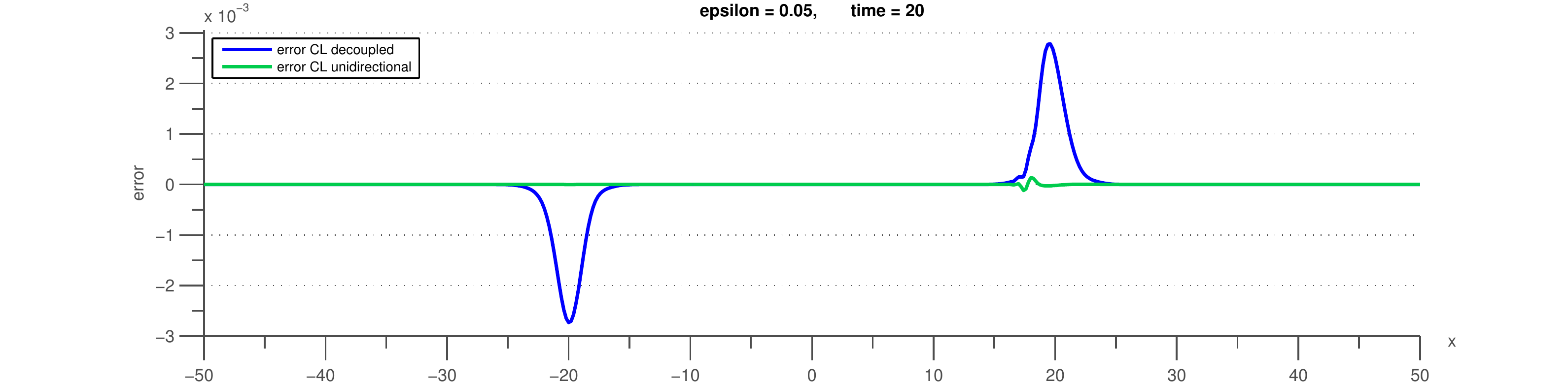}
\label{fig:Uni102c}
}\end{center}
\caption{Unidirectional, localized initial data, non-critical ratio, Camassa-Holm regime}
\label{fig:Uni102}
 \end{figure}
 
In Figure~\ref{fig:Uni201}, we choose a set up which is favorable to the decoupled approximation of Definition~\ref{def:CL}: critical ratio $\delta^2=\gamma=0.64$ and localized initial data, $\zeta\id{t=0}=\exp(-(x/2)^2)$. We see that the unidirectional approximation offers a greater accuracy than the decoupled approximation.
In particular, the decoupled model predicts a small wave moving on the left (of size $\O(\epsilon^2)$, as $\alpha_1=0$ in~\eqref{eq:ztov}), which is not predicted in the unidirectional model, and almost nonexistent in the solution of the Green-Naghdi system, as we can see in Figure~\ref{fig:Uni201c}. This short-time $\O(\epsilon^2)$ error of the CL decomposition is preserved over times of order $T=\O(1/\epsilon)$. As for the unidirectional model, the produced error is clearly, and as predicted by Corollary~\ref{cor:convergenceuni}, of size $\O(\epsilon^4\ t)$.

In Figure~\ref{fig:Uni102}, the ratio is non critical ($\delta=0.5,\gamma=0.9$), and initial data as previously.
The accuracy of the CL approximate solution of Definition~\ref{def:CL} is worse than in the critical case, as the short-time error is of size $\O(\epsilon)$. Once again, the same error estimates holds over times of order $T=\O(1/\epsilon)$. 
The accuracy of the unidirectional model is not affected, and is still of size $\O(\epsilon^4\ t)$: {\em the criticality of the depth-ratio do not play a role in the accuracy of the unidirectional approximation.}
\medskip

Let us now turn to the following question: {\em is it true that after a certain time, any perturbation will decompose into two waves, each one satisfying (approximatively) an equation of the form~\eqref{eq:ztov}?}
Our answer is numerical. We use the numerical simulations of Section~\ref{sssec:CH1} (non-critical ratio, localized initial data), and test the right wave, defined simply as the part of the signal located in the right half-line $x>0$, of the numerical solution of the Green-Naghdi system against~\eqref{eq:ztov}. As we can see in Figure~\ref{fig:ztova} (where the log of the error is plotted to ease the viewing), even with such a rough a crude exploration, a very good agreement appears after a given time $T_0$, which is independent of $\epsilon$ (but rather depends on the thickness, or wavelength of the initial data). The accuracy of this agreement is in our simulation of size $\O(\epsilon^4)$, and valid for long times; see Figure~\ref{fig:ztovb}. 

\begin{figure}[!hbt] 
 \subfigure[behavior of the error with respect to time]{
\includegraphics[width=0.5\textwidth]{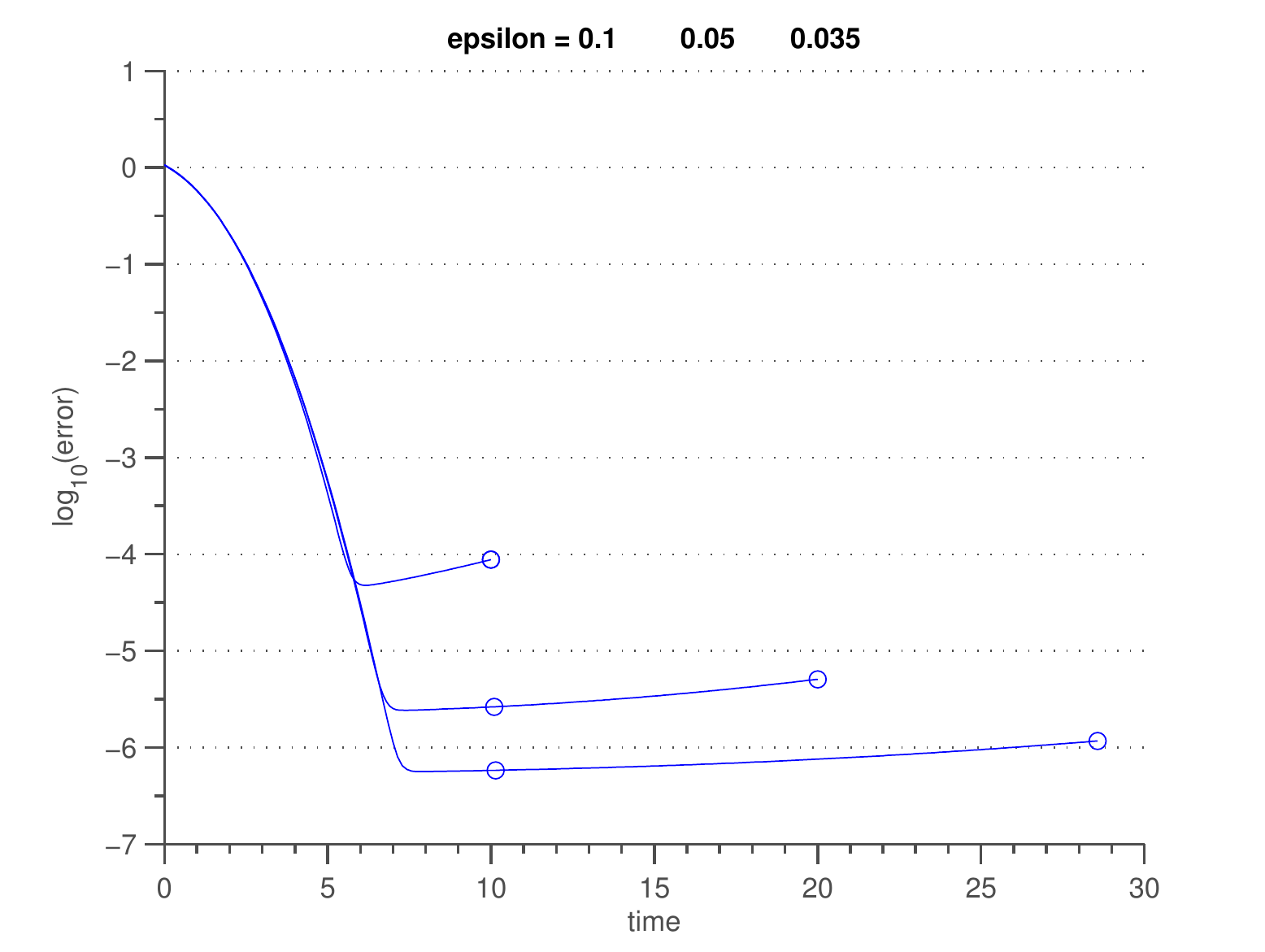}
\label{fig:ztova}
}
 \subfigure[behavior of the error with respect to $\epsilon=\sqrt\mu$]{
\includegraphics[width=0.5\textwidth]{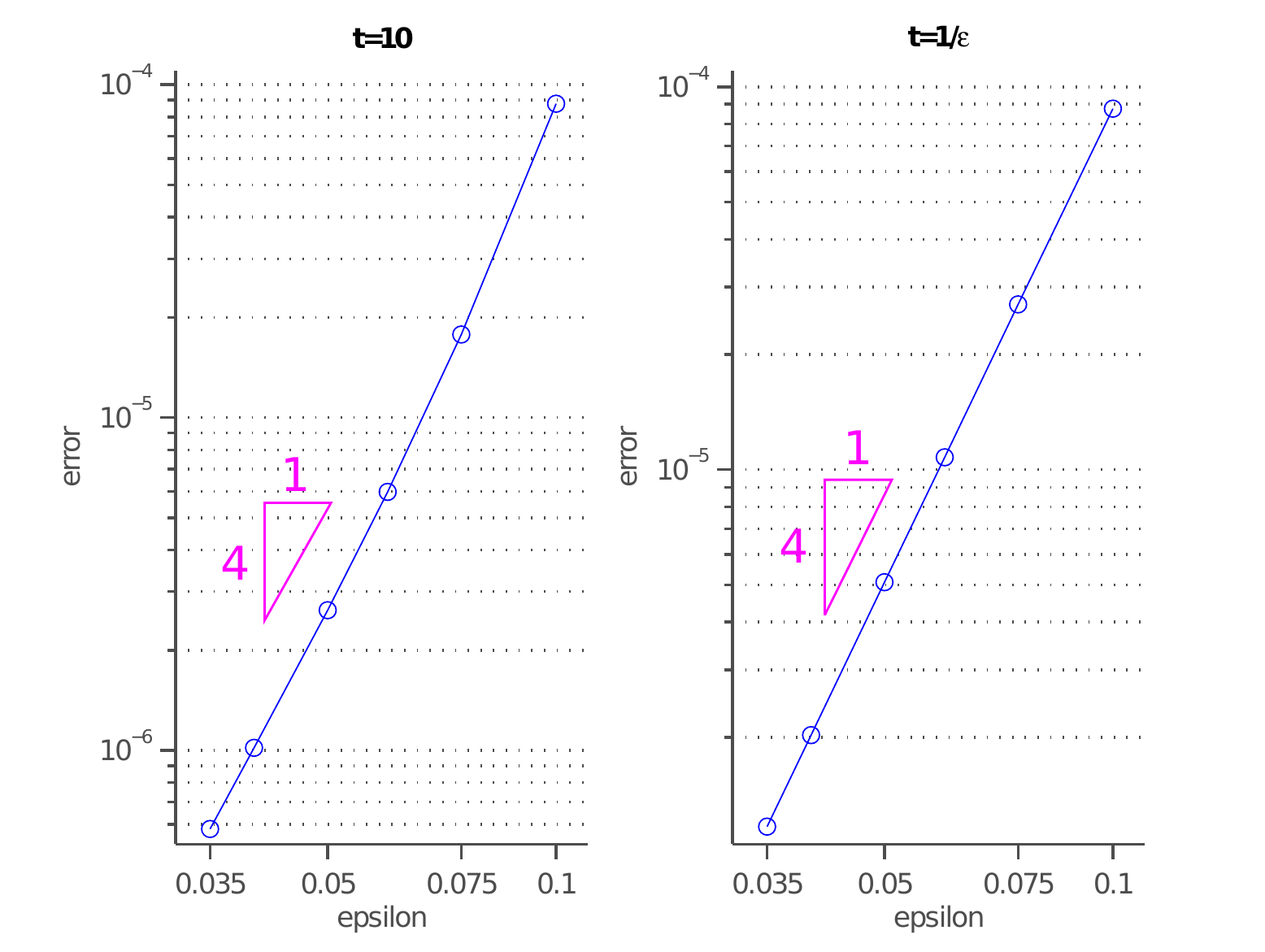}
\label{fig:ztovb}
}
\caption{Validity of~\eqref{eq:ztov} for generic initial data}
\label{fig:ztov}
\end{figure}

\paragraph{Numerical scheme.} The major difficulty concerning the numerical simulations computed throughout this paper is the fact that as $\epsilon$ is small (down to $0.035$), the time domain $t\in[0,1/\epsilon]$ ($t\in[0,\epsilon^{-3/2}]$ in Section~\ref{ssec:discussion}), and therefore the space domain of computation becomes very large. The decoupled models can be solved very efficiently by using a frame of reference moving with the decoupled wave, but solving numerically the Green-Naghdi scheme has to be carried out on the full time/space domain. Thus we need a scheme which allows a great accuracy, with a relatively small computation cost. With this in mind, we turn to multi-step, explicit and spectral methods.
The space discretization, and in particular the discrete differentiation matrices use trigonometric polynomial on an equispaced grid, as described in~\cite{Trefethen} (precisely (1.5)). This yields an exponential accuracy with the size of the grid $\Delta x$, if the signal is smooth (note that the major drawback is that the discrete differentiation matrices are not sparse). It turns out setting $\Delta x=0.2$ is sufficient for the numerical errors to be undetectable. After this space discretization, one has to solve a system of ordinary differential equations in time, and we use the Matlab solver \verb+ode113+, which is based on the explicit, multistep, Adams-Bashforth-Moulton method~\cite{ShampineReichelt97}, with a stringent tolerance of $10^{-8}$.

\section{Decomposition of the flow}\label{sec:decomposition}

In this section, our aim is to obtain approximate solutions to~\eqref{eqn:GreenNaghdiMeanI}, through a decomposition of the flow into two independent waves, each one satisfying a scalar evolution equation. Our aim is dual. First, we want to investigate which scalar equation each of these waves should to satisfy, in order to be as accurate as possible. Then, we want to estimate the size of the error we commit by neglecting the coupling between the two components. 

We first give a fairly simple formal approach in Section~\ref{ssec:formal}, which allows to heuristically construct the decoupled equations at stake, and consequently the approximate solutions (as precisely defined in Definitions~\ref{def:CL} and~\ref{def:other}). In Section~\ref{sec:prop:decompositionI}, we give the proof of the rigorous justification of such approximations; see Propositions~\ref{prop:decompositionI} and~\ref{prop:other}. Section~\ref{ssec:discussion} contains a discussion on our result, considering various different regimes and decoupled models, and supported with numerical simulations.

\subsection{Formal approach}\label{ssec:formal}

The main idea of the decomposition is that at first order (that is, setting $\epsilon=\mu=0$), our system of equations~\eqref{eqn:GreenNaghdiMeanI} is simply a linear wave equation
\begin{equation}\label{eq:wave} \partial_t U \ + \ A_0 \partial_x U \ = \ \O(\epsilon,\mu), \quad \text{with}\quad A_0=\begin{pmatrix} 0& \frac{1}{\gamma+\delta} \\ \gamma+\delta & 0
\end{pmatrix}
 \end{equation}
and $U=(\zeta,\b v)^T$. It is straightforward to check that $A_0$ has two distinct eigenvalues, therefore, we can find a basis of $\RR^2$ such that~\eqref{eq:wave} reduces to two decoupled equations. More precisely, there exists
\[ P \ = \ \begin{pmatrix} 1& 1 \\ \gamma+\delta &-(\gamma+\delta)
\end{pmatrix}, P^{-1} \ = \ \frac{1}{2}\begin{pmatrix} 1& \frac{1}{\gamma+\delta} \\ 1 &\frac{-1}{\gamma+\delta}
\end{pmatrix} \quad \text{such that}\quad P^{-1} A_0P \ = \ \begin{pmatrix} 1& 0 \\ 0 & -1
\end{pmatrix}. \]
Now, let us define $( u_l ,u_r)=P^{-1}U=\frac12(\zeta+\frac{\bar v}{\gamma+\delta},\zeta-\frac{\bar v}{\gamma+\delta})$. Multiplying~\eqref{eq:wave} by $P^{-1}$ (on the left), and keeping only first order terms, yields
\begin{equation}\label{eq:dwave} \partial_t \begin{pmatrix} u_l \\ u_r \end{pmatrix} \ + \ \begin{pmatrix} 1 & 0 \\ 0& -1 \end{pmatrix}\partial_x \begin{pmatrix} u_l \\ u_r \end{pmatrix} \ = \ \O(\epsilon,\mu). \end{equation}
{\em As a conclusion, the solution $U=(\zeta,\b v)^T$ may be decomposed in the following way:}
\[ U \ \approx \ \Big( v_+(x-t)+ v_-(x+t) \ , \ (\gamma+\delta)\big(v_+(x-t)+ v_-(x+t) \big)\Big)^T.\]
\medskip

Now, let us take into account the higher order terms in~\eqref{eqn:GreenNaghdiMeanI}. One will make use of the following straightforward expansions:
 \begin{align*} 
 \dfrac{h_1 h_2}{h_1+\gamma h_2}\ &= \ \frac1{\gamma+\delta}\ +\ \epsilon\frac{\delta^2-\gamma}{(\gamma+\delta)^2}\zeta \ - \ \epsilon^2\frac{\gamma\delta(\delta+1)^2}{(\gamma+\delta)^3}\zeta^2 \ - \ \epsilon^3\frac{\gamma\delta^2(\delta+1)^2(1-\gamma)}{(\gamma+\delta)^4}\zeta^3 \ + \ \O(\epsilon^4),\\
 \dfrac{{h_1}^2 -\gamma {h_2}^2 }{(h_1+\gamma h_2)^2} \ &= \ \frac{\delta^2-\gamma}{(\gamma+\delta)^2} \ - \ 2\epsilon\frac{\gamma\delta(\delta+1)^2}{(\gamma+\delta)^3}\zeta \ - \ 3\epsilon^2\frac{\gamma\delta^2(\delta+1)^2(1-\gamma)}{(\gamma+\delta)^4}\zeta^2 \ + \ \O(\epsilon^3),\\
 \overline{\Q}[h_1,h_2]\b v \ &= \ - \frac{1+\gamma\delta}{3\delta(\gamma+\delta)}\partial_x^2 \b v-\epsilon\frac{\gamma+\delta}3 \left((\beta-\alpha)\b v \partial_x^2\zeta \ + \ (\alpha+2\beta)\partial_x(\zeta\partial_x\b v)-\beta\zeta\partial_x^2\b v\right) \ + \ \O(\epsilon^2), \\
 \overline{\R}[h_1,h_2]\b v \ &\equiv \ \alpha\left(\frac12 (\partial_x \b v)^2+\frac13\b v\partial_x^2 \b v\right) \ + \ \O(\epsilon)
 \end{align*}
with $\alpha=\dfrac{1-\gamma}{(\gamma+\delta)^2}$ and $\beta=\dfrac{(1+\gamma\delta)(\delta^2-\gamma)}{\delta(\gamma+\delta)^3}$.

Using the decomposition as above: $( u_l ,u_r)=P^{-1}(\zeta,\b v)^T=\frac12(\zeta+\frac{\b v}{\gamma+\delta},\zeta-\frac{\b v}{\gamma+\delta})$, and withdrawing every term of size $\O(\mu\epsilon^2,\epsilon^4)$,
one can check that the Green-Naghdi system~\eqref{eqn:GreenNaghdiMeanI} becomes the following coupled system in terms of $u_l$ (left-going wave) and $u_r$ (right-going wave):
 \begin{subequations}
 \begin{align}
 \partial_t u_l \ + \ \partial_x u_l \ + \ f_l^{\epsilon,\mu}(u_l,u_r) \ &= \ 0, \label{eq:dSerrel} \\
 \partial_t u_r \ - \ \partial_x u_r \ + \ f_r^{\epsilon,\mu}(u_l,u_r) \ &= \ 0, \label{eq:dSerrer}
 \end{align}
 where $f_l^{\epsilon,\mu}$ and $f_l^{\epsilon,\mu}$ are defined below:
 \begin{align*} f_l^{\epsilon,\mu} &= \ \epsilon\frac{\alpha_1}{2}\partial_x \big((u_l+\frac13u_r)(u_l-u_r)\big) + \epsilon^2 \frac{\alpha_2}3 \partial_x\big((u_l-u_r)u_l(u_l+u_r)\big)\\
&\qquad +\epsilon^3 \frac{\alpha_3}4\partial_x\big((u_l-\frac15u_r)(u_l-u_r)(u_l+u_r)^2\big) 
\\
&\qquad - \mu\nu\partial_x^2\partial_t(u_l-u_r)+\mu\epsilon\kappa_3\partial_x\Big( \frac13(u_l-u_r)(\partial_x^2 u_l-\partial_x^2 u_r)+\frac12(\partial_x u_l-\partial_x u_r)^2 \Big)\\
&\qquad - \mu\epsilon \partial_t\Big[ \kappa_1(u_l\partial_x^2u_l-u_r\partial_x^2 u_r)+\kappa_2 (u_r\partial_x^2u_l-u_l\partial_x^2u_r)+(\kappa_1+\frac{\kappa_2}2)\big((\partial_x u_l)^2-(\partial_x u_r)^2\big)\Big], \\
 f_r^{\epsilon,\mu} &= -\epsilon\frac{\alpha_1}2\partial_x \big((\frac13 u_l+u_r)(u_r-u_l)\big) - \epsilon^2 \frac{\alpha_2}3\partial_x\big((u_r-u_l)u_r(u_l+u_r)\big)\\
 &\qquad -\epsilon^3\frac{\alpha_3}4\partial_x\big((u_r-\frac15u_l)(u_r-u_l)(u_l+u_r)^2\big) \\
&\qquad -\mu\nu \partial_x^2\partial_t(u_r-u_l) -\mu\epsilon\kappa_3\partial_x\Big( \frac13(u_r-u_l)(\partial_x^2 u_r-\partial_x^2 u_l)+\frac12(\partial_x u_r-\partial_x u_l)^2 \Big) \\
&\qquad - \mu\epsilon \partial_t\Big[ \kappa_1(u_r\partial_x^2u_r-u_l\partial_x^2 u_l)+\kappa_2 (u_l\partial_x^2u_r-u_r\partial_x^2u_l)+(\kappa_1+\frac{\kappa_2}2)\big((\partial_x u_r)^2-(\partial_x u_l)^2\big)\Big],
\end{align*}
 \end{subequations}
and where we used the notation
\begin{equation}\label{eqn:par0}
\begin{array}{c}
\alpha_1=\frac32\frac{\delta^2-\gamma}{\gamma+\delta},\quad \alpha_2=-3\frac{\gamma\delta(\delta+1)^2}{(\gamma+\delta)^2},\quad \alpha_3=-5\frac{\delta^2(\delta+1)^2\gamma(1-\gamma)}{(\gamma+\delta)^3},\quad \nu=\frac16\frac{1+\gamma\delta}{\delta(\delta+\gamma)}.\\
\kappa_1=\frac{(1+\gamma\delta)(\delta^2-\gamma)}{3\delta(\gamma+\delta)^2},\quad \kappa_2=\frac{(1-\gamma)}{3(\gamma+\delta)},\quad \kappa_3=\frac{\gamma-1}{2(\gamma+\delta)}.\end{array}\end{equation}

\begin{Remark}\label{rem:CH}
One could use higher order expansions with respect to the parameter $\epsilon$, which would lead to decoupled models with formally higher accuracy. However, let us note that our results (see discussion in Section~\ref{ssec:discussion}) show that the main error of our decoupled models comes from neglecting coupling terms which arise at low order, rather than the unidirectional error which is produced by neglecting higher order terms in the scalar equations. Thus including these higher order terms in the evolution equation is unlikely to produce substantial improvement. 
\end{Remark}

\begin{Proposition}[Consistency of~\eqref{eq:dSerrel}--\eqref{eq:dSerrer}]\label{prop:ConsSerre}
Let $\big(u_l^\p,u_r^\p\big)$ be strong solutions of~\eqref{eq:dSerrel}--\eqref{eq:dSerrer}, depending on sets of parameters $(\epsilon,\mu,\delta,\gamma)=\p\in\P$, as defined in~\eqref{eqn:defRegimeFNL}. We assume that for $s\geq s_0>1/2$, $u_l^\p,u_r^\p \in W^{1}([0,T);H^{s+3}(\RR))$, uniformly in $\p$.
Additionally, we assume that there exists $h>0$ such that 
\[h_1^\p \ \equiv\ 1-\epsilon(u_l^\p+u_r^\p) \geq h>0, \qquad h_2^\p \ \equiv \ \frac1\delta +\epsilon (u_l^\p+u_r^\p)\geq h>0.\]
Then $\big(\zeta^\p,\b v^\p\big)^T\equiv \big(u_l^\p+u_r^\p,(\gamma+\delta)(u_l^\p-u_r^\p)\big)$ is consistent with Green-Naghdi equations~\eqref{eqn:GreenNaghdiMeanI} of order $s$ on $[0,T)$, at precision $\O(\mu\epsilon^2+\epsilon^4)$ (in the sense of Definition~\ref{def:consistency}).
\end{Proposition}
\begin{proof}
Proposition~\ref{prop:ConsSerre} is straightforward once one obtains a rigorous statement of the expansions of $ \frac{h_1^\p h_2^\p}{h_1^\p+\gamma h_2^\p}\ , \
 \frac{(h_1^\p)^2 -\gamma (h_2^\p)^2 }{(h_1^\p+\gamma h_2^\p)^2} \ , \
 \overline{\Q}[h_1^\p,h_2^\p]\b v^\p \ , \ 
 \overline{\R}[h_1^\p,h_2^\p]\b v^\p $, as stated above, the residual being estimated in $W^1([0,T);H^{s+1})$ norm. These expansions are easily checked, provided the assumptions of the proposition ($u_l^\p,u_r^\p \in W^{1}([0,T);H^{s+3}(\RR))$, $h_1^\p,h_2^\p\geq h>0$) are satisfied, and using uniformly continuous (for $s\geq s_0>1/2$) Sobolev embedding $H^s \hookrightarrow L^\infty$.
\end{proof}
\medskip

The decoupled approximation simply consists in neglecting all coupling terms in~\eqref{eq:dSerrel}--\eqref{eq:dSerrer}, that is replacing $f_l^{\epsilon,\mu}(u_l,u_r)$ by $f_l^{\epsilon,\mu}(u_l,0)$, and $f_r^{\epsilon,\mu}(u_l,u_r)$ by $f_r^{\epsilon,\mu}(0,u_r)$. This yields
 \begin{subequations}
\begin{align}\label{eq:CL0l}
\partial_t u_l \ +\ \partial_x u_l \ +\ \epsilon\alpha_1 u_l\partial_x u_l \ +\ \epsilon^2 \alpha_2 u_l^2\partial_x u_l \ +\ \epsilon^3 \alpha_3 {u_l}^3\partial_x u_l\ -\ \mu\nu\partial_x^2\partial_t u_l \quad &\nn \\
- \mu\epsilon \partial_t\big( \kappa_1 u_l\partial_x^2u_l+(\kappa_1+1/2\kappa_2)(\partial_x u_l)^2\big)+ \mu\epsilon\kappa_3\partial_x\big(\frac13u_l\partial_x^2 u_l+\frac12(\partial_x u_l)^2\big) &\ = 0, \\
\label{eq:CL0r}
\partial_t u_r\ -\ \partial_x u_r\ -\ \epsilon\alpha_1 u_r\partial_x u_r\ -\ \epsilon^2 \alpha_2 u_r^2\partial_x u_r \ -\ \epsilon^3 \alpha_3 {u_r}^3\partial_x u_r\ -\ \mu\nu\partial_x^2\partial_t u_r \quad & \nn \\
- \mu\epsilon \partial_t\big( \kappa_1 u_r\partial_x^2u_r+(\kappa_1+1/2\kappa_2)(\partial_x u_r)^2\big)- \mu\epsilon\kappa_3\partial_x\big(\frac13u_r\partial_x^2 u_r+\frac12(\partial_x u_r)^2\big) &\ = 0, 
\end{align}
 \end{subequations}
and~\eqref{eq:CL0l}--\eqref{eq:CL0r} are the decoupled equations we consider; see Definition~\ref{def:CL}.
\medskip

Let us now reckon that one can deduce from~\eqref{eq:CL0l}--\eqref{eq:CL0r} a large family of formally equivalent models, with different parameters, following the techniques used for example in~\cite{BonaChenSaut02,BonaColinLannes05,ConstantinLannes09,Duchene11a}, and that we discuss below.
\begin{itemize}
\item{\em The BBM trick} (from Benjamin-Bona-Mahony~\cite{BenjaminBonaMahony72}). 
Keeping only the first order terms in equations~\eqref{eq:CL0l}--\eqref{eq:CL0r}, one has the simple transport equations
\[ \partial_t u_l \ + \ \partial_x u_l \ = \ \O(\mu,\epsilon), \quad \text{ and } \quad \partial_t u_r \ - \ \partial_x u_r \ = \ \O(\mu,\epsilon).\]
It follows that one can replace time derivatives in higher order terms (of order $\O(\mu\epsilon)$) by spatial derivatives (up to a sign), and both equations have formally the same order of accuracy. In order to simplify, we consider only equations with spatial derivatives in $\O(\mu\epsilon)$ terms (if they exist). In particular,~\eqref{eq:CL0l}--\eqref{eq:CL0r} becomes
\begin{subequations}
\begin{align}\label{eq:CL1l}
\partial_t u_l \ +\ \partial_x u_l \ +\ \epsilon\alpha_1 u_l\partial_x u_l \ +\ \epsilon^2 \alpha_2 u_l^2\partial_x u_l \ +\ \epsilon^3 \alpha_3 {u_l}^3\partial_x u_l\ -\ \mu\nu\partial_x^2\partial_t u_l \quad &\nn \\
+ \mu\epsilon \partial_x\big(( \kappa_1+1/3\kappa_3) u_l\partial_x^2u_l+(\kappa_1+1/2\kappa_2+1/2\kappa_3)(\partial_x u_l)^2\big) &\ = 0, \\
\label{eq:CL1r}
\partial_t u_r\ -\ \partial_x u_r\ -\ \epsilon\alpha_1 u_r\partial_x u_r\ -\ \epsilon^2 \alpha_2 u_r^2\partial_x u_r \ -\ \epsilon^3 \alpha_3 {u_r}^3\partial_x u_r\ -\ \mu\nu\partial_x^2\partial_t u_r \quad & \nn \\
- \mu\epsilon \partial_x\big( (\kappa_1+1/3\kappa_3) u_r\partial_x^2u_r+(\kappa_1+1/2\kappa_2+1/2\kappa_3)(\partial_x u_r)^2\big) &\ = 0. 
\end{align}
 \end{subequations}
\medskip

Following a similar idea, we make use of the low order identity obtained from~\eqref{eq:CL1l}--\eqref{eq:CL1r}:
\[ \partial_t u_l \ + \ \partial_x u_l \ + \ \epsilon\alpha_1 u_l\partial_x u_l \ = \ \O(\mu,\epsilon^2),\]
so that one has, for any $\theta\in\RR$,
\[ \partial_t u_l \ = \ \theta\partial_t u_l+(\theta-1)\big(\partial_x u_l + \epsilon\alpha_1 u_l\partial_x u_l\big) \ + \ \O(\mu,\epsilon^2).\] 
Plugging back into~\eqref{eq:CL1l} and withdrawing $\O(\mu^2,\mu\epsilon^2)$ terms yields
\begin{subequations}
\begin{align}\label{eq:CL2l}
\partial_t u_l \ +\ \partial_x u_l \ +\ \epsilon\alpha_1 u_l\partial_x u_l \ +\ \epsilon^2 \alpha_2 u_l^2\partial_x u_l \ +\ \epsilon^3 \alpha_3 {u_l}^3\partial_x u_l\ -\ \mu\nu^\theta_t\partial_x^2\partial_t u_l \ \quad & \nn \\
+ \ \mu\nu^\theta_x\partial_x^3 u_l \ +\ \mu\epsilon \partial_x\big( \kappa_1^\theta u_l\partial_x^2u_l+\kappa_2^\theta(\partial_x u_l)^2\big) &\ = 0,\\\label{eq:CL2r}
\partial_t u_r\ -\ \partial_x u_r\ -\ \epsilon\alpha_1 u_r\partial_x u_r\ -\ \epsilon^2 \alpha_2 u_r^2\partial_x u_r \ -\ \epsilon^3 \alpha_3 {u_r}^3\partial_x u_r\ -\ \mu\nu^\theta_t\partial_x^2\partial_t u_r\ \quad & \nn \\
 -\ \mu^\theta_x\nu\partial_x^3 u_r\ - \mu\epsilon \partial_x\big( \kappa_1^\theta u_r\partial_x^2u_r+\kappa_2^\theta(\partial_x u_r)^2\big) &\ = 0,\end{align}
 \end{subequations}
where we have defined, after parameters~\eqref{eqn:par0},
\begin{equation}\label{eqn:partheta}
 \nu_t^\theta\equiv \theta\nu, \quad \nu_x^\theta\equiv (1-\theta)\nu, \quad \kappa_1^\theta \equiv \kappa_1+\frac{\kappa_3}3+(1-\theta)\alpha_1\nu, \quad \kappa_2^\theta \equiv \kappa_1+\frac{\kappa_2}2+ \frac{\kappa_3}2+(1-\theta)\alpha_1\nu.
 \end{equation}
 
 \item {\em Near identity changes of variables}. We used system~\eqref{eqn:GreenNaghdiMeanI} as our reference system, and therefore the unknowns we consider are $(\zeta,\overline v)$, where $\b v=\overline{u}_{2}-\gamma\overline{u}_{1}$ is the shear layer-mean velocity. However, other natural variables may also be used, such as the the {\em shear velocity at the interface} (leading to system~\eqref{eqn:GreenNaghdi2})
 \[ v_0 \ = \ \partial_x\left(\big(\phi_2 - \gamma\phi_1 \big)\id{z=\epsilon\zeta}\right) \ = \ \partial_x\Big( \phi_2(x,r_2(x,0))-\gamma\phi_1(x,r_1(x,0))\Big) , \]
 (where we use the change of coordinate flattening the fluid domains: $r_i(x, \t z)=\t z h_i(x)+\epsilon\zeta(x)$; see~\cite{Duchene10}), 
 or more generally, using the horizontal derivative of the potential at specific heights ($z_1,z_2)\in [0,1)\times(-1,0]$: 
 \[ v^{z_1,z_2} \ = \ \partial_x\Big( \phi_2(x,r_2(x,z_2))-\gamma \phi_1(x,r_2(x,z_1))\Big) . \]
Using the expansion of the velocity potentials and layer-mean velocities, as obtained in~\cite{Duchene10}, as well as the identity $h_1\overline{u}_{1}+h_2\overline{u}_{2}=0$ (obtained through the rigid lid assumption), yields the following approximation: 
\[v^{z_1,z_2} \ = \ \bar v \ + \ \mu \lambda^{z_1,z_2}\partial_x^2 \bar v\ +\ \mu\epsilon \T^{z_1,z_2}( \zeta,\bar v) \ + \ \O(\mu^2,\mu\epsilon^2),\]
with $\lambda^{z_1,z_2} = - \dfrac{(3 {z_2}^2+6 z_2+2)+\gamma\delta(3z_1^2-6z_1+2)}{6\delta(\gamma+\delta)}$, and $\T^{z_1,z_2}$ a bi-linear differential operator whose precise formula do not play a significant role in our work as we shall discuss below.
Following this idea, we consider
\begin{align*}u_l^{\lambda} \ &\equiv \ u_l \ +\ \mu \lambda \partial_x^2 u_l\ , \\
 u_r^{\lambda} \ &\equiv \ u_r\ -\ \mu \lambda \partial_x^2 u_r\ .\end{align*}
If $(u_l,u_r)$ satisfies~\eqref{eq:CL2l}--\eqref{eq:CL2r}, then $(u_l^{\lambda},u_r^{\lambda})$ satisfies the following equations, up to terms of order $\O(\mu^2,\mu\epsilon^2)$:
\begin{subequations}
\begin{align}\label{eq:CL3l}
\partial_t u_l^\lambda \ +\ \partial_x u_l^\lambda \ +\ \epsilon\alpha_1 u_l^\lambda\partial_x u_l^\lambda \ +\ \epsilon^2 \alpha_2 {u_l^\lambda}^2\partial_x u_l^\lambda \ +\ \epsilon^3 \alpha_3 {u_l^\lambda}^3\partial_x u_l^\lambda\ -\ \mu\nu^{\theta,\lambda}_t\partial_x^2\partial_t u_l^\lambda \ \quad & \nn \\
+ \ \mu\nu^{\theta,\lambda}_x\partial_x^3 u_l^\lambda \ +\ \mu\epsilon \partial_x\big( \kappa_1^{\theta,\lambda} u_l^\lambda\partial_x^2u_l^\lambda+\kappa_2^\theta(\partial_x u_l^\lambda)^2\big) &\ = 0, \\
\label{eq:CL3r}
\partial_t u_r^\lambda\ -\ \partial_x u_r^\lambda\ -\ \epsilon\alpha_1 u_r^\lambda\partial_x u_r^\lambda\ -\ \epsilon^2 \alpha_2 {u_r^\lambda}^2\partial_x u_r^\lambda \ -\ \epsilon^3 \alpha_3 {u_r^\lambda}^3\partial_x u_r^\lambda\ -\ \mu\nu^{\theta,\lambda}_t\partial_x^2\partial_t u_r^\lambda\ \quad & \nn \\
 -\ \mu\nu^{\theta.\lambda}_x\partial_x^3 u_r^\lambda\ - \mu\epsilon \partial_x\big( \kappa_1^{\theta,\lambda} u_r^\lambda\partial_x^2u_r^\lambda+\kappa_2^\theta(\partial_x u_r^\lambda)^2\big) &\ = 0,\end{align}
 \end{subequations}

where we have defined, after parameters~\eqref{eqn:partheta},
\begin{equation}\label{eqn:parthetalambda}
\nu_t^{\theta,\lambda}\equiv \nu_t^\theta + \lambda , \quad \nu_x^{\theta,\lambda}\equiv \nu_x^\theta -\lambda, \quad \kappa_1^{\theta,\lambda} \equiv \kappa_1^{\theta,\lambda}+\alpha_1\lambda . 
\end{equation}

 Note that in order to fit as much as possible with variables $(\zeta,v^{z_1,z_2})$, one could have used more complex change of variables, such as
\begin{align*}u_l^{\lambda} \ &\equiv \ u_l +\mu \lambda \partial_x^2 u_l\ + \ \mu\epsilon\big( \lambda_2 u_l \partial_x^2 u_l + \lambda_3\partial_x( {u_l}^2)\big),\\
 u_r^{\lambda} \ &\equiv \ u_r-\mu \lambda \partial_x^2 u_r\ - \ \mu\epsilon\big( \lambda_2 u_r \partial_x^2 u_r + \lambda_3\partial_x( {u_r}^2)\big),\end{align*}
 with $\lambda=\frac{\lambda^{z_1,z_2}}{2(\gamma+\delta)}$, and $\lambda_2,\lambda_3$ obtained through $\T^{z_1,z_2}$.
It turns out $u_l^{\lambda} $ and $u_r^{\lambda} $ would then satisfy the same equation as above: the new parameters do not depend on $\lambda_2$ and $\lambda_3$, as their contribution is of order $\O(\mu\epsilon^2,\mu^2\epsilon)$, after using BBM trick to suppress higher order derivatives with respect to time. We thus do not consider such changes of variable.
\end{itemize}
\bigskip

Ultimately, one obtains the family of approximations described in Definition~\ref{def:CL}, with the following set of parameters:
\begin{equation}\label{eqn:parameters}
\begin{array}{c} \displaystyle
\alpha_1=\frac32\frac{\delta^2-\gamma}{\gamma+\delta},\quad \alpha_2=-3\frac{\gamma\delta(\delta+1)^2}{(\gamma+\delta)^2},\quad \alpha_3=-5\frac{\delta^2(\delta+1)^2\gamma(1-\gamma)}{(\gamma+\delta)^3},\\ \displaystyle
 \nu_t^{\theta,\lambda}\equiv \frac\theta6\frac{1+\gamma\delta}{\delta(\delta+\gamma)} + \lambda , \qquad \nu_x^{\theta,\lambda}\equiv \frac{1-\theta}6\frac{1+\gamma\delta}{\delta(\delta+\gamma)} -\lambda,\\ \displaystyle
 \kappa_1^{\theta,\lambda} \equiv \frac{(1+\gamma\delta)(\delta^2-\gamma)}{3\delta(\gamma+\delta)^2}(1+\frac{1-\theta}4)-\frac{(1-\gamma)}{6(\gamma+\delta)}+\lambda\frac32\frac{\delta^2-\gamma}{\gamma+\delta} , \\ \displaystyle
 \kappa_2^\theta \equiv \frac{(1+\gamma\delta)(\delta^2-\gamma)}{3\delta(\gamma+\delta)^2}(1+\frac{1-\theta}4)-\frac{(1-\gamma)}{12(\gamma+\delta)}.\end{array}
\end{equation}

\subsection{Rigorous justification; proof of Proposition~\ref{prop:decompositionI} }\label{sec:prop:decompositionI}

In this section, we give the rigorous justification of the Constantin-Lannes decoupled approximation, as defined in Definition~\ref{def:CL}. More precisely, we prove Proposition~\ref{prop:decompositionI}, that we recall below.
\begin{Proposition}[Consistency]\label{prop:decomposition}
Let $\zeta^0,v^0\in H^{s+6}$, with $s\geq s_0> 3/2$. For $(\epsilon,\mu,\delta,\gamma)=\p\in\P$, as defined in~\eqref{eqn:defRegimeFNL}, we denote $U_{\text{CL}}^\p$ the unique solution of the CL approximation, as defined in Definition~\ref{def:CL}. For some given $M^\star_{s+6}>0$, sufficiently big, we assume that there exists $T^\star>0$ and a family $(U_{\text{CL}}^\p)_{\p\in\P}$ such that
\[ T^\star \ = \ \max\big( \ T\geq0\quad \text{such that}\quad \big\Vert U_{\text{CL}}^\p \big\Vert_{L^\infty([0,T);H^{s+6})}+\big\Vert \partial_t U_{\text{CL}}^\p \big\Vert_{L^\infty([0,T);H^{s+5})} \ \leq \ M^\star_{s+6}\ \big) \ .\]

Then there exists $U^c=U^c[U_{\text{CL}}^\p]$ such that $U\equiv U_{\text{CL}}^\p+U^c$ is consistent with Green-Naghdi equations~\eqref{eqn:GreenNaghdiMeanI} of order $s$ on $[0,t]$ for $t<T^\star$, at precision $\O(\eps^\star_{\text{CL}})$ with
\[ \eps^\star_{\text{CL}} \ = \ C\ \max(\epsilon^2(\delta^2-\gamma)^2,\epsilon^4,\mu^2)\ (1+\sqrt t) ,\]
with $C=C(\frac1{s_0-3/2},M^\star_{s+6},\frac1{\delta_{\text{min}}},\delta_{\text{max}},\epsilon_{\text{max}},\mu_{\text{max}},|\lambda|,|\theta|)$, and the corrector term $U^c$ is estimated as
\[ \big\Vert U^c \big\Vert_{L^\infty([0,T^\star];H^{s})}+\big\Vert \partial_t U^c \big\Vert_{L^\infty([0,T^\star];H^{s})} \leq C\ \max(\epsilon(\delta^2-\gamma),\epsilon^2,\mu) \min( t, \sqrt t) .\]

Additionally, if there exists $\alpha>1/2$, $M^\sharp_{s+6},\ T^\sharp>0$ such that 
\[ \sum_{k=0}^6\big\Vert (1+x^2)^\alpha \partial_x^k U_{\text{CL}}^\p \big\Vert_{L^\infty([0,T);H^{s})} +\sum_{k=0}^5\big\Vert (1+x^2)^\alpha \partial_x^k\partial_t U_{\text{CL}}^\p \big\Vert_{L^\infty([0,T);H^{s})} \ \leq \ M^\sharp_{s+6}\ ,\]
then $U\equiv U_{\text{CL}}^\p+U^c$ is consistent with Green-Naghdi equations~\eqref{eqn:GreenNaghdiMeanI} of order $s$ on $[0,t]$ for $t<T^\sharp$, at precision $\O(\eps_{\text{CL}}^\sharp)$ with
\[ \eps^\sharp_{\text{CL}} \ = \ C\ \max(\epsilon^2(\delta^2-\gamma)^2,\epsilon^4,\mu^2),\]
with $C=C(\frac1{s_0-3/2},M^\sharp_{s+6},\frac1{\delta_{\text{min}}},\delta_{\text{max}},\epsilon_{\text{max}},\mu_{\text{max}},|\lambda|,|\theta|)$ and $U^c$ is uniformly estimated as
\[ \big\Vert U^c \big\Vert_{L^\infty([0,T^\sharp];H^{s})}+\big\Vert \partial_t U^c \big\Vert_{L^\infty([0,T^\sharp];H^{s})} \leq C\ \max(\epsilon(\delta^2-\gamma),\epsilon^2,\mu) \min( t, 1). \]
\end{Proposition}

\begin{Remark}The proof of Proposition~\ref{prop:other} is identical to the one of Proposition~\ref{prop:decomposition}, presented below. More precisely, the choice of lower-order evolution equations modifies mostly the last term of~\eqref{eq:estR} in Lemma~\ref{lem:proofcons} (below), as the contribution of neglected unidirectional terms should be added. The additional error is therefore uniformly bounded over times $[0,T^\star]$ and $[0,T^\sharp]$. The detailed proof of Proposition~\ref{prop:other} is omitted.\end{Remark}

\paragraph{Strategy of the proof.} 

Our strategy is the following. Inspired by the calculations of Section~\ref{ssec:formal}, we seek
an approximate solution of~\eqref{eqn:GreenNaghdiMeanI} under the form
\[ U_{\text{app}}\ \equiv \ \Big(v_+(t,x-t)+v_-(t,x+t),(\gamma+\delta)\big(v_+(t,x-t)-v_-(t,x+t)\big)\Big) \ + \ \sigma U_{c}[v_\pm] \ ,\]
where $v_\pm$ satisfies the Constantin-Lannes equation~\eqref{eq:CL}, and $U_{c}$ contains the leading order coupling terms. The parameter $\sigma$ is assumed to be small, and we want to justify our approximate solutions over times of order $\O(1/\sigma)$.

Precisely, our aim is to prove that
\begin{enumerate}[i.]
\item the coupling term $\sigma U_{c}$ can be controlled, and grows sublinearly in time;
\item the approximate function $U_{\text{app}} $ solves the coupled equation~\eqref{eqn:GreenNaghdiMeanI}, up to a small remainder.
\end{enumerate}
\medskip 

As we shall see, controlling the secular growth of $U_{c}$ requires to consider separately the short time scale, where the coupling effects may be strong, and long time scale, where the coupling between two localized waves moving in opposite directions can be controlled. Thus we introduce the long time scale $\tau=t/\sigma$, and will seek an approximate solution of~\eqref{eqn:GreenNaghdiMeanI} as
\begin{align*} & U_{\text{app}}(t, x) \ = \ U_{\text{app}}(\sigma t, t, x), \quad \text{with} \\
& U_{\text{app}}(\tau, t, x) \ = \ \Big(v_+(\tau, t, x)+v_-(\tau, t, x)\ ,\ (\gamma+\delta)\big(v_+(\tau, t, x)-v_-(\tau, t, x)\big)\Big) \ + \ \sigma U_{c}[v_\pm](\tau, t, x) \ , \nn \end{align*}
with obvious misuses of notation.

Plugging the Ansatz into the coupled Green-Naghdi equation~\eqref{eqn:GreenNaghdiMeanI} yields at first order
\begin{equation}\label{eq:transport}
\partial_t v_+ \ + \ \partial_x v_+ \ = \ 0 \qquad \text{ and } \qquad \partial_t v_- \ - \ \partial_x v_- \ = \ 0,
\end{equation}
so that $v_\pm(\tau, t,x) \ = \ \t v_\pm(\tau, x_\pm)\equiv \t v_\pm(\tau, x\mp t)$. 
\medskip

At next order, and following the calculations of Section~\ref{ssec:formal}, one obtains the decoupled equations
\begin{align}\label{eq:v0}
\sigma(1- \mu\nu_t \partial_{x}^2)\partial_\tau \t v_\pm \ \pm \ \epsilon\alpha_1 \t v_\pm\partial_x \t v_\pm \ \pm \ \epsilon^2 \alpha_2 ( \t v_\pm )^2 \partial_x \t v_\pm\ \pm \ \epsilon^3 \alpha_3 (\t v_\pm)^3\partial_x \t v_\pm\ & \\
\pm\ \mu(\nu_x+\nu_t) \partial_x^3 \t v_\pm \ 
\pm \ \mu\epsilon\partial_x\big(\kappa_1 \t v_\pm\partial_x^2 \t v_\pm+\kappa_2(\partial_x \t v_\pm)^2\big) \ &= \ 0, \nn
\end{align}
where parameters satisfy identities of Definition~\ref{def:CL}. In order to deal with the coupling terms, we introduce the following first order correction.
\begin{Definition}\label{def:corrector}
We denote $U^c \equiv \ (u^c_++u^c_-,(\gamma+\delta)(u^c_+-u^c_-)\big)$ where $u^c_\pm[v_\pm](\tau, t, x)$ satisfies initial condition ${u^c_\pm}\id{t =0} \ \equiv\ 0$, and equation
\begin{equation}\label{eq:coupling}
\sigma(\partial_t+\partial_x) u^c_+ \ + \ f^l(\t v_+,\t v_-)-f^l(\t v_+,0) \ = \ 0, \ \text{ and } \ \sigma (\partial_t-\partial_x) u^c_- \ + \ f^r(\t v_+,\t v_-)-f^r(0,\t v_-) \ = \ 0,
\end{equation}
where $f^l$ and $f^r$ are defined as in~\eqref{eq:dSerrel}--\eqref{eq:dSerrer}. 
\end{Definition}

\begin{Remark} \label{rem:visCL} The above discussion does not take into account the use of near identity change of variable as described in Section~\ref{ssec:formal}. In that case, we set the initial data as follows:
\begin{equation} \label{eq:initialdata}
\t v_\pm\id{\tau=0} \ = \ (1\pm\mu\lambda\partial_x^2)\left(\frac{\zeta^0\pm\frac{v^0}{\gamma+\delta}}2\right).
\end{equation}
Thus the function $\t v_\pm(\sigma t,x\mp t)$ , defined by~\eqref{eq:v0} and with initial data~\eqref{eq:initialdata}, satisfies precisely
\[ v_\pm(\sigma t, t,x) \ = \ \t v_\pm(\sigma t,x\mp t) \ = \ v_\pm^\lambda(t,x\mp t),\]
where $v_\pm^\lambda$ is the solution of~\eqref{eq:CL} as defined in Definition~\ref{def:CL}. 
\end{Remark}

The proof is now as follows. First, we state that the approximate solution, $ U_{\text{app}}(\sigma t,t, x) $, constructed as above, does satisfy~\eqref{eqn:GreenNaghdiMeanI}, up to a small remainder (although depending on the size of $v_\pm$, $u^c_\pm$, and their derivatives). The fact that $v_\pm$ is uniformly bounded follows from the assumptions of the proposition. The key ingredient in the proof consists in estimating the growth over long times (that is, in the variable $\tau$), of $u^c_\pm$ satisfying~\eqref{eq:coupling}. Each of these steps are described in details in the following.
\bigskip

\paragraph{Construction and accuracy of the approximate solution $U_\text{app}$.}

The following Lemma states carefully the definition of our approximate solution $U_\text{app}$, and its precision in the sense of consistency.
\begin{Lemma}\label{lem:proofcons}
Let $\zeta^0,v^0\in H^{s+6}$, with $s\geq s_0>3/2$, and $(\epsilon,\mu,\delta,\gamma)=\p\in\P$, as defined in~\eqref{eqn:defRegimeFNL}. Let $v_\pm^\lambda(\tau, t,x)$ be defined by~\eqref{eq:transport},~\eqref{eq:v0} and with initial data~\eqref{eq:initialdata}. Then define $v_\pm(\tau,t,x)$ as 
\[ v_\pm(\tau,t,x) \ = \ (1\pm\mu\lambda\partial_x^2)^{-1} v_\pm^\lambda(\tau,t,x)\ , \] and set
\[ U_{\text{app}}(\tau, t, x) \ = \ \Big(v_+(\tau, t, x)+v_-(\tau, t, x)\ ,\ (\gamma+\delta)\big(v_+(\tau, t, x)-v_-(\tau, t, x)\big)\Big) \ + \ \sigma U_{c}[v_\pm](\tau, t, x),\]
where $U^c$ is defined in Definition~\ref{def:corrector}. Then for $\epsilon$ small enough,
$ U_{\text{app}}(\sigma t,t, x) $ satisfies the coupled equations~\eqref{eqn:GreenNaghdiMeanI}, up to a remainder, $\R$, bounded by 
\begin{align} \big\vert \R \big\vert_{H^s} \ \leq \ F\Big( \epsilon^4 (\big\vert v_\pm \big\vert_{H^{s+1}}+\sigma \big\vert u^c_\pm \big\vert_{H^{s+1}}) \ + \ \mu\epsilon^2(\big\vert v_\pm \big\vert_{H^{s+3}}+\big\vert \partial_t v_\pm \big\vert_{H^{s+2}}+\sigma \big\vert u^c_\pm \big\vert_{H^{s+3}}+\sigma\big\vert u^c_\pm \big\vert_{H^{s+2}}) \nn \\
 +\sigma^2 \big\vert \partial_\tau u^c_\pm \ \big\vert_{H^{s}} \
 + \ \epsilon\alpha_1\sigma \big\vert u^c_\pm \ \big\vert_{H^{s+1}} \left(\big\vert v_\pm \ \big\vert_{H^{s+1}}+\sigma \big\vert u^c_\pm \ \big\vert_{H^{s+1}}\right) \ +\ \epsilon^2\sigma (\big\vert u^c_\pm \ \big\vert_{H^{s+1}}\big\vert v_\pm \ \big\vert_{H^{s+1}}) \nn \\
\ + \ \mu\sigma \left(\big\vert \partial_x^3 u^c_\pm \ \big\vert_{H^{s}} + \big\vert \partial_x^2\partial_t u^c_\pm \ \big\vert_{H^{s}}\right) \left(1+\epsilon \left(\big\vert \partial_x^3 v_\pm \ \big\vert_{H^{s}} + \big\vert \partial_x^2\partial_t v_\pm \ \big\vert_{H^{s}}\right) \right)\nn\\
 + \ \max(\epsilon^4,\mu\epsilon^2,\mu^2)C(\big\vert v_\pm \big\vert_{H^{s+5}}+\big\vert \partial_t v_\pm \big\vert_{H^{s+4}}) \Big), \label{eq:estR}
\end{align}
with a function $F$ satisfying $F(X)\leq C(\frac1{s_0-3/2},M^\star_{s+6},\frac{1}{\delta_{\text{min}}},\delta_{\text{max}},\epsilon_{\text{max}},\mu_{\text{max}},|\lambda|,|\theta|) X$.
\end{Lemma}
\begin{proof}
This result a direct application of the definitions above, and calculations presented in Section~\ref{ssec:formal}. More precisely, one obtains
\[ \R \ = \ \R_0 \ + \ \sigma^2\partial_\tau u^c_\pm \ + \ \big(f(v_++\sigma u^c_+,v_-+\sigma u^c_-)-f(v_+,v_-)\big)\ + \ \R_\theta \ + \R_\lambda,\]
where \begin{itemize}
\item $\R_0$ is the contribution due to the expansion of $\frac{h_1 h_2}{h_1+\gamma h_2},
 \frac{{h_1}^2 -\gamma {h_2}^2 }{(h_1+\gamma h_2)^2} ,
 \overline{\Q}[h_1,h_2]\b v$ and $\overline{\R}[h_1,h_2]\b v$ with respect to parameter $\epsilon$;
 \item $f$ is a linear combination of $f^{\epsilon,\mu}_{l}$ and $f^{\epsilon,\mu}_{r}$ as defined in~\eqref{eq:dSerrel}--\eqref{eq:dSerrer};
 \item $\R_\theta$ and $\R_\lambda$ are respectively the components due to the use of the BBM trick and near-identity change of variable.
 \end{itemize}

Let us detail the first terms leading to $\R_0$. We want to control the contribution of the expansion of
$\frac{h_1 h_2}{h_1+\gamma h_2}$, that is, more precisely, estimate
\[\left\vert \R_0^{(1)} \right\vert_{H^{s+1}} \ = \ \left\vert\ \dfrac{h_1 h_2}{h_1+\gamma h_2} - \frac1{\gamma+\delta} - \epsilon\frac{\delta^2-\gamma}{(\gamma+\delta)^2}\zeta + \epsilon^2\frac{\gamma\delta(\delta+1)^2}{(\gamma+\delta)^3}\zeta^2 + \epsilon^3\frac{\gamma\delta^2(\delta+1)^2(1-\gamma)}{(\gamma+\delta)^4}\zeta^3 \right\vert_{H^{s+1}},\]
where we denote $\zeta=\zeta_{\text{app}}=(v_++v_-+\sigma(u^c_++u^c_))(\sigma t,t,x)$, and $h_1=1-\epsilon \zeta_{\text{app}},\ h_2=\frac1\delta+\epsilon \zeta_{\text{app}}$.

Note that, as we shall prove that $\zeta_{\text{app}}$ is bounded in $L^\infty([0,T^\star);H^{s+1})$, there exists $\epsilon_0>0$ such that for any $0\leq\epsilon<\epsilon_0$, one has
\[ \big\vert\zeta\big\vert_{L^\infty}<\min(1,1/\delta) \quad \text{ and } \quad\big\vert \dfrac{1}{h_1+\gamma h_2}\big\vert_{H^{s+1}} \leq \frac1{2(1+\gamma/\delta)} .\]
Now, one can easily check that, in that case, $\R_0^{(1)} \ = \ \epsilon^4 \frac{P^{(1)}(\zeta)}{h_1+\gamma h_2} $, where $P^{(1)}(X)$ is a polynomial of degree 4, and estimate
$ \left\vert \R_0^{(1)} \right\vert_{H^{s+1}} \ \leq \ \epsilon^4 F(\big\vert v_\pm \ \big\vert_{H^{s+1}}+\sigma \big\vert u^c_\pm \ \big\vert_{H^{s+1}})$ follows.

As $\frac{h_1 h_2}{h_1+\gamma h_2}$ gets multiplied by $\bar v$ and differentiated once, the first term in~\eqref{eq:estR} follows.

The contributions due to the expansion of $\frac{{h_1}^2 -\gamma {h_2}^2 }{(h_1+\gamma h_2)^2}$ is estimated in the same way. The contribution due to the expansion of $ \overline{\R}[h_1,h_2]\b v$ requires more derivatives, but may be estimated as above by $\mu\epsilon^2 F\Big( \big\vert v_\pm \big\vert_{H^{s+3}}+\sigma \big\vert u^c_\pm \big\vert_{H^{s+3}}\Big)$. The contribution due to the expansion of $\overline{\Q}[h_1,h_2]\b v$ involves one time derivative, and is controlled by $\mu\epsilon^2 F\Big( \big\vert \partial_t v_\pm \big\vert_{H^{s+2}}+\sigma\big\vert u^c_\pm \big\vert_{H^{s+2}}\Big)$. This yields the first line of~\eqref{eq:estR}.
\medskip

Then, one can check that
\begin{align*}& f^{\epsilon,\mu}_{l}(v_++\sigma u^c_+,v_-+\sigma u^c_-)-f^{\epsilon,\mu}_{l}(v_+,v_-)\ = \ \partial_x \left( P(\sigma u^c_\pm , v_\pm) \right)+\mu\sigma\nu\partial_x^2\partial_t u^c_\pm\\
&\quad +\mu\epsilon\partial_x\Big( Q_1(\sigma u^c_\pm,\partial_x^2 v_\pm)+Q_2(\sigma \partial_x u^c_\pm,\partial_x v_\pm) +Q_3(\sigma \partial_x^2 u^c_\pm, v_\pm)+Q_4(\sigma u^c_\pm,\partial_x^2 u^c_\pm)+Q_5(\sigma \partial_x u^c_\pm,\partial_x v_\pm)\Big) \\
&\quad +\mu\epsilon\partial_t \Big( Q_6(\sigma u^c_\pm,\partial_x^2 v_\pm)+Q_7(\sigma \partial_x u^c_\pm,\partial_x v_\pm) +Q_8(\sigma \partial_x^2 u^c_\pm, v_\pm)+Q_9(\sigma u^c_\pm,\partial_x^2 u^c_\pm)+Q_{10}(\sigma \partial_x u^c_\pm,\partial_x v_\pm)\Big), 
\end{align*}
where $P$ is a bivariate polynomial of degree 4, and whose leading order terms are 
\[ P(\sigma u^c_\pm , v_\pm) \ = \ \alpha_1\epsilon (\sigma u^c_\pm v_\pm)+\epsilon(\sigma u^c_\pm)^2 \ + \ \alpha_2\epsilon^2 \sigma u^c_\pm {v_\pm}^2 \ + \ \dots ,\]
 and $Q_i$ are bilinear forms. Each of these terms are bounded as in~\eqref{eq:estR}. 

Finally, $\R_\theta$ and $\R_\lambda$ depend uniquely on decoupled terms $v_\pm$, and one has similarly
\[
\big\vert \R_\theta\big\vert_{H^s} + \big\vert \R_\lambda\big\vert_{H^s} \leq C_0 \max(\epsilon^4,\mu\epsilon^2,\mu^2) (\big\vert v_\pm \big\vert_{H^{s+5}}+\big\vert \partial_t v_\pm \big\vert_{H^{s+4}}),\]
with $C_0=C(\frac1{s_0-3/2},\frac{1}{\delta_{\text{min}}},\delta_{\text{max}},\epsilon_{\text{max}},\mu_{\text{max}},|\lambda|,|\theta|)$. The proposition follows.
\end{proof}

Our aim is now to estimate each of the terms in~\eqref{eq:estR}.
\bigskip

\paragraph{Estimates of the decoupled approximation $v_\pm$.}
The bound on $v_\pm$ is a direct consequence of the assumptions of the Proposition,
 using the identity
\[ v_\pm(\sigma t, t,x) \ = \ \t v_\pm(\sigma t,x\mp t) \ = \ v_\pm(t,x\mp t),\]
where $v_\pm(t,x\mp t)$ is defined through Definition~\ref{def:CL}, and therefore is uniformly controlled as assumed in the Proposition. It follows
\begin{equation}\label{eq:estv0}
\big\Vert \t v_\pm \big\Vert_{L^\infty([0,\sigma T^\star)\times[0,T^\star);H^{s+6})}
\ + \ 
\sigma\big\Vert \partial_\tau \t v_\pm \big\Vert_{L^\infty([0,\sigma T^\star)\times[0,T^\star);H^{s+5})}\ \leq \ \ M^\star_{s+6}.
\end{equation}

 However, let us note that one can gain extra smallness on $\sigma \partial_\tau \t v_\pm$ (trading with a loss of derivatives) from the fact that $\t v_\pm$ satisfies~\eqref{eq:v0}. Indeed, one has for any $k\in \NN$,
\[ \sigma\big\vert\partial_\tau \t v_\pm\big\vert_{H^{s+k}} \ \leq \ \max(\epsilon\alpha_1,\epsilon^2,\mu) P(\big\vert\t v_\pm\big\vert_{H^{s+k+3}}) \ + \ \mu\nu_t \sigma\big\vert\partial_x^2\partial_\tau \t v_\pm\big\vert_{H^{s+k}} ,\]
where $P(X)$ is a polynomial, so that \footnote{Here and in the following, we do not explicitly keep track of the dependence with respect to all the parameters; as $(\epsilon,\mu,\delta,\gamma)=\p\in\P$, as defined in~\eqref{eqn:defRegimeFNL}, and parameters satisfy identities of Definition~\ref{def:CL}, one should replace $C(M^\star_{s+6})$ by $C(\frac1{s_0-3/2},M^\star_{s+6},\frac1{\delta_{\text{min}}},\delta_{\text{max}},\epsilon_{\text{max}},\mu_{\text{max}},|\lambda|,|\theta|)$.}
\begin{equation}\label{eq:estv1}
\sigma\big\Vert \partial_\tau \t v_\pm \big\Vert_{L^\infty([0,\sigma T^\star)\times[0,T^\star);H^{s+3})}\ \leq \ \ \max(\epsilon\alpha_1,\epsilon^2,\mu) C(M^\star_{s+6}).
\end{equation}

As for the case of localized initial data, the assumption of the Proposition yields
\begin{equation}\label{eq:estv0s}
\sum_{k=0}^6\big\Vert(1+x^2)^\alpha \partial_x^k v_\pm \big\Vert_{L^\infty([0,\sigma T^\sharp)\times[0,T^\sharp);H^{s}}+\sum_{k=0}^5\sigma\big\Vert(1+x^2)^\alpha \partial_\tau v_\pm \big\Vert_{L^\infty([0,\sigma T^\sharp)\times[0,T^\sharp);H^{s}} \leq M^\sharp_{s+6} .
\end{equation}
and one deduces as above, for $k=0,1,2,3$
\begin{align} \label{eq:estv1s}\sigma\big\vert(1+x^2)^\alpha\partial_x^k\partial_\tau \t v_\pm\big\vert_{H^{s}} &\leq \max(\epsilon\alpha_1,\epsilon^2,\mu) P\left(\sum_{i=0}^3\big\vert(1+x^2)^\alpha\partial_x^{k+i} \t v_\pm\big\vert_{H^{s}}\right) + \mu\sigma\big\vert(1+x^2)^\alpha\partial_x^{k+2}\partial_\tau \t v_\pm\big\vert_{H^{s}} \nn \\
 &\leq \ \max(\epsilon\alpha_1,\epsilon^2,\mu) C(M^\sharp_{s+6}).\end{align}

\paragraph{Control of the secular growth of the coupling terms $u^c_\pm$.}
Let us now study the term $U^c \equiv \ (u^c_++u^c_-,(\gamma+\delta)(u^c_+-u^c_-)\big)$, which contains all the coupling effects between the different components. Our aim is to control the secular growth of this term. This will be achieved through the following two Lemmas.
\begin{Lemma}
\label{lem:easy}
 Let $s\geq 0$, and $f^0 \in H^{s+1}(\RR)$. Then there exists a unique global strong solution, $u(t,x)\in C^0(\RR; H^{s+1})\cap C^1(\RR; H^{s})$, of 
\[
 \left\{\begin{array}{l}
 (\partial_t+c_1\partial_x) u=\partial_x f \\
 u\id{t=0}=0
 \end{array}\right. \ \text{ with }\ \ \quad
 \left\{\begin{array}{l}
 (\partial_t+c_2\partial_x) f =0\\
 {f_i}\id{t=0}=f^0
 \end{array}\right.
\]
where $c_1\neq c_2$. 

Moreover, one has the following estimates for any $t\in\RR$:
\[ \big\vert u(t,\cdot) \big\vert_{H^{s+1}(\RR)} \ \leq \ \frac{2}{|c_1-c_2|}\big\vert f^0 \big\vert_{H^{s+1}(\RR)} ,\qquad \big\vert u(t,\cdot) \big\vert_{H^{s}(\RR)} \ \leq \ | t| \ \big\vert f^0 \big\vert_{H^{s+1}(\RR)} .\]
\end{Lemma}
\begin{Lemma} \label{lem:Lannes} Let $s\geq s_0>1/2$, and $v^0_1$, $v^0_2 \in H^{s+1}(\RR)$. Then there exists a unique global strong solution, $u\in C^0(\RR; H^{s+1})\cap C^1(\RR; H^{s})$, of 
\[
 \left\{\begin{array}{l}
 (\partial_t+c\partial_x) u=g(v_1,v_2)\\
 u\id{t=0}=0
 \end{array}\right. \ \text{ with }\ \quad\forall i\in\{1,2\}, \quad
 \left\{\begin{array}{l}
 (\partial_t+c_i\partial_x) v_i=0\\
 {v_i}\id{t=0}=v^0_i
 \end{array}\right.
\]
where $c_1\neq c_2$, and $g$ is a bilinear mapping defined on $\RR^2$ and with values in $\RR$. 

Moreover, one has the following estimates:
\begin{enumerate}[i.]
\item 
\[ \big\Vert u \big\Vert_{L^\infty([0,t);H^{s})} \ \leq \ C_{s_0}\ t|\big\vert v^0_1 \big\vert_{H^{s}(\RR)}\big\vert v^0_2 \big\vert_{H^{s}(\RR)} , \quad \big\Vert \partial_t u \big\Vert_{L^\infty([0,t);H^{s})} \ \leq \ C_{s_0}\ |t|\big\vert v^0_1 \big\vert_{H^{s+1}(\RR)}\big\vert v^0_2 \big\vert_{H^{s+1}(\RR)}.\]
\item Using $c_1\neq c_2$, one has the sublinear growth 
\[ \big\vert u(t,\cdot) \big\vert_{H^{s}(\RR)}=C_{s_0}\frac{\sqrt t}{\sqrt{|c_1-c_2|}} \big\vert v^0_1 \big\vert_{H^{s}}\big\vert v^0_2 \big\vert_{H^{s}} \ \qquad \ \big\vert \partial_t u(t,\cdot) \big\vert_{H^{s}}=C_{s_0}\ \frac{1+\sqrt t}{\sqrt{|c_1-c_2|}} \big\vert v^0_1 \big\vert_{H^{s+1}}\big\vert v^0_2 \big\vert_{H^{s+1}},\]
with $C_{s_0}= C(\frac1{s_0-1/2})$.
\item If moreover, there exists $\alpha>1/2$ such that $v^0_1 (1+x^2)^\alpha$, and $v^0_2 (1+x^2)^\alpha\in H^{s}(\RR)$, then one has the (uniform in time) estimate 
\[ \big\Vert u \big\Vert_{L^\infty H^{s}(\RR)}\ \leq\ C_{s_0} \big\vert v^0_1 (1+x^2)^\alpha \big\vert_{H^{s}(\RR)}\big\vert v^0_2 (1+x^2)^\alpha \big\vert_{H^{s}(\RR)},\]
with $C_{s_0}=C(\frac1{s_0-1/2},\frac1{c_1-c_2},\frac1{\alpha-1/2}) $.
\end{enumerate}
 \end{Lemma}

\begin{proof} 
Lemma~\ref{lem:easy} is straightforward, since one has an explicit expression for the solution:
\[ u(t,x) \ = \ \frac1{c_2-c_1}\big(\ f(t,x+(c_2-c_1) t)-f(t,x)\ \big) \ = \ \frac1{c_2-c_1}\big(\ f^0(x-c_1 t)-f^0(x-c_2 t)\ \big).\]
It follows $\big\lvert u\big\rvert_{H^{s+1}} \ \leq \ \frac{2}{|c_2-c_1|}\big\lvert f^0 \big\rvert_{H^{s+1}}$, and
\[ \big\lvert u\big\rvert_{H^s} \ = \ \frac{1}{|c_2-c_1|}\big\lvert \int_{c_2t}^{c_1t} \ \partial_x f^0(x-y)\ dy \big\rvert_{H^s_x}\ \leq \ \frac{1}{|c_2-c_1|} \int_{c_2t}^{c_1t} \big\lvert \partial_x f^0(x-y) \big\rvert_{H^s_x} \ dy \ \leq \ |t| \big\lvert \partial_x f^0 \big\rvert_{H^s}.\]

As for Lemma~\ref{lem:Lannes}, the well-posedness as well as
estimate i. are standard; the remaining estimates, controlling the secular growth, are proved in~\cite{Lannes03}, Propositions~3.2 and~3.5.
\end{proof}
\medskip

One can now state our key result, concerning the control of the secular growth of $U^c$.
\begin{Lemma}\label{lem:estimatescoupling}
Let $v_\pm$ defined as in Lemma~\ref{lem:proofcons}, for $(\tau,t,x)\in[0,\sigma T^\star)\times[0,T^\star)\times\RR$. Then there exists a unique $U_{c}(\tau, t, x)\in L^\infty([0,T)\times\RR;H^{s})$ such that $U^c \equiv \ (u^c_++u^c_-,(\gamma+\delta)(u^c_+-u^c_-)\big)$, with $u^c_\pm$ satisfying~\eqref{eq:coupling} and $U^c\id{t =0} \ \equiv\ 0$. Moreover, one has the estimate for and $t\in [0, T^\star]$
\begin{align}
\sigma\big\Vert u^c_\pm \big\Vert_{L^\infty([0,\sigma T^\star]\times[0,t];H^{s})} &\leq C(M^\star_{s+3}) \max(\epsilon\alpha_1,\epsilon^2,\mu) \min( t, \sqrt t) , \label{eq:est1}\\
\sigma^2\big\Vert \partial_\tau u^c_\pm \big\Vert_{L^\infty([0,T]\times[0,t];H^{s})} &\leq C(M^\star_{s+6})\max(\epsilon\alpha_1,\epsilon^2,\mu)^2 \min( t, \sqrt t) .\label{eq:est2}
\end{align}

Moreover, if $(1+x^2)\zeta^0,(1+x^2)v^0\in H^{s+3}$, $s\geq s_0>1/2$, and $\zeta^0,v^0\in H^{s+3+2m}$, then one has the uniform estimate 
\begin{align}
\sigma\big\Vert u^c_\pm \big\Vert_{L^\infty([0,\sigma T^\sharp]\times[0,T^\sharp];H^{s})} &\leq C(M^\sharp_{s+3}) \max(\epsilon\alpha_1,\epsilon^2,\mu) \min( t, 1) , \label{eq:est3}\\
\sigma^2\big\Vert \partial_\tau u^c_\pm \big\Vert_{L^\infty([0,T]\times[0,T^\sharp];H^{s})} &\leq C(M^\sharp_{s+6}) \max(\epsilon\alpha_1,\epsilon^2,\mu)^2 \min( t, 1) .\label{eq:est4}
\end{align}
\end{Lemma}
\begin{proof} The function $u^c_\pm$ are defined as the solutions of a transport equation, with a known and controlled forcing term, thus the existence and uniqueness of $u^c_\pm$ and $U^c$ is straightforward.

As for the estimates, we focus below on $u^c_+$. Estimates for $u^c_-$ follow in the same way.

 By definition, $u^c_+$ satisfies
\[\sigma (\partial_t+\partial_x) u^c_+ \ = \ f^l(\t v_+,0) \ -\ f^l(\t v_+,\t v_-), \]
and the right-hand-side can be decomposed in the following way:
\begin{align*} f^l( v_+,0) \ -\ f^l( v_+, v_-)\ &= \ \sum_{j=1}^3 a_j\epsilon^{j} \partial_x\big(v_-^{j+1}\big) + \sum_{j=1}^3\epsilon^j \sum_{i=1}^{j} a_{i,j} \partial_x\big( v_+^{i}v_-^{j-i+1})+ \mu b \partial_x^3v_- \\
&\qquad +\mu\epsilon \partial_x(c_1v_-\partial_x^2v_-+c_2(\partial_x v_-)^2)+\mu\epsilon \sum_{k=0}^3 d_{k} ( \partial_x^k v_+)(\partial_x^{3-k} v_-) \\
&\equiv \sum_{j=1}^3 \epsilon^{j} \partial_xf_{j} \ +\ \sum_{j=1}^3\epsilon^j \sum_{i=1}^{j} g_{i,j} \ +\ \mu \partial_x f_4\ +\ \mu\epsilon \partial_x f_5\ +\ \mu\epsilon \sum_{k=0}^3 g_k,
\end{align*}
where the parameters $a_j$, $a_{i,j}$, $b_j$, $c_j$ and $d_k$ depend on $\gamma,\delta,\lambda,\theta$, and where we used $\partial_tv_-=\partial_x v_-$. Note that the first order terms $\epsilon a_1$ and $\epsilon a_{1,1}$ are factored by $\alpha_1=\frac32\frac{\delta^2-\gamma}{\gamma+\delta}$.

The contribution from $f_j$, $j=1,\dots,5$ are estimated thanks to Lemma~\ref{lem:easy}, whereas we use Lemma~\ref{lem:Lannes} for the contribution of $g_{i,j}$ and $g_k$.
\medskip

For all $0\leq j\leq 4$, one has $\partial_t f_j - \partial_x f_j =0$, and
\begin{itemize}
\item $\big\lvert f_j \big\rvert_{H^{s+1}} \ \leq \ C\big\lvert v_- \big\rvert_{L^\infty}\big\lvert v_- \big\rvert_{H^{s+1}} \ \leq \ C(\big\lvert \t v_-(\tau,\cdot) \big\rvert_{H^{s+1}})$, for $j=1,2,3$;
\item $\big\lvert f_4 \big\rvert_{H^{s+1}} \ \leq \ C \big\lvert \partial_x^2 v_- \big\rvert_{H^{s+1}} \ \leq \ C(\big\lvert \t v_-(\tau,\cdot) \big\rvert_{H^{s+3}})$;
\item $\big\lvert f_5 \big\rvert_{H^{s+1}} \ \leq \ C\big\lvert v_- \big\rvert_{L^\infty}\big\lvert \partial_x^2v_- \big\rvert_{H^{s+1}}+\big\lvert (\partial_x v_-)^2 \big\rvert_{H^{s+1}} \ \leq \ C(\big\lvert \t v_-(\tau,\cdot) \big\rvert_{H^{s+3}})$.
\end{itemize}
\medskip

It is then straightforward to check that $g_{ij}$ and $ g_k$ are the sum of bilinear mappings, applied to functions $v_\pm$, satisfying $\partial_t v_\pm \pm \partial_x v_\pm =0$, with $v_\pm$ bounded as below:
\begin{itemize}
\item $\big\lvert v_\pm^{i} \big\rvert_{H^{s}} \ \leq \ C(\big\lvert v_\pm \big\rvert_{H^{s}}) \ \leq \ C(\big\lvert \t v_\pm(\tau,\cdot) \big\rvert_{H^{s}})$, for $i=1,2,3$.
\item $\big\lvert \partial_x v_\pm^{j} \big\rvert_{H^{s}} \ \leq \ C(\big\lvert v_\pm \big\rvert_{H^{s+1}}) \ \leq \ C(\big\lvert \t v_\pm(\tau,\cdot) \big\rvert_{H^{s+1}})$, for $j=1,2,3$.
\item $\big\lvert \partial_x^k v_\pm \big\rvert_{H^{s}} \ \leq \ \mu C( \big\lvert v_\pm \big\rvert_{H^{s+k}}) \ \leq \ C(\big\lvert \t v_\pm(\tau,\cdot) \big\rvert_{H^{s+3}})$, for $k=0,1,2,3$.
\end{itemize}

It follows form Lemmata~\ref{lem:easy} and~\ref{lem:Lannes} that $u^c_+$ satisfies
\[\sigma \big\lvert u^c_+(\tau,t,\cdot) \big\rvert_{H^{s}} \ \leq \ \min(t,\sqrt t ) \max(\alpha_1\epsilon,\epsilon^2,\mu,\mu\epsilon) C\left(\big\lvert \t v_\pm(\tau,\cdot) \big\rvert_{H^{s+3}}\right),\]
and the last term is uniformly estimated through~\eqref{eq:estv0}, so that~\eqref{eq:est1} follows.
\medskip

As for the case of localized initial data, we make use of the fact that for $k=0,1,2,3$, one has
\[ \big\vert (1+x^2)^\alpha\partial_x^k \t v_\pm(\tau,t,\cdot) \big\vert_{H^s} \ = \ \big\vert (1+x^2)^\alpha\partial_x^k \t v_\pm(\tau,0,\cdot) \big\vert_{H^s} \ = \ \big\vert (1+x^2)^\alpha\partial_x^k v_\pm^\lambda(\tau/\sigma,\cdot) \big\vert_{H^s} \ \leq \ M^\sharp_{s+3},\]
so that~\eqref{eq:est3} follows in the same way for $\tau\leq \sigma T^\sharp$.
\medskip

The estimates on $\partial_\tau v_\pm(\tau,x)$ are obtained similarly, as
 \begin{align*} 
 (\partial_t+\partial_x) \partial_\tau u^c_+ \ \equiv \sum_{j=1}^3 \epsilon^{j} \partial_x \partial_\tau f_{j} \ +\ \sum_{j=1}^3\epsilon^j \sum_{i=1}^{j} \partial_\tau g_{i,j} \ +\ \mu \partial_x \partial_\tau f_4\ +\ \mu\epsilon \partial_x \partial_\tau f_5\ +\ \mu\epsilon \sum_{k=0}^3 \partial_\tau g_k.
\end{align*}
One can check
that each term of the right hand side satisfies the hypotheses of Lemmata~\ref{lem:easy} or~\ref{lem:Lannes}, and estimate~\eqref{eq:estv1} allows to obtain~\eqref{eq:est2}.
\medskip

The case of weighted Sobolev spaces~\eqref{eq:est4} follows as above, using estimate~\eqref{eq:estv1s}. The Lemma is proved.
\end{proof}

\paragraph{Completion of the proof.}
The consistency result stated in Proposition~\ref{prop:decomposition} is now a straightforward consequence of Lemma~\ref{lem:proofcons}, together with estimates~\eqref{eq:estv0}--\eqref{eq:estv0s}, and Lemma~\ref{lem:estimatescoupling}.

One can check that the remainder in Lemma~\ref{lem:proofcons} can be estimated as
\[ \big\vert \R \big\vert_{H^s} \ \leq \ C(M^\star_{s+6})\Big(\ \max(\alpha_1^2\epsilon^2,\epsilon^4,\mu^2) \min( t, \sqrt t) \ + \ \max(\epsilon^4,\mu^2) \ \Big)\]

Note that we use
\[ \partial_t\big(v_\pm(\sigma t,t,x) \big)\ = \ \sigma\partial_\tau \t v_\pm+\partial_t v_\pm , \quad \text{ and } \sigma\big\vert\partial_\tau \t v_\pm\big\vert_{H^{s}} \ \leq \ \max(\epsilon\alpha_1,\epsilon^2,\mu\epsilon) \big\vert v_\pm\big\vert_{H^{s+2}},\]
as above, as well as a uniform estimate
\[ \sigma \big\Vert \ u^c_\pm \ \big\Vert_{L^\infty([0,T^\star;H^{s+1})} \ \leq \ M. \]
The latter estimate can be enforced using Lemma~\ref{lem:estimatescoupling}, by restricting the time interval with
\[ (T^\star)^{1/2} \ \leq \ \frac{M}{\max(\epsilon\alpha_1,\epsilon^2,\mu) },\]
although such condition is not necessary, as Proposition~\ref{prop:decomposition} is empty for $t\geq M(\max(\epsilon\alpha_1,\epsilon^2,\mu) )^{-2}$, the accuracy of the decoupled solutions being of order $\O(1)$.
\medskip

More precisely, we proved the following Lemma.
\begin{Lemma}\label{lem:fullconsistency} For $\zeta^0,v^0\in H^{s+6}$, with $s\geq s_0>3/2$, and $(\epsilon,\mu,\delta,\gamma)=\p\in\A$, as defined in~\ref{eqn:defRegimeFNL},
let $U_{\text{app}}$ be defined as in Lemma~\ref{lem:proofcons}. Then $U_{\text{app}}(\sigma t,t,x)$ satisfies the coupled equations~\eqref{eqn:GreenNaghdiMeanI}, up to a remainder, $\R$, estimated for $t\in [0,T^\star)$ by
\[ \big\Vert \R \big\Vert_{L^\infty([0,t);H^s)} \ \leq \ C(M^\star_{s+6})\Big(\ \max(\alpha_1^2\epsilon^2,\epsilon^4,\mu^2) \min( t, \sqrt t) \ + \ \max(\epsilon^4,\mu^2) \ \Big),\]
and on $[0,T^\sharp)$ by 
\[ \big\Vert \R \big\Vert_{L^\infty([0,t);H^s)} \ \leq \ C(M^\sharp_{s+6})\ \Big(\ \max(\alpha_1^2\epsilon^2,\epsilon^4,\mu^2) \min( t, 1) \ + \ \max(\epsilon^4,\mu^2) \ \Big) .\]
\end{Lemma}
Proposition~\ref{prop:decomposition} is a consequence of the above Lemma, together with the estimates of Lemma~\ref{lem:estimatescoupling} (see also Remark~\ref{rem:visCL}).

\subsection{Discussion} \label{ssec:discussion}
Our result takes its full meaning in the light of a conjectured stability result on Green-Naghdi equations~\eqref{eqn:GreenNaghdiMeanI}, {\em or any consistent model}, as described in Hypothesis~\ref{conj:stab}. If such a result holds, then one can deduce from the consistency result of Proposition~\ref{prop:decompositionI} that the difference between the solution of the Green-Naghdi system~\eqref{eqn:GreenNaghdiMeanI} (and in the same way, the solution of the full Euler system if it exists) and the weakly coupled approximate solution $U\equiv U_{\text{CL}}+U^c$ is small. The estimates concerning the difference between the solution of~\eqref{eqn:GreenNaghdiMeanI} and the fully decoupled solution $U_{\text{CL}}$ simply follows from the estimates on $ U_{c}$, in Proposition~\ref{prop:decompositionI}. This strategy has been used in the water-wave case in~\cite{BonaColinLannes05}, and in the case of internal waves (when restricting to the long wave regime) in~\cite{Duchene11a}. 
 
Throughout this section, we assume that Hypothesis~\ref{conj:stab} holds, and study the convergence results between solutions of the Green-Naghdi model and the different approximate solutions which proceeds. 
\medskip

Let us first state the convergence results concerning the weakly coupled model, defined in Proposition~\ref{prop:decompositionI}.
\begin{Corollary}[Convergence of weakly coupled model]\label{cor:convergenceWC}
For $(\epsilon,\mu,\delta,\gamma)=\p\in\P$, as defined in~\eqref{eqn:defRegimeFNL}, let $U_{\text{GN}}^\p$ be a solution of Green-Naghdi equations~\eqref{eqn:GreenNaghdiMeanI} such that the family $(U_{\text{GN}}^\p)$ is uniformly bounded on $H^s$, $s$ sufficiently big, over time interval $[0,T_{\text{GN}}]$. 
Assume that hypotheses of Proposition~\ref{prop:decompositionI} hold, and denote $U^\p\equiv U_{\text{CL}}^\p+U^c[U_{\text{CL}}^\p]$ the approximate solution, satisfying ${U_{\text{GN}}}^\p\id{t=0}=U^\p\id{t=0}$. Then if Hypothesis~\ref{conj:stab} is valid, for any $t\leq \min(T_{\text{GN}},T^\star_{s+6})$, one has
\begin{align*} \big\Vert U^\p -U_{\text{GN}}^\p\big\Vert_{L^\infty([0,t];H^s)} \ &\leq \ C\ \max(\epsilon^2(\delta^2-\gamma)^2,\epsilon^4,\mu^2)\ t (1+\sqrt{t}),
\end{align*}
with $C=C\big(\big\Vert U_{\text{GN}}^\p\big\Vert_{L^\infty([0,T_{\text{GN}}];H^s)} ,M^\star_{s+6},\frac{1}{\delta_{\text{min}}},\delta_{\text{max}},\epsilon_{\text{max}},\mu_{\text{max}},|\lambda|,|\theta| \big)$.

If moreover, the initial data is sufficiently localized in space, then one has for $t\leq \min(T_{\text{GN}},T^\sharp_{s+6})$,
\begin{align*}\big\Vert U^\p -U_{\text{GN}}^\p\big\Vert_{L^\infty([0,t];H^s)} \ &\leq \ C\ \max(\epsilon^2(\delta^2-\gamma)^2,\epsilon^4,\mu^2)\ t,
\end{align*}
with $C=C\big(\big\Vert U_{\text{GN}}^\p\big\Vert_{L^\infty([0,T];H^s)} ,M^\sharp_{s+6},\frac{1}{\delta_{\text{min}}},\delta_{\text{max}},\epsilon_{\text{max}},\mu_{\text{max}},|\lambda|,|\theta| \big)$.
\end{Corollary} 
We therefore see that the weakly coupled model achieves the same accuracy as the fully coupled Green-Naghdi model, in the Camassa-Holm regime~\eqref{eqn:defRegimeCH}, if one restricts to critical case $\delta^2-\gamma=\O(\mu)$, and if the initial data is sufficiently localized in space.
\medskip

Let us now turn to the fully decoupled models. The following result is a straightforward application of the above Corollary, together with the estimate of Proposition~\ref{prop:decompositionI}. Estimates concerning lower order decoupled models are obtained in the same way, using Proposition~\ref{prop:other}.

\begin{Corollary}[Convergence of decoupled models]\label{cor:convergence}For $(\epsilon,\mu,\delta,\gamma)=\p\in\P$, as defined in~\eqref{eqn:defRegimeFNL}, let $U_{\text{GN}}^\p$ be a solution of Green-Naghdi equations~\eqref{eqn:GreenNaghdiMeanI} such that the family $(U_{\text{GN}}^\p)$ is uniformly bounded on $H^s$, $s$ sufficiently big, over time interval $[0,T_{\text{GN}}]$. Denote respectively $U_{\text{CL}}^\p,U_{\text{eKdV}}^\p,U_{\text{KdV}}^\p,U_{\text{iB}}^\p$ the decoupled approximations defined in Definitions~\ref{def:CL} and~\ref{def:other}. Assume that the hypotheses of Proposition~\ref{prop:decompositionI} hold, and Hypothesis~\ref{conj:stab} is valid. Define $\eps_0=\max(\epsilon(\delta^2-\gamma),\epsilon^2,\mu)$. Then for any $t\leq \min(T_{\text{GN}},T^\star_{s+6})$, one has
\begin{align*} \big\Vert U_{\text{CL}}^\p-{U_{\text{GN}}}^\p\big\Vert_{L^\infty([0,t];H^s)} \ &\leq \ C\ \eps_0\ \min(t,t^{1/2}) (1 \ + \ \eps_0 t ) ,\\
\big\Vert U_{\text{eKdV}}^\p-{U_{\text{GN}}}^\p\big\Vert_{L^\infty([0,t];H^s)} \ &\leq \ C\ \eps_0\ \min(t,t^{1/2}) (1 \ + \ \eps_0 t ) \ + \ C\ \max(\epsilon^3,\mu\epsilon)\ t ,\\
\big\Vert U_{\text{KdV}}^\p-{U_{\text{GN}}}^\p\big\Vert_{L^\infty([0,t];H^s)} \ &\leq \ C\ \eps_0\ \min(t,t^{1/2}) (1 \ + \ \eps_0 t ) \ + \ C\ \epsilon^2 \ t ,\\
\big\Vert U_{\text{iB}}^\p-{U_{\text{GN}}}^\p\big\Vert_{L^\infty([0,t];H^s)} \ &\leq \ C\ \eps_0\ \min(t,t^{1/2}) (1 \ + \ \eps_0 t ) \ + \ C\ \max(\epsilon^2,\mu)\ t,
\end{align*}
with $C=C\big(\big\Vert U_{\text{GN}}^\p\big\Vert_{L^\infty([0,T_{\text{GN}}];H^s)} ,M^\star_{s+6},\frac{1}{\delta_{\text{min}}},\delta_{\text{max}},\epsilon_{\text{max}},\mu_{\text{max}},|\lambda|,|\theta| \big)$.

If the initial data is sufficiently localized in space, then one has for $t\leq \min(T,T^\sharp)$,
\begin{align*} \big\Vert U_{\text{CL}}-U_S\big\Vert_{L^\infty([0,t];H^s)} \ &\leq \ C\ \eps_0\ \min(t,1) (1 \ + \ \eps_0 t ) ,\\
\big\Vert U_{\text{eKdV}}-U_S\big\Vert_{L^\infty([0,t];H^s)} \ &\leq \ C\ \eps_0\ \min(t,1) (1 \ + \ \eps_0 t ) \ + \ C\ \max(\epsilon^3,\mu\epsilon)\ t ,\\
\big\Vert U_{\text{KdV}}-U_S\big\Vert_{L^\infty([0,t];H^s)} \ &\leq \ C\ \eps_0\ \min(t,1) (1 \ + \ \eps_0 t ) \ + \ C\ \epsilon^2 \ t ,\\
\big\Vert U_{\text{iB}}-U_S\big\Vert_{L^\infty([0,t];H^s)} \ &\leq \ C\ \eps_0\ \min(t,1) (1 \ + \ \eps_0 t ) \ + \ C\ \max(\epsilon^2,\mu)\ t,
\end{align*}
with $C=C\big(\big\Vert U_{\text{GN}}^\p\big\Vert_{L^\infty([0,T_{\text{GN}}];H^s)} ,M^\sharp_{s+6},\frac{1}{\delta_{\text{min}}},\delta_{\text{max}},\epsilon_{\text{max}},\mu_{\text{max}},|\lambda|,|\theta| \big)$.
\end{Corollary} 

\begin{figure}[htb]
\subfigure[non-localized initial data]{
\includegraphics[width=0.46\textwidth]{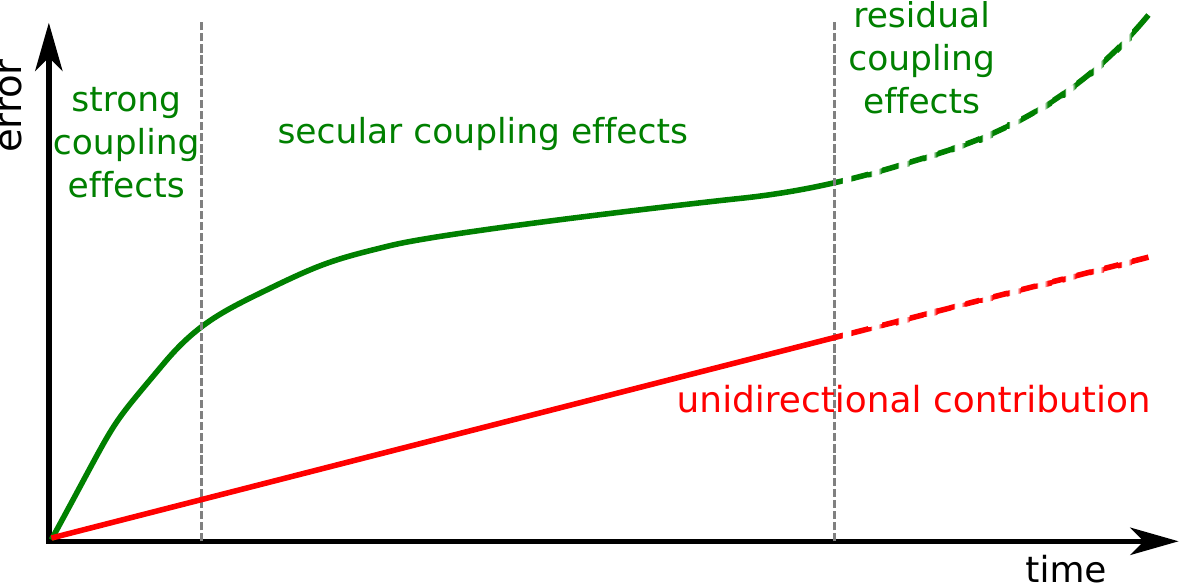}
}
\subfigure[localized initial data]{
\includegraphics[width=0.46\textwidth]{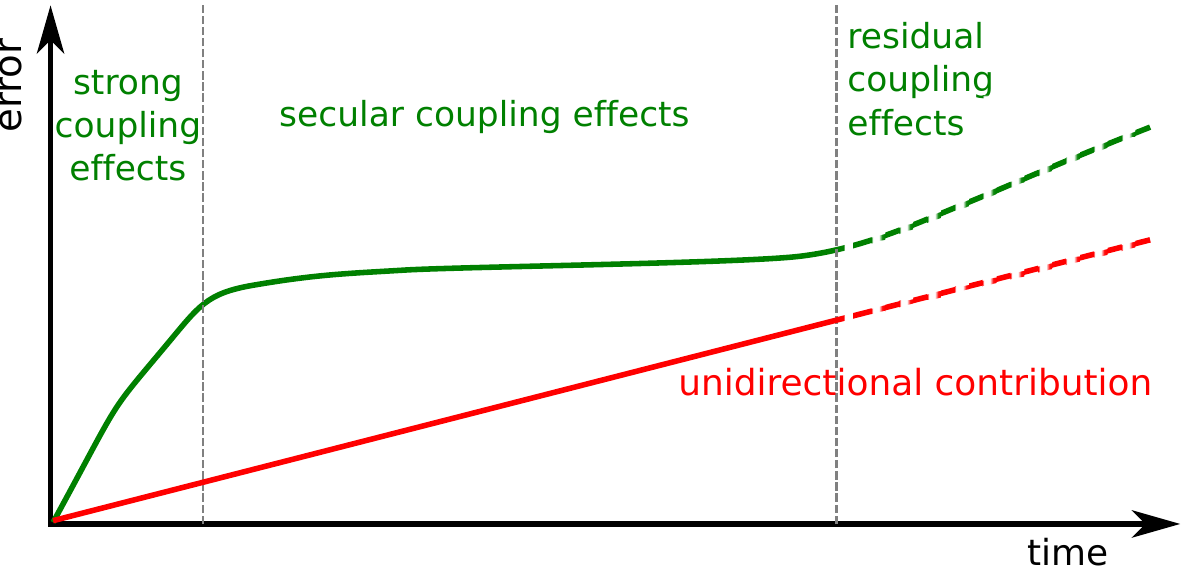}
}
\caption{Sketch of the error}
\label{fig:SketchError}
\end{figure}
 Corollary~\ref{cor:convergence} exhibits different sources of error with different time scales.
\begin{itemize}
\item The first (and often main) source of error comes from the coupling between the two counterpropagating waves which are neglected by our decoupled models, but recovered at first order by the weakly coupled model (see Corollary~\ref{cor:convergenceWC}). Following~\cite{Wright05}, we name the contribution of this term the {\em counterpropagation error}. 
\subitem--- This error grows linearly in time, for times of order $\O(1)$, as the two waves are located at the same position; and coupling effects are {\em strong}. 
\subitem--- However, one is able to control sublinearly in time the {\em secular growth} of such coupling terms, which is the key ingredient of our result, and yields a contribution of size $\O(\eps_0\sqrt t)$ in general, $\O(\eps_0)$ if the initial data is sufficiently localized. 
\item After very long time, one sees the effect of the precision of the consistency result, through the stability hypothesis. Such a contribution is unavoidable, as it appears after many manipulations of the equations, such as the use of BBM trick, or near-identity change of variables. We call these contributions residual errors. The error generated in that way affects both the solution of the scalar evolution equation and the coupling term, thus there are two contributions. 
\subitem--- The so-called {\em unidirectional error} is linear in time as long as the solution of the decoupled model is uniformly bounded, and varies greatly with the choice of the model.
\subitem--- The {\em residual error} coming from the coupling term may be superlinear (if the initial data is not localized in space), as the coupling term grows in time.
\end{itemize}
All of these contributions are summarized in Figure~\ref{fig:SketchError}. The total error of a decoupled models is the sum of the green and the red curves. The green curve represents the error generated by neglecting the coupling between two waves, and therefore is affected by three of the contributions listed above, at different time scales. The red curve is the unidirectional contribution, and varies with the choice of the model.

However, let us precise that the full pattern is not always visible, as in many cases, one source of error will be negligible in front of the others. In particular, keep in mind that we obtained existence and uniform control in weighted Sobolev spaces of the solutions of decoupled models only over times of order $\O(1/\max(\epsilon,\mu))$ (Proposition~\ref{prop:WP}). Consequently, the time domain for which residual contributions appear is out of the scope of our rigorous results.

In order to understand the relevance of a specific evolution equation, one has to look whether the unidirectional contribution, which depends on the evolution equation itself, is smaller or greater than the different sources of error due to coupling. For example, unidirectional error is always negligible in front of coupling terms in the Constantin-Lannes model, whereas it is expected to be preponderant in the inviscid Burgers' model.
\medskip

Let us look precisely at several interesting scenari, in the following subsections. Our discussion is supported by numerical simulations. Each time, we compute the coupled Green-Naghdi model and various decoupled approximations, for different values of $\epsilon$ and corresponding $\mu$ (depending on the situation, $\mu=\epsilon$ or $\mu=\sqrt\epsilon$), and over times $\O(\epsilon^{-3/2})$. As pointed out above, such time scale is out of the scope of our rigorous result. However, interesting behavior appears after times $\O(1/\epsilon)$, as we shall see. Each of the figures below contains three panels. On the left-hand side, we represent the difference between the Green-Naghdi model and decoupled models, with respect to time and for $\epsilon=0.1,0.05,0.035$. Values at times $1/\epsilon$ and $\epsilon^{-3/2}$ are marked. In the two right-hand-side panels, we plot the difference in a log-log scale for several values of $\epsilon$ (the markers reveal the positions which have been computed), at aforementioned times. The pink triangles express the convergence rate.

The initial data is fixed such that the left-going wave is initially two-third the right going wave, in order to avoid symmetry cancellations: ${v_-}\id{t=0}=2/3{v_+}\id{t=0}$.
For localized initial data, we choose initial data ${v_+}\id{t=0}=\exp(-(x/2)^2)$ and for non-localized initial data, we choose ${v_+}\id{t=0}=(1+10x^2)^{-1/3}$.
We set $\delta^2=\gamma=0.64$ in the critical ratio setting, and $\delta=0.5,\gamma=0.9$ in the non-critical ratio setting. Finally, we set $\theta=1/2$ and $\lambda=0$. The numerical scheme we use has been briefly described at the end of Section~\ref{sec:unidirectional}

\subsubsection{The long wave regime.}\label{sssec:KdV} 

Let us assume that we are in the long wave regime: $\epsilon=\O(\mu)$, and therefore $\eps_0\approx\mu$ in Corollary~\ref{cor:convergence}.

It follows that {\em the decoupled KdV approximation has the same order of accuracy as any higher order model}, whatever the initial data or critical ratio is. Indeed, one has the same estimates for the decoupled approximations $U=U_{\text{KdV}}$ or $U=U_{\text{eKdV}}$ or $U=U_{\text{CL}}$:
\[ \big\Vert U_{\text{KdV}}-U_{\text{GN}}\big\Vert_{L^\infty([0,t];H^s)} \ \approx \ \big\Vert U_{\text{CL}}-U_{\text{GN}}\big\Vert_{L^\infty([0,t];H^s)} \ \leq \ C_0\ \mu\ \min(t,t^{1/2}) (1 \ + \ \mu t ),\]
for any $t\in [0,T^\star_{s+6})$, and if the initial data is localized,
\[\big\Vert U_{\text{KdV}}-U_{\text{GN}}\big\Vert_{L^\infty([0,t];H^s)} \ \approx \ \big\Vert U_{\text{CL}}-U_{\text{GN}}\big\Vert_{L^\infty([0,t];H^s)} \ \leq \ C_0\ \mu\ \min(t,1) (1 \ + \ \mu t ),\]
for any $t\in [0,T^\sharp_{s+6})$.
Let us note that one recovers the results of~\cite{Duchene11a}, for the KdV approximation.
\begin{figure}[!hbt] 
 \subfigure[behavior of the error with respect to time]{
\includegraphics[width=0.5\textwidth]{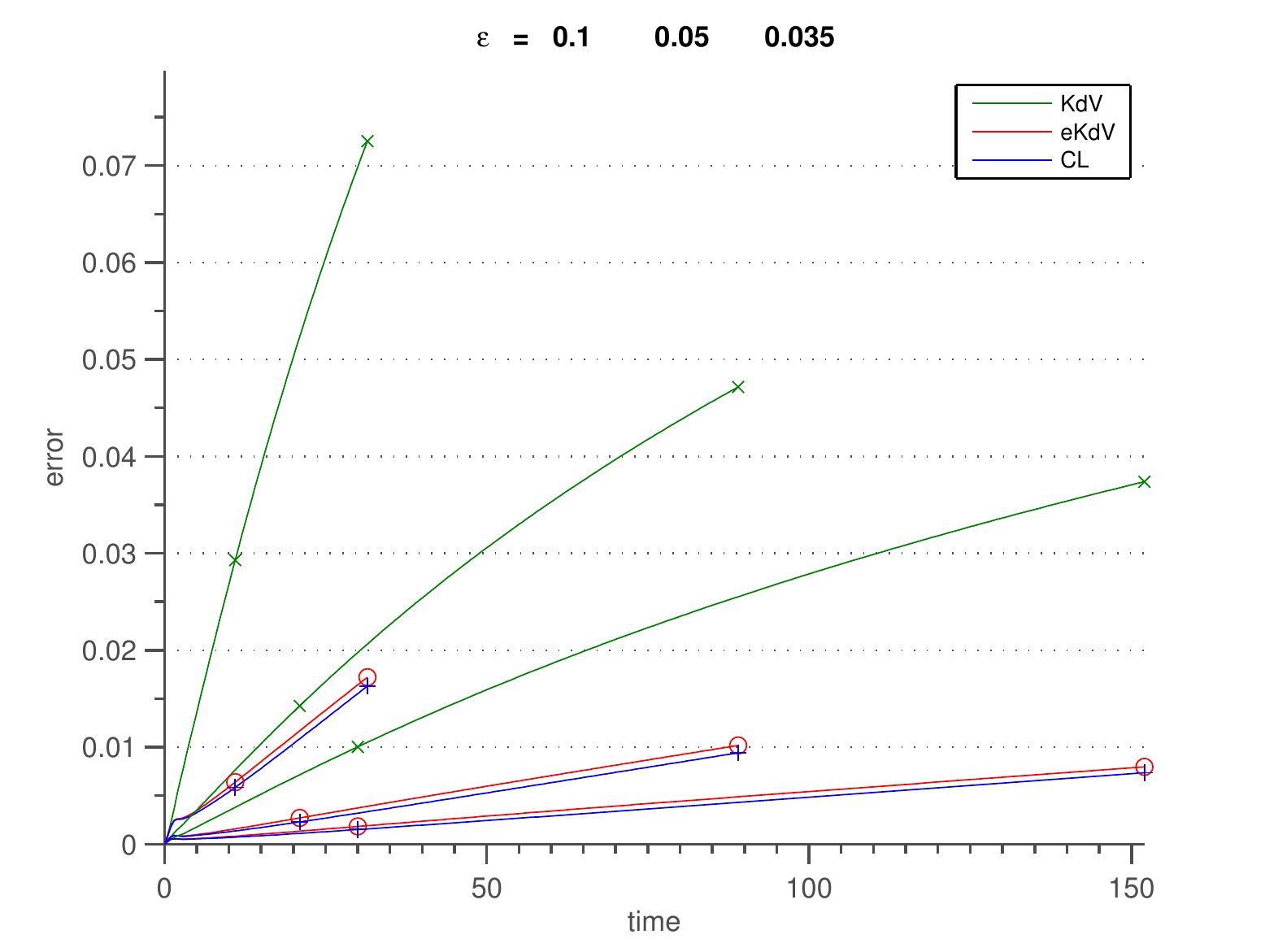}
\label{fig:21KdVa}
}\hspace{-1cm}
 \subfigure[behavior of the error with respect to $\epsilon=\mu$]{
\includegraphics[width=0.5\textwidth]{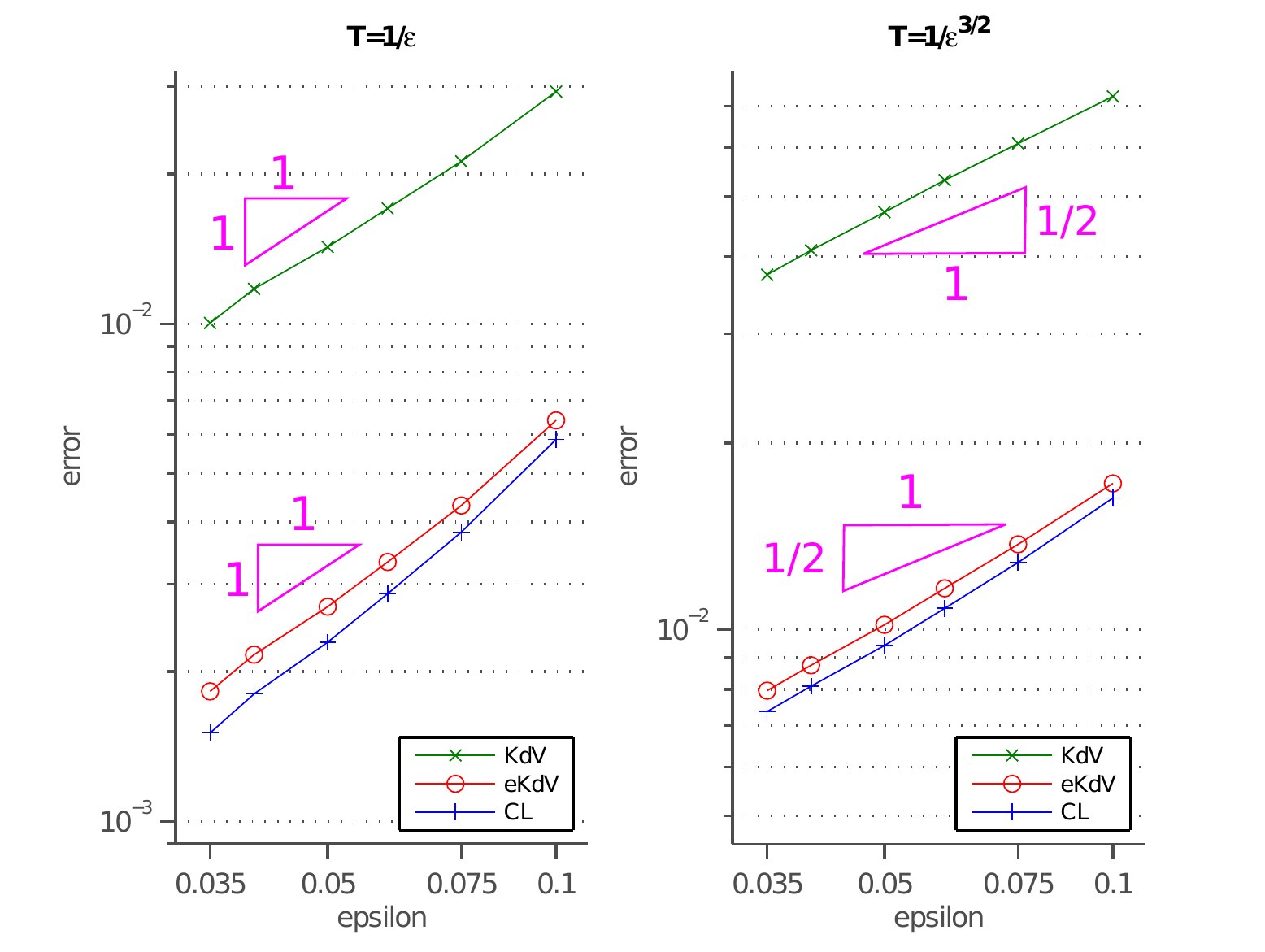}
\label{fig:21KdVb}
}
\caption{Long wave regime, critical ratio, localized initial data}
\label{fig:21KdV}
\end{figure}

More precisely, a unidirectional error of size $\O(\mu^2 t)$ is not detectable, as it is smaller than the coupling terms, presented above, so that the additional error produced by neglecting higher order terms in the Constantin-Lannes approximation, and considering only the KdV equation, does not change the accuracy of the model. However, let us note that for localized initial data, unidirectional error of size $\O(\mu^2 t)$ is the limiting case. This is why one sees in Figure~\ref{fig:21KdVa} a noticeable difference between the KdV approximation and higher order approximations (eKdV, CL). The rate of convergence in figure~\ref{fig:21KdVb}, however, is identical: the error is of size $\O(\mu)$ at time $T=\O(1/\mu)$ and $\O(\mu^{1/2})$ at time $T=\O(1/\mu^{3/2})$. We remark that the last panel of Figure~\ref{fig:21KdV} shows a slight discrepancy with respect to the predicted convergence rate in $\O(\mu^{1/2})$. This may be due to the fact that higher order sources of error (typically of size $\O(\mu^3 t)$) are detectable for larger values of $\mu$ ($\mu=0.1$), and eventually becomes negligible for smaller values ($\mu=0.035$), therefore artificially improving the convergence rate.
Let us note also that the criticality of the depth ratio do not play a role in this analysis, and that simulations in the case of non-critical ratio gives similar outcome.

\subsubsection{The Camassa-Holm regime, with non-critical ratio.} \label{sssec:CH1}
We are now in the case where $|\delta^2-\gamma|\geq \alpha_0>0$, and $\epsilon\approx\sqrt\mu$, thus $\eps_0\approx\epsilon\approx\sqrt\mu$ in Corollary~\ref{cor:convergence}.

In that case, the contribution of the coupling error are always greater than the one of the unidirectional error, and one has the same estimates as for the decoupled approximations $U=U_{\text{iB}}$, $U=U_{\text{KdV}}$, $U=U_{\text{eKdV}}$ or $U=U_{\text{CL}}$:
\[\big\Vert U_{\text{iB}}-U_{\text{GN}}\big\Vert_{L^\infty([0,t];H^s)} \ \approx \ \big\Vert U_{\text{CL}}-U_{\text{GN}}\big\Vert_{L^\infty([0,t];H^s)} \ \leq \ C_0\ \sqrt\mu\ \min(t,t^{1/2}) (1 \ + \ \sqrt\mu t ),\]
for any $t\in [0,T^\star_{s+6})$, and if the initial data is localized,
\[ \big\Vert U_{\text{iB}}-U_{\text{GN}}\big\Vert_{L^\infty([0,t];H^s)} \ \approx \ \big\Vert U_{\text{CL}}-U_{\text{GN}}\big\Vert_{L^\infty([0,t];H^s)} \ \leq \ C_0\ \sqrt\mu\ \min(t,1) (1 \ + \ \sqrt\mu t ),\]
for any $t\in [0,T^\sharp_{s+6})$.
As a consequence, {\em the inviscid Burgers' approximation is as precise as any higher order decoupled model.}
\begin{figure}[!hbt] 
 \subfigure[behavior of the error with respect to time]{
\includegraphics[width=0.5\textwidth]{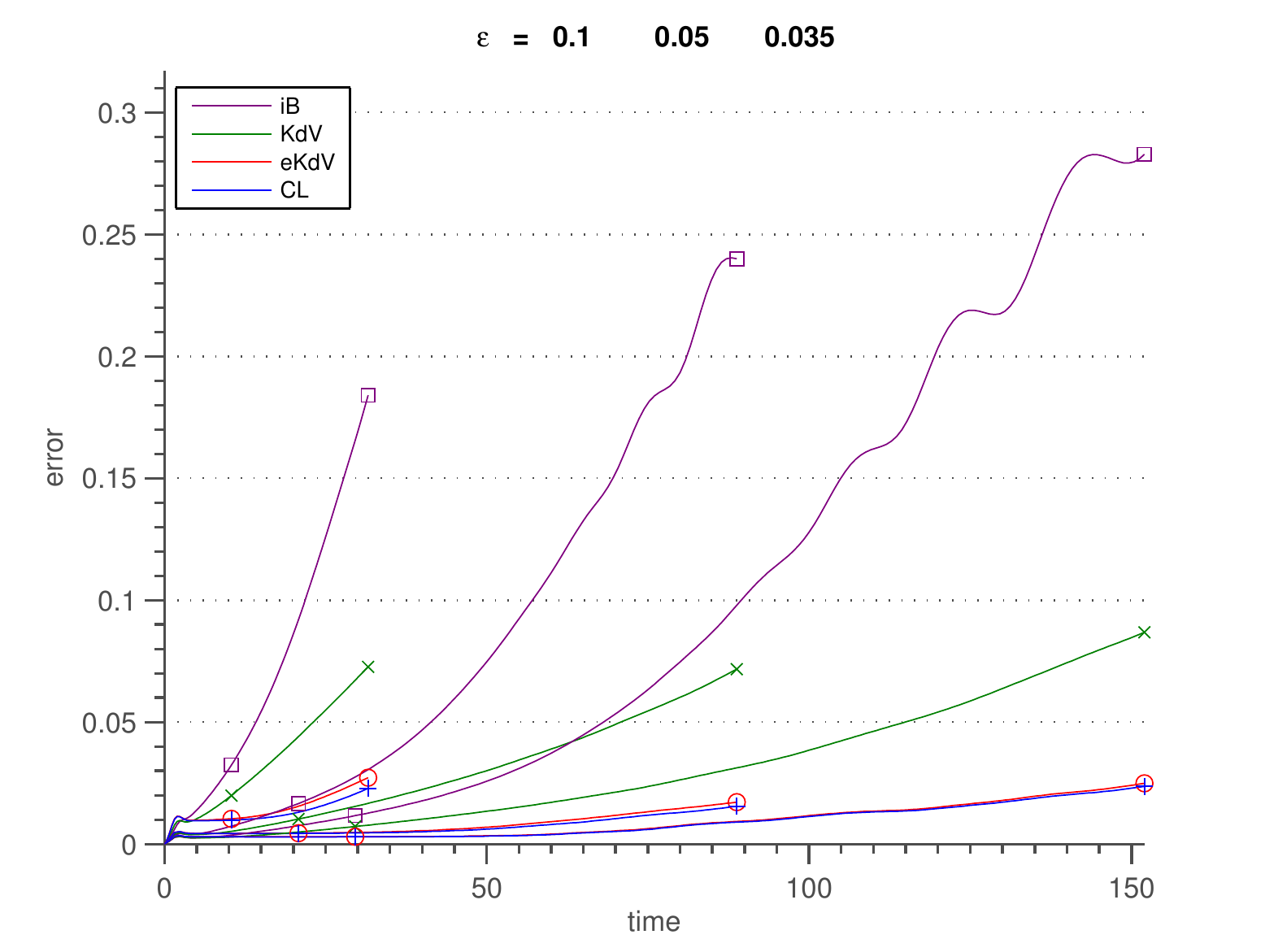}
\label{fig:11CHa}
}\hspace{-1cm}
 \subfigure[behavior of the error with respect to $\epsilon=\sqrt\mu$]{
\includegraphics[width=0.5\textwidth]{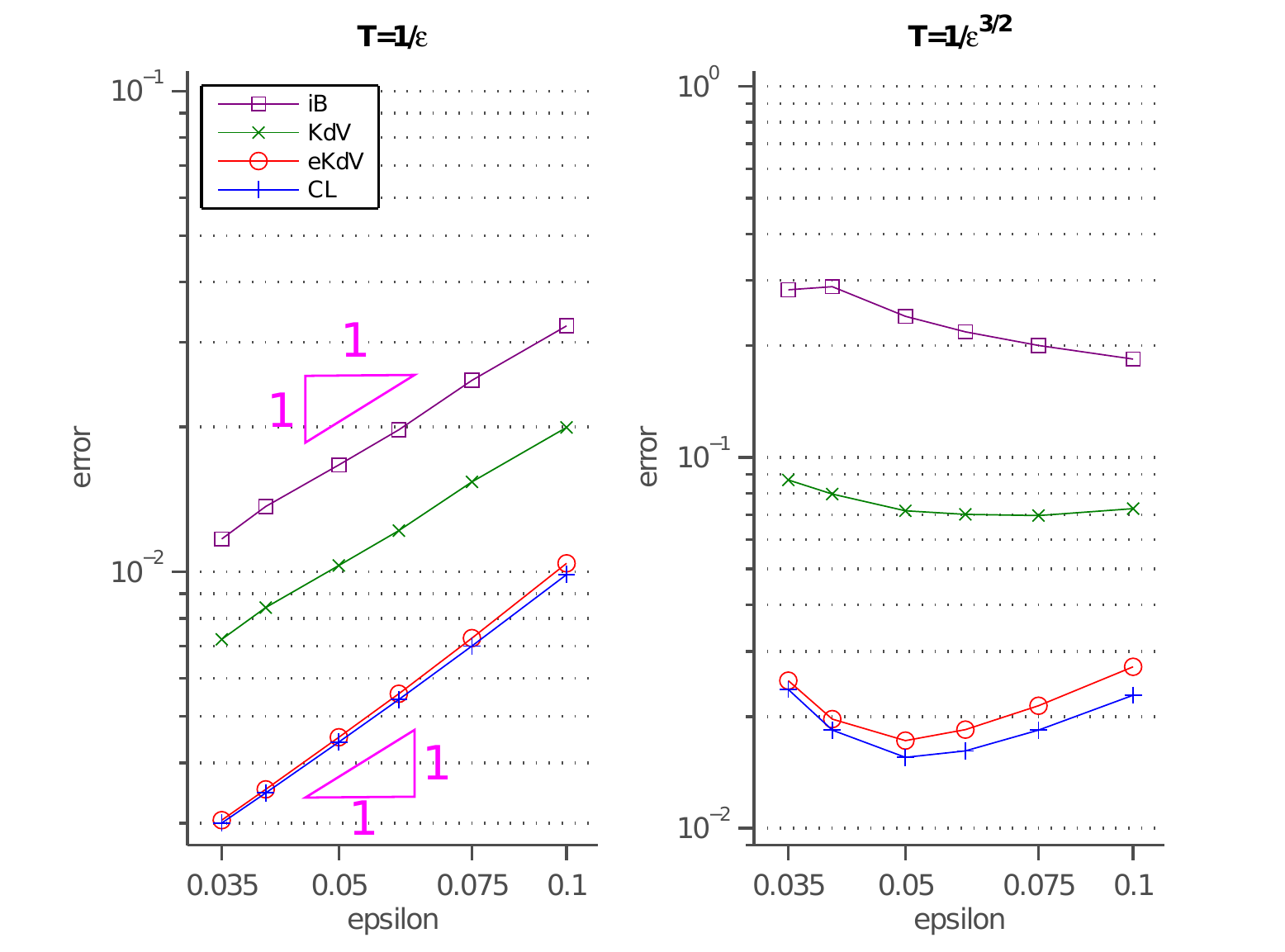}
\label{fig:11CHb}
}
\caption{Camassa-Holm regime, non-critical ratio, localized initial data}
\label{fig:11CH}
\end{figure}

 Here, the unidirectional error of the inviscid Burgers' approximation is of the same order of magnitude as the residual coupling error, if the initial data is sufficiently localized. Therefore, in figure~\ref{fig:11CH}, one can observe a noticeable difference between the iB approximation and higher order models for long times, although the convergence rate is similar at time $T=1/\epsilon=1/\sqrt\mu$. For longer times, the decoupled approximate solutions do not seem to converge, which may indicate that the exact solution $U_{\text{GN}}$ cannot be controlled over times $\O(\epsilon^{-3/2})$, in the setting we use.

\subsubsection{The Camassa-Holm regime, critical case.} \label{sssec:CH2}
Now let us assume that $\delta^2-\gamma \ = \ \O(\sqrt\mu)$, and $\epsilon\ \approx \ \sqrt\mu$, so that $\eps_0\approx\mu$ in Corollary~\ref{cor:convergence}.

 We recall that $T^\star$ and $T^\sharp$ are known to exist and to be of size $T=\O(1/\max(\epsilon,\mu))=\O(\mu^{-1/2})$, as a result of Proposition~\ref{prop:WP}. Over such time, the unidirectional error in the eKdV model is smaller than the coupling errors: one has for both $U=U_{\text{eKdV}}$ and $U=U_{\text{CL}}$,
\[ \big\Vert U_{\text{eKdV}}-U_{\text{GN}}\big\Vert_{L^\infty([0,t];H^s)} \ \approx \ \big\Vert U_{\text{CL}}-U_{\text{GN}}\big\Vert_{L^\infty([0,t];H^s)} \ \leq \ C_0\ \mu\ \min(t,t^{1/2}) ,\]
for any $t\in [0,T^\star_{s+6})$ if $T^\star=\O(\mu^{-1/2})$. If moreover, the initial data is localized, then
\[ \big\Vert U_{\text{eKdV}}-U_{\text{GN}}\big\Vert_{L^\infty([0,t];H^s)} \ \approx \ \big\Vert U_{\text{CL}}-U_{\text{GN}}\big\Vert_{L^\infty([0,t];H^s)} \ \leq \ C_0\ \mu\ \min(t,1),\]
for any $t\in [0,T^\sharp_{s+6})$ if $T^\sharp=\O(\mu^{-1/2})$. The accuracy of both eKdV and CL approximations are of size $\O(\mu^{3/4})$ at time $T=\O(\mu^{-1/2})=\O(\epsilon^{-1})$, or improved to size $\O(\mu)$ if the initial data is sufficiently localized. The accuracy of the KdV and iB approximation is of size $\O(\mu^{1/2})$ at the latter time.

However, if one looks at longer times, then the picture is different. In our simulations, we looked at times up to $T=\O(\epsilon^{-3/2})=\O(\mu^{-3/4})$. At that time, the localization in space of the initial data plays an important role. If the initial data is non-localized in space, then contribution of the coupling terms are predicted to be greater than the unidirectional error of the eKdV approximation: the secular coupling error is dominant up to $T=\O(\mu^{-1})$ ($\mu\sqrt t\geq \mu^{3/2}t$), after what the residual coupling term is dominant ($\mu^2 t^{3/2}\geq \mu^{3/2} t$ for $t\geq \mu^{-1}$). Thus {\em the eKdV approximation is predicted to be as precise as the CL approximation in the Camassa-Holm regime and if the initial data is non-localized in space, even if the depth ratio is critical}.

 On the contrary, if the initial data is localized, then the unidirectional error is dominant: $\mu^{3/2}t\geq \mu(1+\mu t)$ for any $t\geq\mu^{-1/2}$. This leads, at time $T=\O(\mu^{-3/4})$, to an error of size $\O(\mu)$ for the CL approximation, and of size $\O(\mu^{3/4})$ for the eKdV approximation. {\em If the initial data is localized, then the CL approximation is substantially more accurate than the eKdV approximation after very long time, in the Camassa-Holm regime with critical depth ratio.}
 
 The difference between localized and non-localized initial data can be clearly seen when comparing Figure~\ref{fig:21CH} and Figure~\ref{fig:23CH}. The fact the convergence rate of our numerical simulations fit the predictions of Corollary~\ref{prop:WP}, even for times which are out of the scope of our rigorous results, gives hope to extend the validity of the CL approximation to such times, where it becomes more precise than lower order approximations.

\begin{figure}[!hbt] 
 \subfigure[behavior of the error with respect to time]{
\includegraphics[width=0.5\textwidth]{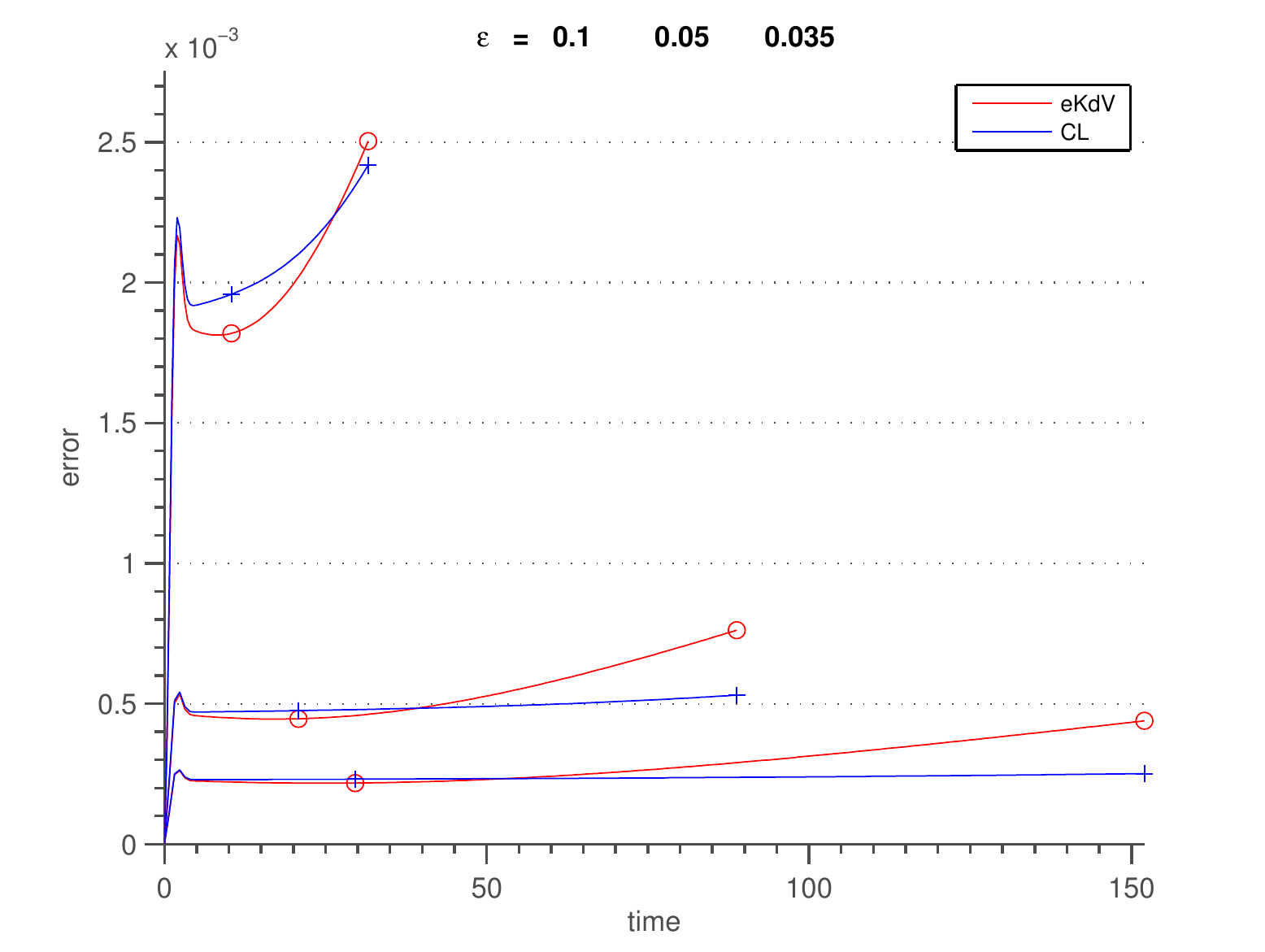}
\label{fig:21CHa}
}\hspace{-1cm}
 \subfigure[behavior of the error with respect to $\epsilon=\sqrt\mu$]{
\includegraphics[width=0.5\textwidth]{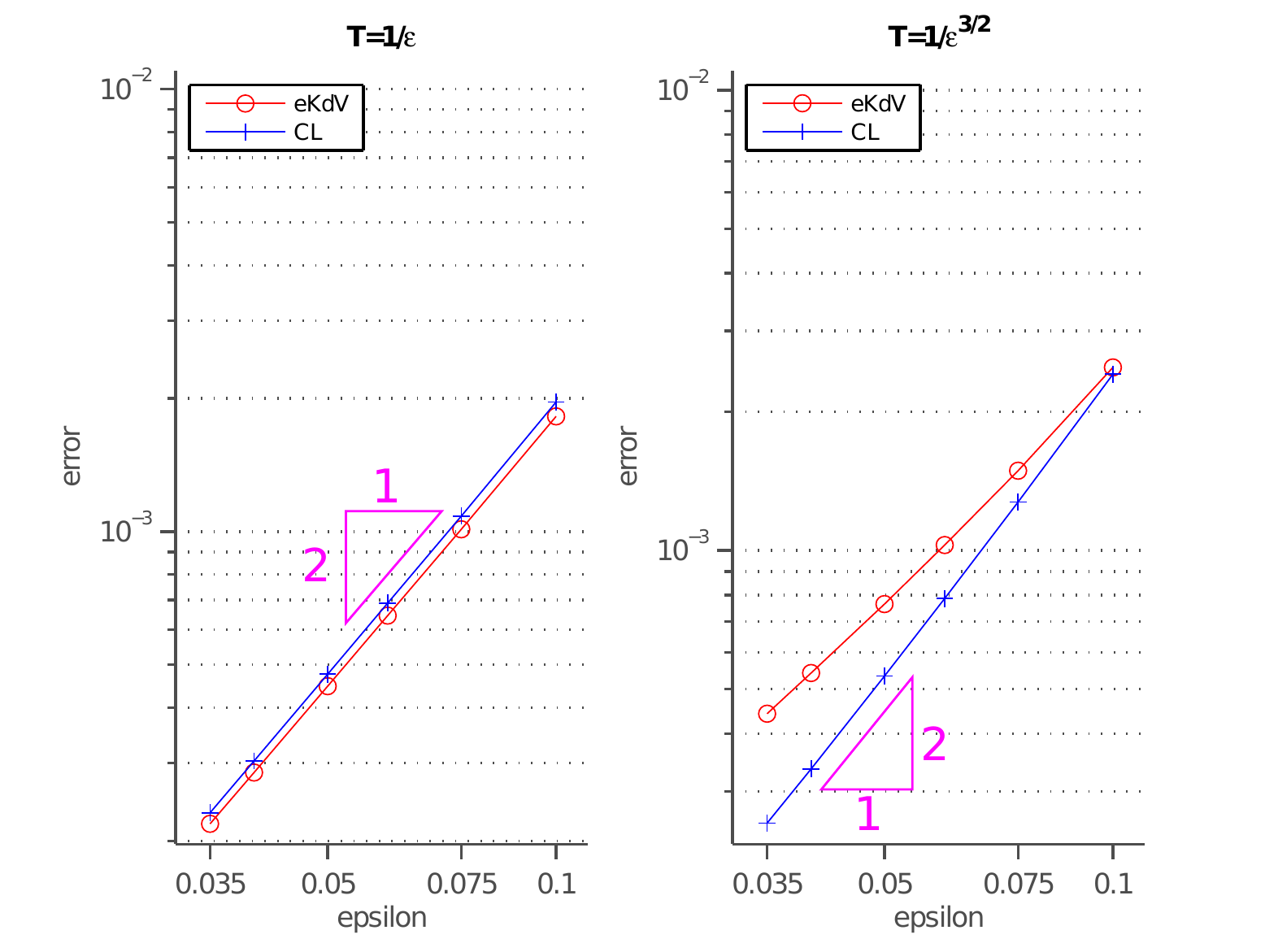}
\label{fig:21CHb}
}
\caption{Camassa-Holm regime, critical ratio, localized initial data}
\label{fig:21CH}
\end{figure}
\begin{figure}[!hbt] 
 \subfigure[behavior of the error with respect to time]{
\includegraphics[width=0.5\textwidth]{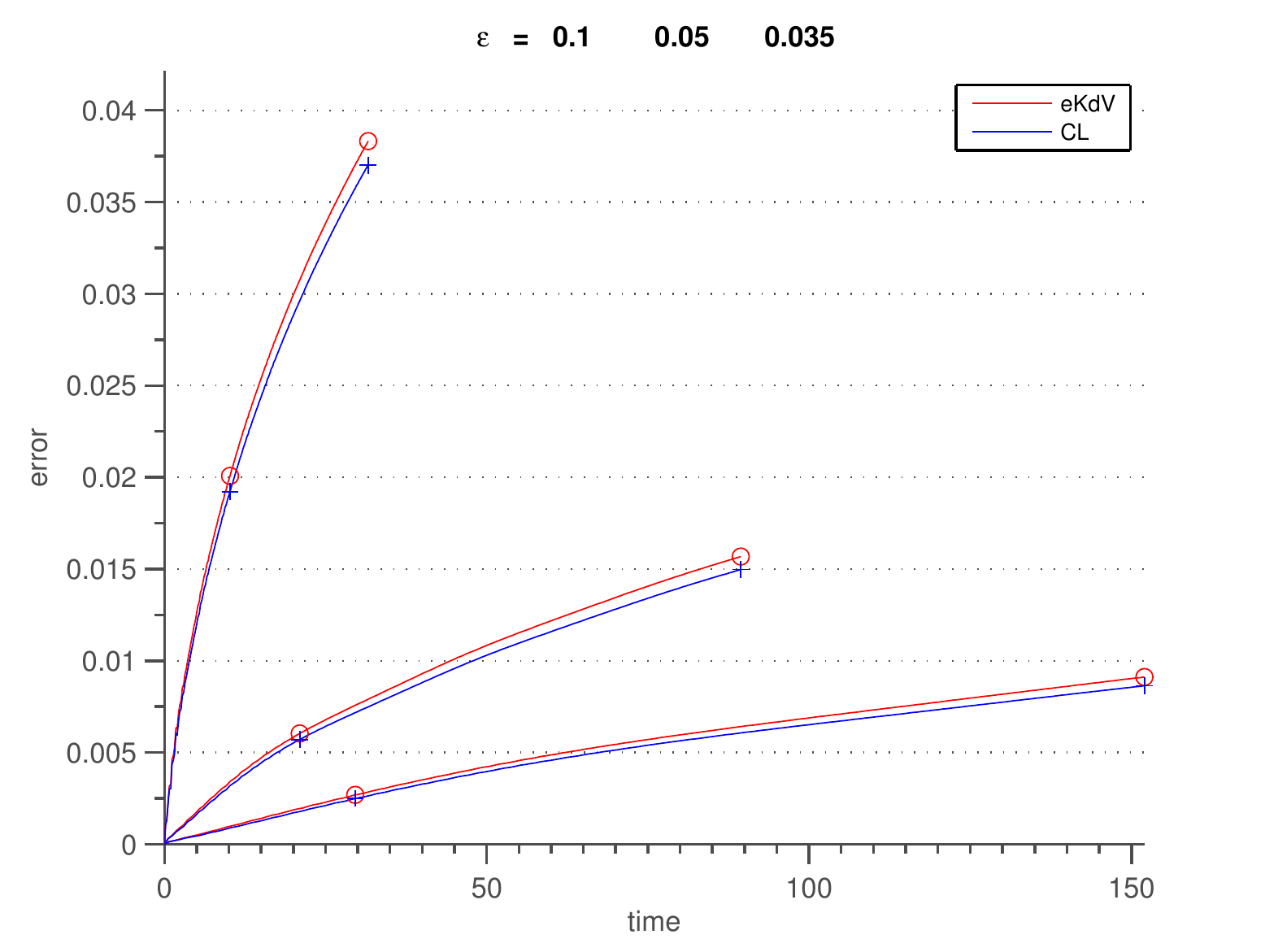}
\label{fig:23CHa}
}\hspace{-1cm}
 \subfigure[behavior of the error with respect to $\epsilon=\sqrt\mu$]{
\includegraphics[width=0.5\textwidth]{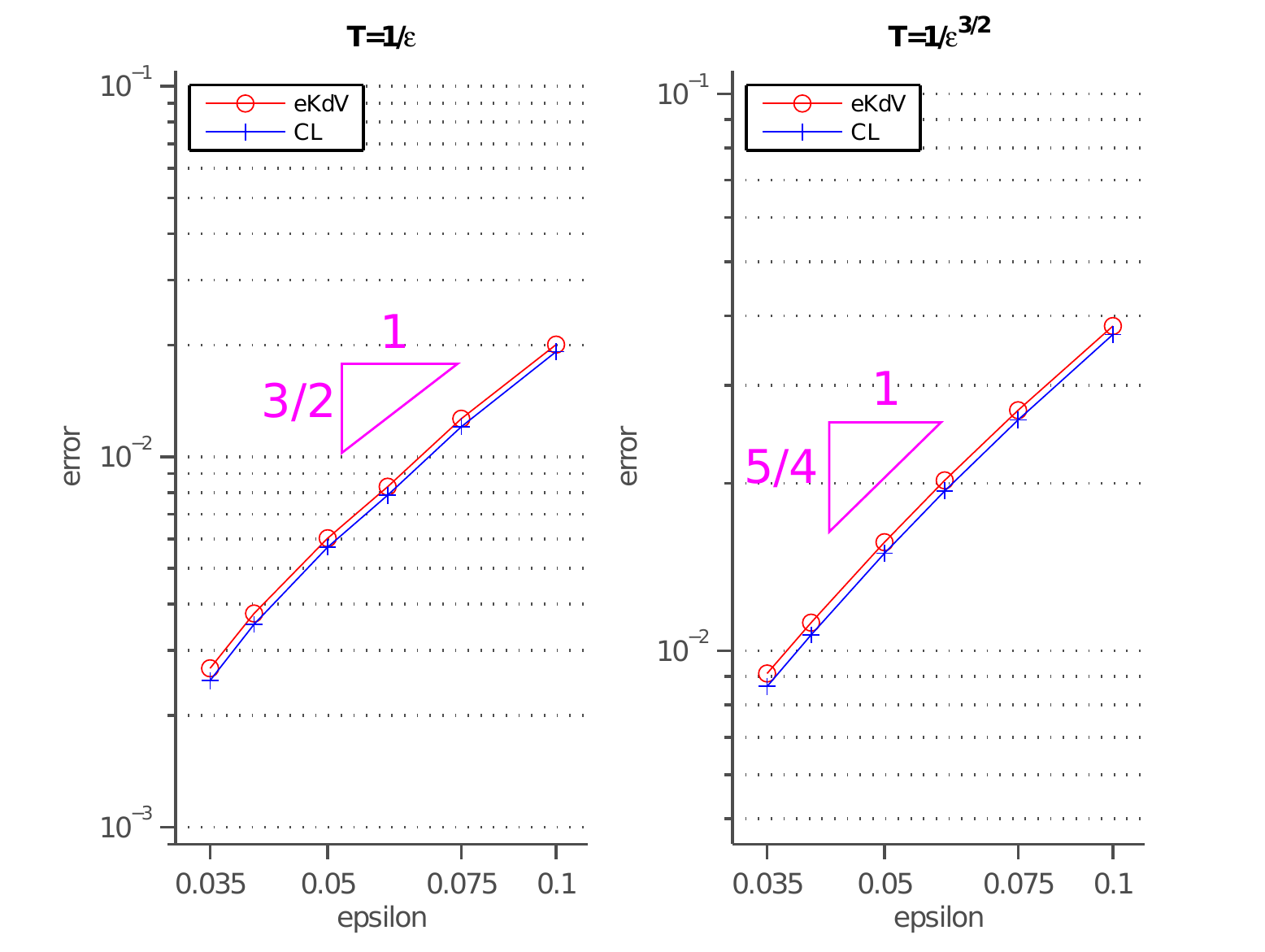}
\label{fig:23CHb}
}
\caption{Camassa-Holm regime, critical ratio, non-localized initial data}
\label{fig:23CH}
\end{figure}

\end{document}